\documentclass{amsart}
\usepackage{amscd,amsmath,xypic,amssymb,combelow,tikz-cd,etoolbox,calligra,mathrsfs,enumitem,mathtools}

\DeclareMathOperator{\sHom}{\mathscr{H}\text{\kern -3pt {\calligra\large om}}\,}
\DeclareMathOperator{\sRHom}{\mathscr{RH}\text{\kern -3pt {\calligra\large om}}\,}
\DeclareMathOperator{\sQuot}{\mathscr{Q}\text{\kern -3pt {\calligra\large uot}}\,}

\makeatletter
\patchcmd{\@settitle}{\uppercasenonmath\@title}{}{}{}
\makeatother

\xyoption{all}
\CompileMatrices

\newcommand{\nc}{\newcommand}

\makeatletter
\@addtoreset{equation}{section}
\makeatother

\newtheorem{theorem}[subsection]{Theorem}
\newtheorem{proposition}[subsection]{Proposition}
\newtheorem{lemma}[subsection]{Lemma}
\newtheorem{corollary}[subsection]{Corollary}
\newtheorem{conjecture}[subsection]{Conjecture}
\newtheorem{definition}[subsection]{Definition}

\newtheorem{claim}[subsection]{Claim}

\newtheorem{example}[subsection]{Example}
\newtheorem{remark}[subsection]{Remark}

\nc{\fa}{{\mathfrak{a}}}
\nc{\fb}{{\mathfrak{b}}}
\nc{\fg}{{\mathfrak{g}}}
\nc{\fh}{{\mathfrak{h}}}
\nc{\fj}{{\mathfrak{j}}}
\nc{\fn}{{\mathfrak{n}}}
\nc{\fm}{{\mathfrak{m}}}
\nc{\fu}{{\mathfrak{u}}}
\nc{\fp}{{\mathfrak{p}}}
\nc{\fr}{{\mathfrak{r}}}
\nc{\ft}{{\mathfrak{t}}}
\nc{\fsl}{{\mathfrak{sl}}}
\nc{\fgl}{{\mathfrak{gl}}}
\nc{\hsl}{{\widehat{\mathfrak{sl}}}}
\nc{\hgl}{{\widehat{\mathfrak{gl}}}}
\nc{\hg}{{\widehat{\mathfrak{g}}}}
\nc{\chg}{{\widehat{\mathfrak{g}}}{}^\vee}
\nc{\hn}{{\widehat{\mathfrak{n}}}}
\nc{\chn}{{\widehat{\mathfrak{n}}}{}^\vee}

\nc{\Mod}{{\textrm{Mod}}}

\nc{\wGL}{{\widehat{GL}^+}}

\nc{\BA}{{\mathbb{A}}}
\nc{\BC}{{\mathbb{C}}}
\nc{\BG}{{\mathbb{G}}}
\nc{\BM}{{\mathbb{M}}}
\nc{\BK}{{\mathbb{K}}}
\nc{\BN}{{\mathbb{N}}}
\nc{\BF}{{\mathbb{F}}}
\nc{\BH}{{\mathbb{H}}}
\nc{\BP}{{\mathbb{P}}}
\nc{\BQ}{{\mathbb{Q}}}
\nc{\BR}{{\mathbb{R}}}
\nc{\BZ}{{\mathbb{Z}}}

\nc{\ff}{{\mathbb{F}}}
\nc{\kk}{{\mathbb{K}}}
\nc{\kko}{{\mathbb{K}}}

\nc{\coh}{{\text{Coh}}}

\nc{\CA}{{\mathcal{A}}}
\nc{\CC}{{\mathcal{C}}}
\nc{\CB}{{\mathcal{B}}}
\nc{\DD}{{\mathcal{D}}}
\nc{\CE}{{\mathcal{E}}}
\nc{\CF}{{\mathcal{F}}}
\nc{\tCF}{{\widetilde{\CF}}}
\nc{\tCM}{{\widetilde{\CM}}}
\nc{\tCT}{{\widetilde{\CT}}}
\nc{\oCF}{{\bar{\CF}}}
\nc{\CG}{{\mathcal{G}}}
\nc{\CL}{{\mathcal{L}}}
\nc{\CK}{{\mathcal{K}}}
\nc{\CI}{{\mathcal{I}}}
\nc{\CM}{{\mathcal{M}}}
\nc{\CH}{{\mathcal{H}}}
\nc{\CN}{{\mathcal{N}}}
\nc{\CO}{{\mathcal{O}}}
\nc{\CP}{{\mathcal{P}}}
\nc{\CR}{{\mathcal{R}}}
\nc{\CQ}{{\mathcal{Q}}}
\nc{\CS}{{\mathcal{S}}}
\nc{\CT}{{\mathcal{T}}}
\nc{\tCU}{{\widetilde{\CU}}}
\nc{\CU}{{\mathcal{U}}}
\nc{\CV}{{\mathcal{V}}}
\nc{\CW}{{\mathcal{W}}}

\nc{\tpsi}{{\widetilde{\Psi}}}
\nc{\wpi}{{\widetilde{\pi}}}
\nc{\Ker}{{\text{Ker }}}

\nc{\CX}{{\mathcal{X}}}
\nc{\tCX}{{\widetilde{\mathcal{X}}}}
\nc{\CY}{{\mathcal{Y}}}
\nc{\tCY}{{\widetilde{\mathcal{Y}}}}

\nc{\tN}{{\widetilde{\CN}}}
\nc{\pN}{{\BP\widetilde{\CN}}}

\nc{\tT}{{T}}

\nc{\fC}{{\mathfrak{C}}}
\nc{\fZ}{{\mathfrak{Z}}}
\nc{\fU}{{\mathfrak{U}}}
\nc{\fV}{{\mathfrak{V}}}
\nc{\fW}{{\mathfrak{W}}}
\nc{\fS}{{\mathfrak{S}}}

\nc{\bfZ}{{\bar{\fZ}}}

\nc{\od}{{\overline{d}}}
\nc{\rg}{{\textrm{R}\Gamma}}
\nc{\erg}{{\emph{R}\Gamma}}
\nc{\id}{{\textrm{Id}}}
\nc{\rhom}{{\textrm{RHom}}}

\def\ph{\varphi}

\def\Ext{\textrm{Ext}}
\def\Hom{\textrm{Hom}}

\def\e{\varepsilon}

\def\td{\tilde{d}}

\def\and{\textrm{ }\&\textrm{ }}

\def\sym{\textrm{Sym}}
\def\esym{\emph{\sym}}

\def\tCF{\widetilde{\CF}}

\def\res{\textrm{Res }}

\def\wCA{\widehat{\CA}}

\def\loccit{\emph{loc. cit. }}
\def\loccitt{\emph{loc. cit.}}
\def\op{{\textrm{op}}}

\def\tab{\text{} \\}

\def\Tan{\text{Tan}}

\def\sgn{\text{sign }}

\def\km{K_\CM}

\def\kms{K_{\CM \times S}}
\def\kmss{K_{\CM \times S \times S}}
\def\kmms{K_{\CM' \times S}}

\def\ks{K_S}
\def\kss{K_{S \times S}}

\def\ext{\text{ext}}
\def\eext{\emph{ext}}

\def\ebig{\emph{big}}

\def\fr{\text{Frac }\BF}

\def\bp{\bar{p}}
\def\bpi{\bar{\pi}}

\def\tCL{\widetilde{\CL}}

\def\qis{\stackrel{\text{q.i.s.}}\cong}

\def\loc{\text{loc}}

\def\oi{{\int_{0 - \infty}}}
\def\i{{\int_{\infty - 0}}}

\def\ring{{\BK}}
\def\field{{\BF}}

\def\conv{{\text{conv}}}
\def\pss{p_{S*}^1 \times p_{S*}^2}

\def\cur{{\curvearrowright}}
\def\ccur{{\underline{\curvearrowright}}}

\begin{document}

\title[$W$--algebras associated to surfaces]{\large{\textbf{$W$--ALGEBRAS ASSOCIATED TO SURFACES}}}
\author[Andrei Negu\cb t]{Andrei Negu\cb t}
\address{MIT, Department of Mathematics, Cambridge, MA, USA}
\address{Simion Stoilow Institute of Mathematics, Bucharest, Romania}
\email{andrei.negut@gmail.com}

\maketitle

\renewcommand{\thefootnote}{\fnsymbol{footnote}} 
\renewcommand{\thefootnote}{\arabic{footnote}} 


\begin{abstract} We define an integral form of the deformed $W$--algebra of type $\fgl_r$, and construct its action on the $K$--theory groups of moduli spaces of rank $r$ stable sheaves on a smooth projective surface $S$, under certain assumptions. Our construction generalizes the action studied by Nakajima, Grojnowski and Baranovsky in cohomology, although the appearance of deformed $W$--algebras by generators and relations is a new feature. Physically, this action encodes the AGT correspondence for 5d supersymmetric gauge theory on $S \times$circle.

\end{abstract}

\section{Introduction} 
\label{sec:introduction}

\medskip

\noindent The purpose of the present paper is to study the representation theory of the $K$--theory groups of moduli spaces of sheaves on a smooth surface $S$ over an algebraically closed field of characteristic 0 (henceforth denoted by $\BC$). For any $r>0$, we study a certain $\BZ[q_1^{\pm 1}, q_2^{\pm 1}]$-algebra $\CA_r$ and construct a module for the specialization of $\CA_r$ when $q_1$ and $q_2$ are set equal to the Chern roots of $\Omega_S^1$. When $S = \BA^2$, this module is the $\BC^* \times \BC^*$ equivariant $K$--theory of the moduli space of rank $r$ framed sheaves on $\BP^2$. When $S$ is projective, the module is defined as follows. Fix $c_1 \in H^2(S,\BZ)$ and an ample divisor $H \subset S$ such that $\gcd(r, c_1 \cdot H) = 1$. Let:
\begin{equation}
\label{eqn:def k}
K_\CM = \bigoplus_{c_2 = \left \lceil \frac {r-1}{2r} c_1^2 \right \rceil}^\infty K_0(\CM_{(r,c_1,c_2)})
\end{equation}
where $\CM_{(r,c_1,c_2)}$ denotes the moduli space of stable sheaves on $S$ of rank $r$ and Chern classes $(c_1,c_2)$ (see \cite{HL} for a review of the theory). The fact that stable sheaves have $c_2$ bounded below by $\frac {r-1}{2r} c_1^2$ is called the Bogomolov inequality. The coprimality of $r$ and $c_1 \cdot H$ is called \textbf{Assumption A} in \cite{Univ}, and it implies that: \\

\begin{itemize}

\item every semistable sheaf is stable, hence $\CM$ is projective \\

\item there exists a universal sheaf $\CU$ on $\CM \times S$ \\

\end{itemize}

\noindent where (once we have fixed $r$ and $c_1$) we will always write: 
$$
\CM = \bigsqcup_{c_2 = \left \lceil \frac {r-1}{2r} c_1^2 \right \rceil}^\infty \CM_{(r,c_1,c_2)}
$$
As we recall in Section \ref{sec:w}, $\CA_r$ is closely related to the deformed $W$--algebra of type $\fgl_r$ (\cite{AKOS, FF}). It is generated over $\BZ[q_1^{\pm 1}, q_2^{\pm 1}]$ by symbols $W_{n,k}$ indexed by a ``degree" $n \in \BZ$ and a ``level" $k \in \{1,...,r\}$. There exists an $r \rightarrow \infty$ limit of this construction, which we denote by $\CA_\infty$, with generators $W_{n,k}$ indexed by $n \in \BZ$, $k \in \BN$. We have:
\begin{equation}
\label{eqn:quotient algebras}
\CA_r = \frac {\CA_\infty}{(W_{n,k})_{n \in \BZ, k>r}}
\end{equation}
To approach the algebra $\CA_\infty$, one invokes another object of great importance in geometric representation theory: the double shuffle algebra $\CA$ (also known as the Ding-Iohara-Miki algebra, elliptic Hall algebra, spherical DAHA of type $A_\infty$) and many other names. For our purposes, one may define this algebra as:
\begin{equation}
\label{eqn:double shuffle intro}
\CA = \Big \langle E_{n,k} \Big \rangle_{(n,k) \in \BZ^2 \backslash (0,0)} \Big / \text{relation \eqref{eqn:e triangle double}}
\end{equation}
(see Sections \ref{sec:shuf} and \ref{sec:w} for details). Theorem \ref{thm:w algebra} implies that there is an injection:
$$
\CA_\infty \hookrightarrow \widehat{\CA}
$$
where $\widehat{\CA}$ is a certain completion of $\CA$. With this in mind, our construction for the action $\CA_r \curvearrowright \km$ can be summarized logically as follows:
\begin{equation}
\label{eqn:infty acts}
\begin{rcases}
\CA \curvearrowright \km \\
\CA_\infty \hookrightarrow \widehat{\CA}
\end{rcases} \quad \Rightarrow \quad \CA_\infty \curvearrowright \km
\end{equation}
\begin{equation}
\label{eqn:r acts}
\{W_{n,k}\}_{n\in \BZ}^{k > r} \text{ act by 0 in } \km \quad \Rightarrow \quad \CA_r \curvearrowright \km
\end{equation}

\medskip

\begin{conjecture}
\label{conj:shuf acts}
	
There exists an ``action" of the double shuffle algebra:
\begin{equation}
\label{eqn:shuf acts}
\CA \curvearrowright \km
\end{equation}
given explicitly in Subsection \ref{sub:explicit action}. \\
	
\end{conjecture} 

\noindent The fact that the action \eqref{eqn:shuf acts} extends to the completion $\wCA$ is a consequence of the fact that $\km$ is a good $\CA$-module (see Definition \ref{def:completion}), which in turn follows from Bogomolov's inequality. Explicitly, to each generator $W_{n,k} \in \CA$, we associate in Section \ref{sec:shuf acts} an abelian group homomorphism via certain explicit correspondences:
\begin{equation}
\label{eqn:endomorphism}
\km \xrightarrow{w_{n,k}} \kms
\end{equation}
satisfying the compatibility conditions spelled out in Definition \ref{def:shuf acts}: this is the meaning of the quotes around the word ``action" in the statement of Conjecture \ref{conj:shuf acts}. When $S = \BA^2$, the same construction was done in \cite{W} in the context of the equivariant $K$--theory of the moduli space of framed sheaves, but the arbitrary surface case poses interesting features: for instance, the parameters $q_1$ and $q_2$ of the algebra $\CA_r$ are identified with the Chern roots of the cotangent bundle $\Omega_S^1$ (in other words, $[\Omega_S^1] = q_1+q_2$ and $[\omega_S] = q_1q_2$ as elements of $\ks$, the algebraic $K$-theory ring of $S$). \\

\begin{theorem}
\label{thm:w acts}

Assuming Conjecture \ref{conj:shuf acts}, we have $w_{n,k} = 0$ in \eqref{eqn:endomorphism} for all $n \in \BZ$ and $k > r$, hence the ``action" \eqref{eqn:shuf acts} factors through an ``action" $\CA_r \curvearrowright \km$. \\

\end{theorem}

\noindent The group $K_{\CM}$ may be interpreted as the Hilbert space of 5d supersymmetric gauge theory on the projective surface $S$ times a circle. The fact that we construct an action of the deformed $W$--algebra on $K_{\CM}$ yields a mathematical generalization of the Alday-Gaiotto-Tachikawa (AGT) correspondence between gauge theory and conformal field theory (along the lines of \cite{MO,W,SV3}, which treated the $S=\BA^2$ case). Historically, this correspondence has usually been studied for toric surfaces using Nekrasov's equivariant generalization of partition functions. Therefore, the projective surface situation treated in the present paper is a new phenomenon, which to the author's knowledge has not previously been studied for rank $r>1$. In our follow-up paper \cite{AGT}, Theorem \ref{thm:w acts} will be used to compute the interaction between the action $\CA_r \curvearrowright \km$ and the Carlsson-Okounkov Ext operator (\cite{CO}), thus establishing the equality of partition functions that AGT stipulates for bifundamental matter. However, one should note that in order to completely set up the AGT correspondence for an arbitrary surface $S$, the deformed $W$--algebra studied herein would need to be enlarged. Such an enlargement should probably deform the vertex operator algebras (VOAs) associated to 4-manifolds (see \cite{DGP}). \\

\noindent The action $\CA \curvearrowright \km$ is defined via the correspondences $\fZ_1$ and $\fZ_2^\bullet$ (Definitions \ref{def:desc} and \ref{def:fine}, respectively), which yield resolutions of singularities of the cohomological constructions of Nakajima \cite{Nak} and Grojnowski \cite{G} in rank 1, and Baranovsky \cite{Ba} in rank $r$. In general, the correspondences $\fZ_1$ and $\fZ_2^\bullet$ are defined as dg schemes, but we show in Propositions \ref{prop:z1 smooth} and \ref{prop:z2 smooth} that they are smooth schemes whenever the moduli spaces $\CM$ are smooth (more precisely, under \textbf{Assumption S} of \eqref{eqn:assumption s}). The proof of the smoothness of $\fZ_1$ and $\fZ_2^\bullet$ relies on presenting their tangent spaces in terms of cones of Ext groups, which is an idea that has appeared repeatedly in the literature (see \cite{GT} for a development in the context of virtual degeneration loci). \\

\noindent Through an analysis of the relation between the double shuffle algebra and $W$--algebras, which we perform in Sections \ref{sec:shuf} and \ref{sec:w}, we show that Conjecture \ref{conj:shuf acts} boils down to Conjecture \ref{conj:comm}. In Subsection \ref{sub:proof ass b}, we prove both Conjectures under:\\

\noindent  \textbf{Assumption B:} The Kunneth decomposition $\kss \cong \ks \boxtimes \ks$ holds, and:
\begin{equation}
\label{eqn:diag decomp intro}
[\CO_\Delta] = \sum_c l_c \boxtimes l^c \in \ks \boxtimes \ks \cong \kss
\end{equation}
for $\{l_c,l^c\} \in \ks$. Moreover, we may decompose the class of the universal sheaf as:
\begin{equation}
\label{eqn:univ decomp intro}
[\CU] = \sum_c [\CT_c] \boxtimes l^c \in \km \boxtimes \ks \cong \kms
\end{equation}
Finally, we assume that $\{\CT_c\}$ which appear in \eqref{eqn:univ decomp} are actually locally free sheaves on $\CM$ whose exterior powers generate $K_{\CM_{(r,c_1,c_2)}}$ as a ring, for any $c_2 \in \BZ$. \\

\noindent Assumption B holds for $S = \BP^2$, as shown in \cite[Example 3.12]{Univ}. More generally, all statements of Assumption B hold for rational surfaces (see \cite{Markman}) except maybe the fact that $\CT_c$ thus obtained are locally free (although we believe that the results of Subsection \ref{sub:proof ass b} do not crucially require this fact, only our terminology does). \\

\noindent In general, the full statement of Conjecture \ref{conj:shuf acts} would follow if one knew that the the algebra $\CA$ is contained inside the $K$--theoretic Hall algebra in the sense of \cite{SV}, but this fact is open (see \cite{Min} for certain results when $S = T^*C$ is the cotangent bundle to a curve). When $S = \BA^2$, Theorem \ref{thm:w acts} was proved in \cite{W} using the fact that the group $\km$ is (generically) an irreducible module for the $W$--algebra. We do not have this feature for a general surface $S$, and we have little control on the size of the abelian group $\km$. Instead, we give a geometric proof of Theorem \ref{thm:w acts} in Section \ref{sec:w acts}, which admits an obvious lift to the level of functors between derived categories. Let us mention certain possible avenues for generalizing our results. \\

\begin{enumerate}

\item Do the results in the present paper hold without Assumption A, namely the fact that $\gcd(r,c_1 \cdot H)=1$? This would involve replacing universal sheaves by quasi-universal or twisted sheaves in all our definitions and computations, and we make no claims about the technical difficulties that may arise. \\

\item Do the results in the present paper hold for $r = 0$? What about when one replaces $S$ by a quasi-projective surface endowed with a torus action with projective fixed point set? Although the philosophy clearly generalizes in either case, care must be taken in choosing the moduli space $\CM$ correctly. \\

\item Do the results hold if the moduli space $\CM$ of stable sheaves is replaced by the moduli stack of all sheaves? It is conceivable that $\CA$ still acts on $K_\CM$, but whether the algebra $\CA_\infty$ still acts is unclear. The reason for this is that elements of $\CA_\infty$ are certain infinite sums inside $\CA$, which act correctly on \eqref{eqn:def k} due to the fact that the grading by $c_2$ is bounded from below. \\

\end{enumerate}

\noindent The plan of the present paper is the following: \\

\begin{itemize}[leftmargin=*]

\item In Section \ref{sec:shuf}, we study the shuffle algebra $\CS$ (half of the algebra $\CA$), and show that it is defined over the ring $\BZ[q_1^{\pm 1},q_2^{\pm 1}]$, as opposed from the field $\BQ(q_1,q_2)$ where it has been commonly studied. The main technical step is Theorem \ref{thm:comm}, which states that commutators of the generators $E_{n,k}$ have integral coefficients. \\

\item In Section \ref{sec:w}, we construct the double shuffle algebra $\CA$, its completion $\widehat{\CA}$, and show how to embed the deformed $W$-algebra $\CA_\infty$ in this completion. The main result here is Theorem \ref{thm:w algebra}, which shows that the quadratic relations among the generators $W_{n,k}$ of deformed $W$-algebras hold in the double shuffle algebra $\CA$. \\

\item In Section \ref{sec:mod}, we recall the moduli spaces of stable sheaves $\CM = \{\CF\}$ on the given smooth surface $S$, as well as the dg scheme $\fZ_1 = \{\CF \supset_x \CF'\}$ studied in \cite{Univ}, which parametrizes flags of two such sheaves (hereafter, $\CF \supset_x \CF'$ means that $\CF/\CF'$ is isomorphic to the skyscraper sheaf at the closed point $x$). We also define the dg scheme $\fZ_2^\bullet = \{\CF \supset_x \CF' \supset_x \CF''\}$, and show that under Assumption S of \eqref{eqn:assumption s}, it coincides with the smooth scheme on which it is supported. \\

\item In Section \ref{sec:shuf acts}, we make precise the meaning of ``action" $\CA \curvearrowright \km$ which features in Conjecture \ref{conj:shuf acts} (see Definition \ref{def:shuf acts}), then we postulate that this action should take the generators $E_{n,k} \in \CA$ to certain correspondences built out of the dg schemes $\fZ_1$ and $\fZ_2^\bullet$ (in Subsection \ref{sub:explicit action}), and then prove Conjecture \ref{conj:shuf acts} subject to Assumption B (in Subsection \ref{sub:proof ass b}). \\

\item In Section \ref{sec:w acts}, we prove Theorem \ref{thm:w acts} for an arbitrary smooth surface $S$ and invariants $r,c_1,H$ which satisfy $\gcd(r,c_1\cdot H) = 1$ (namely Assumption A). We note that the proof naturally lifts from $K$-theory to the derived category. \\

\end{itemize} 

\noindent In the present paper, we follow the notations of \cite{Univ}. For a scheme $X$, we will write $K_X$ for the Grothendieck group of the category of coherent sheaves on $X$. This notion straightforwardly generalizes to the dg schemes $X$ considered in the present paper, which are all derived zero loci of sections of locally free sheaves, in which case the appropriate setting for $K$-theory is the spectrum of the $\infty$-category of cohomologically bounded sheaves on $X$. The notation $\Delta$ and $\zeta$ will refer to the following: \\

\begin{itemize}[leftmargin=*]

\item In Sections \ref{sec:shuf} and \ref{sec:w}, $\Delta = (1-q_1)(1-q_2)$ and $\zeta(x) = \displaystyle \frac {(1-xq_1)(1-xq_2)}{(1-x)(1-xq_1q_2)}$ \\

\item In Sections \ref{sec:mod}, \ref{sec:shuf acts} and \ref{sec:w acts}, $\Delta \in K_{S \times S}$ denotes the class of the diagonal and:
$$
\zeta(x) = \wedge^\bullet(-x \cdot \CO_\Delta) \Big |_\Delta = \frac {\wedge^\bullet(x \cdot \Omega_S^1)}{(1-x)(1-x q)} \in \ks(x)
$$

\end{itemize}

\noindent The two bullets above are designed so that the usual algebraic notion of action of the algebra $\CA$ from Sections \ref{sec:shuf} and \ref{sec:w} matches the geometric notion of ``action" from Definition \ref{def:shuf acts}, via the prescription that $q_1+q_2 \leadsto [\Omega_S^1] \in \ks$ and $q_1q_2 \leadsto [\omega_S] \in \ks$. \\

\noindent I would like to thank Sergei Gukov, Tamas Hausel, Davesh Maulik, Alexander Minets, Georg Oberdieck, Francesco Sala, Olivier Schiffmann, Richard Thomas and Alexander Tsymbaliuk for many interesting discussions on the subject, and for all their help in understanding the framework of $W$--algebras and sheaves on surfaces. I gratefully acknowledge the support of NSF grant DMS--1600375. 

\section{Shuffle algebras}
\label{sec:shuf}

\medskip

\subsection{} Let $q=q_1q_2$. Throughout this paper, we will often encounter the ring:
\begin{equation}
\label{eqn:ring}
\BK = \BZ[q_1^{\pm 1}, q_2^{\pm 1}]^{\text{symmetric in }q_1,q_2}
\end{equation}
and its field of fractions:
\begin{equation}
\label{eqn:field}
\BF = \BQ(q_1,q_2)^{\text{symmetric in }q_1,q_2}
\end{equation}
Let us recall the trigonometric version of the Feigin-Odesskii shuffle algebra \cite{FO}: \\ 

\begin{definition}
\label{def:shuffle}

Consider the rational function:
\begin{equation}
\label{eqn:zeta univ}
\zeta(x) = \frac {(1-xq_1)(1-xq_2)}{(1-x)(1-xq)}
\end{equation}
and the vector space:
$$
\CS_\ebig = \bigoplus_{k=0}^\infty \BF(z_1,...,z_k)^\esym
$$
where the superscript $\esym$ refers to rational functions that are symmetric with respect to $z_1,...,z_k$. We endow the vector space $\CS_\ebig$ with the shuffle product:
\begin{equation}
\label{eqn:shuffle product}
R(z_1,...,z_k) * R'(z_1,...,z_{k'}) = 
\end{equation}
$$
= \esym \left[ R(z_1,...,z_k) R'(z_{k+1},...,z_{k+k'}) \prod^{1 \leq i \leq k}_{k+1 \leq j \leq k+k'} \zeta \left( \frac {z_i}{z_j} \right) \right]
$$
Define the \textbf{shuffle algebra}:
\begin{equation}
\label{eqn:shuffle}
\CS \subset \CS_\ebig
\end{equation}
to be the \underline{$\BK$-subalgebra} generated by the elements:
\begin{equation}
\label{eqn:shuffle gen}
E_{d_\bullet} = \esym \left[ \frac{z_1^{d_1} ... z_k^{d_k}}{\left(1 - \frac {z_2 q}{z_1} \right) ... \left(1 - \frac {z_k q}{z_{k-1}} \right)} \prod_{1\leq i < j \leq k} \zeta \left( \frac {z_i}{z_j} \right) \right]
\end{equation}
as $d_\bullet = (d_1,...,d_k)$ ranges over $\BZ^k$. \\
\end{definition}

\noindent The shuffle algebra in Definition \ref{def:shuffle} has been studied in numerous papers, most relevant to our purposes here being \cite{FHHSY,FT,W,SV2}. However, much of the literature has been concerned with studying the shuffle algebra over a field, namely:
\begin{equation}
\label{eqn:rat shuffle}
\CS \otimes_{\BK} \BF
\end{equation}
It was shown in \cite[Theorem 2.2 and Proposition 6.1]{Shuf} that the localized shuffle algebra \eqref{eqn:rat shuffle} is actually generated by $\{E_{(d)} = z_1^d\}_{d \in \BZ}$. For example, the identity:
$$
E_{(0)} * E_{(1)} - E_{(1)} * E_{(0)} = \frac {(1-q_1)(1-q_2)(1-q)z_1z_2 (z_1+z_2)}{(z_1-z_2 q)(z_2-z_1 q)} = (1-q_1)(1-q_2) E_{(1,0)}
$$
suggests that one can obtain the generator $E_{(1,0)}$ from shuffle products of $E_{(1)}$ and $E_{(0)}$ iff one inverts the element $(1-q_1)(1-q_2)$ in the ground ring. We will not do so in the present paper, and therefore emphasize the fact that the shuffle algebra $\CS$ of \eqref{eqn:shuffle} has a more complicated description than the algebra \eqref{eqn:rat shuffle}. For example, the latter admits a straightforward description in terms of Feigin-Odesskii type wheel conditions, but we do not know an analogous description of the former. \\


\begin{proposition}
\label{prop:gen}

The algebra $\CS$ is generated over $\BK$ by the elements:
\begin{equation}
\label{eqn:shuf gen}
E_{k,n} = q^{\gcd(k,n)-1} E_{(d_1,...,d_k)}
\end{equation}
where $d_i = \Big \lceil \frac {ni}k \Big \rceil - \Big \lceil {\frac {n(i-1)}k} \Big \rceil + \delta_i^k - \delta_i^1$, as $k$ ranges over $\BN$, and $n$ ranges over $\BZ$. \\

\end{proposition}

\begin{proof} Since $\CS$ is generated by $E_{d_\bullet}$ for arbitrary $d_\bullet = (d_1,...,d_k) \in \BZ^k$, it is enough to show that any given $E_{d_\bullet}$ can be written as a $\BK$-linear combination of products of \eqref{eqn:shuf gen}. We will prove this statement by induction on $k$. The case $k=1$ is trivial, and for the induction step, consider arbitrary $k$, $d_\bullet$ and let $n = d_1+...+d_k$. Then:
\begin{equation}
\label{eqn:monomial}
z_1^{d_1}... z_k^{d_k} - q^a \prod_{i=1}^k z_i^{\left \lceil \frac {ni}k \right \rceil - \left \lceil {\frac {n(i-1)}k} \right \rceil + \delta_i^k - \delta_i^1} = 
\end{equation}
$$
= \sum_{i=1}^{k-1} \left(1 - \frac {z_{i+1}q}{z_i} \right) \cdot \left( \text{linear combination of monomials in }q^{\pm 1},z^{\pm 1}_1,...,z^{\pm 1}_k \right)
$$
for some suitably chosen $a \in \BZ$. Plugging formula \eqref{eqn:monomial} in \eqref{eqn:shuffle gen} gives us:
\begin{equation}
\label{eqn:rait}
E_{d_\bullet} - q^{a+1-\gcd(k,n)} E_{k,n} = \sum_{i=1}^{k-1} \sum_{d'_\bullet, d''_\bullet} E_{d'_\bullet} * E_{d''_\bullet}
\end{equation}
where in the right-hand side, $d'_\bullet = (d_1',...,d_i')$ and $d''_\bullet = (d''_1,...,d''_{k-i-1})$ are certain vectors of integers modeled after the monomials that appear in the right-hand side of \eqref{eqn:monomial}. By the induction hypothesis, the right-hand side of \eqref{eqn:rait} can be written as a $\BK$-linear combination of products of the elements \eqref{eqn:shuf gen}, hence so can $E_{d_\bullet}$. 

\end{proof}

\subsection{} After establishing Proposition \ref{prop:gen}, the next step is to work out the relations between the generators $E_{k,n}$. The idea for doing so is inspired by \cite{BS}, and combining their construction with the results of \cite{Shuf}, allows us to conclude that, as $\BF$-modules:
\begin{equation}
\label{eqn:inna}
\CS \otimes_{\BK} \BF = \bigoplus_{\frac {n_1}{k_1} \leq ... \leq \frac {n_t}{k_t}}^{k_i \in\BN, \ n_i \in \BZ} \BF \cdot E_{k_1,n_1} ... E_{k_t,n_t} 
\end{equation}
In other words, any element of $\CS$ can be written as a linear combination of \textbf{ordered} monomials $E_{k_1,n_1} ... E_{k_t,n_t}$ if one allows the coefficients to be rational functions in $q_1$ and $q_2$. Our main purpose for the remainder of this Subsection is to show that the equality \eqref{eqn:inna} still holds without localization, i.e. the fact that, as $\BK$-modules:
\begin{equation}
\label{eqn:anna}
\CS = \bigoplus_{\frac {n_1}{k_1} \leq ... \leq \frac {n_t}{k_t}}^{k_i \in\BN, \ n_i \in \BZ} \BK \cdot E_{k_1,n_1} ... E_{k_t,n_t} 
\end{equation}
By Proposition \ref{prop:gen}, any element of $\CS$ can be written as a $\BK$-linear combination of products of $E_{k,n}$. Our task thus reduces to establishing an integral version of the ``straightening Lemma" of \cite{BS}: we must show that an arbitrary product of $E_{k,n}$'s is equal to a sum of ordered products of $E_{k,n}$'s, in non-decreasing order of $\frac nk$. This statement follows by repeated applications of the following result: \\


\begin{theorem}
\label{thm:comm}

For any $k,k' \in \BN$ and $n,n' \in \BZ$ such that $\frac nk > \frac {n'}{k'}$, we have:
\begin{equation}
\label{eqn:e triangle}
[E_{k,n}, E_{k',n'}] = \Delta \mathop{\sum_{\frac {n'}{k'} \leq \frac {n_1}{k_1} \leq ... \leq \frac {n_t}{k_t} \leq \frac {n}{k}}^{n_i \in \BZ, \ \sum n_i = n+n'}}^{k_i \in \BN, \ \sum k_i = k+k'} p_{n,n_1,...,n_t,n'}^{k,k_1,...,k_t,k'}(q_1,q_2) \cdot E_{k_1,n_1}... E_{k_t,n_t} 
\end{equation}
where $\Delta = (1-q_1)(1-q_2)$ and $p_{n,n_1,...,n_t,n'}^{k,k_1,...,k_t,k'}(q_1,q_2)$ are elements of $\BK$. Moreover, all products which appear with non-zero coefficient in the right-hand side of \eqref{eqn:e triangle} have the property that at least one fraction $\frac {n_i}{k_i}$ is strictly between $\frac nk$ and $\frac {n'}{k'}$. \\

\end{theorem}

\subsection{}

Theorem \ref{thm:comm} will be proved at the end of the current Section. The main idea is that the commutation relations \eqref{eqn:e triangle} for all $\frac nk > \frac {n'}{k'}$ are implied by the special case when the triangle spanned by the vectors $(k,n)$ and $(k',n')$ contains no lattice points  inside and on one of the edges. These special cases are usually presented in terms of the generators $P_{k,n}, H_{k,n}, Q_{k,n}$ defined as follows for all $\gcd(k,n) = 1$:
\begin{align}
&\sum_{s=0}^\infty \frac {E_{ks,ns}}{(-x)^s} = \exp \left[ -  \sum_{s=1}^\infty \frac {P_{ks,ns}}{sx^s} \right] \label{eqn:e} \\
&\sum_{s=0}^\infty \frac {H_{ks,ns}}{x^s} = \exp \left[\sum_{s=1}^\infty \frac {P_{ks,ns}}{sx^s} \right] \label{eqn:h} \\
&\sum_{s=0}^\infty \frac {Q_{ks,ns}}{x^s} = \exp \left[ \sum_{s=1}^\infty \frac {P_{ks,ns}}{sx^s} (1-q^{-s})\right] \label{eqn:q}
\end{align}
For fixed slope $\frac nk$, the $P$'s, $H$'s and $Q$'s are in relation to $E$'s as power sum, complete symmetric, and plethystically modified complete symmetric polynomials are in relation to elementary symmetric polynomials. In other words, presenting relations between $E$'s is equivalent with presenting relations between the other generators. These relations were first constructed in \cite{BS} (see \cite{W} for our conventions):
\begin{equation}
\label{eqn:zero triangle}
[P_{k,n}, P_{k',n'}] = 0
\end{equation}
if $\frac nk = \frac {n'}{k'}$, and:
\begin{equation}
\label{eqn:small triangle}
[P_{k,n}, P_{k',n'}] = \frac {(1 - q_1^s)(1 - q_2^s)}{1 - q^{-1}} \cdot Q_{k+k',n+n'}
\end{equation}
when the triangle $(0,0)$, $(k,n)$, $(k+k',n+n')$ is oriented clockwise, and contains no lattice points inside or on one of the edges (we write $s = \gcd(k,n) \gcd(k',n')$). \\

\begin{proposition}
\label{prop:e comm}

Relations \eqref{eqn:zero triangle} and \eqref{eqn:small triangle} imply the following relations in terms of the $E$ generators, for all $s\in \BN$:
\begin{equation}
\label{eqn:e triangle 1}
[E_{ks,ns}, E_{k',n'}] = \Delta \sum_{t=1}^s \frac {q_1^t - q_2^t}{q_1 - q_2} (-1)^{t-1} E_{kt+k',nt+n'} E_{k(s-t),n(s-t)}
\end{equation}
\begin{equation}
\label{eqn:e triangle 2}
[E_{k,n}, E_{k's,n's}] = \Delta \sum_{t=1}^s \frac {q_1^t - q_2^t}{q_1 - q_2} (-1)^{t-1} E_{k'(s-t),n'(s-t)} E_{k+k't,n+n't} 
\end{equation}
whenever the triangle spanned by the vectors $(k,n)$ and $(k',n')$ is oriented clockwise, and has the property that $\gcd(k,n) = \gcd(k',n') = \gcd(k+k',n+n') = 1$. Under the same assumptions, but allowing $\gcd(k+k',n+n') = s \geq 1$, we have:
\begin{equation}
\label{eqn:e triangle 3}
[E_{k,n}, E_{k',n'}] = \frac {\Delta}{1-q^{-1}} \left[ \text{coefficient of } \frac 1{x^s} \text{ in } \frac {E(xq)}{E(x)} \right]
\end{equation}
where we set $E(x) = \sum_{t=0}^\infty \frac {E_{(k+k')\frac ts, (n+n') \frac ts}}{(-x)^t}$. \\


\end{proposition}

\begin{proof} 

Under our assumptions on the vectors $(k,n)$, $(k',n')$, the triangle spanned by the vectors $(ks,ns)$, $(kt+k',nt+n')$ satisfies the assumptions of \eqref{eqn:small triangle}:
$$
[P_{ks,ns}, P_{kt+k',nt+n'}] = (1-q_1^s)(1-q_2^s) P_{k(s+t)+k',n(s+t)+n'}
$$
for all $s \in \BN, t\in \BZ$. Summing this relation over all $s \in \BN$ and $t \in \BZ$, we obtain:
$$
\left[ - \sum_{s=1}^\infty \frac {P_{ks,ns}}{sx^s}, \sum_{t \in \BZ} \frac{P_{kt+k',nt+n'}}{y^t} \right] = - \sum_{s=1}^\infty (1-q_1^s)(1-q_2^s) \frac {y^s}{sx^s} \sum_{t \in \BZ}   \frac {P_{k(s+t)+k',n(s+t)+n'}}{y^{s+t}}
$$
Let us recall the well-known formula:
$$
[X,Y] = c \cdot Y \text{ and } [X,c] = [Y,c] = 0 \quad \Rightarrow \quad \exp(X) Y = \exp(c)\cdot Y \exp(X)
$$
and apply it to $X = - \sum_{s=1}^\infty \frac {P_{ks,ns}}{sx^s}$ and $Y = \sum_{t \in \BZ} \frac{P_{kt+k',nt+n'}}{y^t} = \sum_{t \in \BZ} \frac{E_{kt+k',nt+n'}}{y^t}$: 
\begin{equation}
\label{eqn:wage}
\sum_{s=0}^\infty \frac {E_{ks,ns}}{(-x)^s} \sum_{t \in \BZ} \frac{E_{kt+k',nt+n'}}{y^t} = \zeta \left(\frac yx\right)^{-1} \sum_{t \in \BZ} \frac{E_{kt+k',nt+n'}}{y^t} \sum_{s=0}^\infty \frac {E_{ks,ns}}{(-x)^s} 
\end{equation}
Indeed, we have $P_{kt+k',nt+n'} = E_{kt+k',nt+n'}$ for all $t$ because our assumption on the vectors $(k,n)$ and $(k',n')$ implies $\gcd(kt+k',nt+n') = 1, \ \forall t$. Using the expansion:
$$
\zeta \left(\frac yx\right)^{-1} = 1 - (1-q_1)(1-q_2) \sum_{s=1}^\infty \frac {q_1^s-q_2^s}{q_1-q_2} \cdot \frac {y^s}{x^s} 
$$
and taking the coefficient of $x^{-s} y^0$ in equality \eqref{eqn:wage} yields \eqref{eqn:e triangle 1}. Relation \eqref{eqn:e triangle 2} is proved analogously. As for \eqref{eqn:e triangle 3}, this follows directly from \eqref{eqn:small triangle} since $E_{k,n} = P_{k,n}$, $E_{k',n'} = P_{k',n'}$ and \eqref{eqn:q} implies:
$$
\sum_{t=0}^\infty \frac {Q_{(k+k')\frac ts, (n+n') \frac ts}}{x^t} = \frac {E(xq)}{E(x)} 
$$ \end{proof}



\subsection{} If we expand the ratio of power series $E(xq)/E(x)$, we see that \eqref{eqn:e triangle 3} implies:
\begin{equation}
\label{eqn:e triangle 4}
[E_{k,n}, E_{k',n'}] = \Delta \Big( E_{k+k',n+n'} (-1)^{s-1}[q]_s + \dots \Big)
\end{equation}
where $[q]_s = 1+q^{-1}+...+q^{-s+1}$ and the ellipsis in \eqref{eqn:e triangle 4} stands for a sum of products of the form $E_{k_1,n_1}...E_{k_t,n_t}$ with $t>1$ and all lattice points $(k_i,n_i)$ lying on the line segment from $(0,0)$ to the lattice point $(k+k',n+n')$. \\

\begin{proof} \emph{of Theorem \ref{thm:comm}:} For the first half of the proof, we closely follow \cite{BS}, which will allow us to obtain the following slightly weaker version of \eqref{eqn:e triangle}:
\begin{equation}
\label{eqn:e triangle bis}
[E_{k,n}, E_{k',n'}] = \Delta \mathop{\sum_{\frac {n'}{k'} \leq \frac {n_1}{k_1} \leq ... \leq \frac {n_t}{k_t} \leq \frac nk}^{n_i \in \BZ, \ \sum n_i = n}}^{k_i \in \BN, \ \sum k_i = k} p_{n,n_1,...,n_t,n'}^{k,k_1,...,k_t,k'}(q_1,q_2) \cdot E_{k_1,n_1}... E_{k_t,n_t} 
\end{equation}
where: 
\begin{equation}
\label{eqn:coefficients}
p_{n,n_1,...,n_t,n'}^{k,k_1,...,k_t,k'}(q_1,q_2) \in \BK_\loc := \BK_{(1+q+...+q^{s-1})_{s \in \BN}}
\end{equation}
Then we will use the methods of \cite{Shuf} to show that the expressions \eqref{eqn:coefficients} do not have any poles at $q = \text{non-trivial root of unity}$, and so we will conclude that they actually lie in $\BK$, which is precisely what the Theorem claims. We call a product:
\begin{equation}
\label{eqn:product}
e = E_{k_1,n_1},...,E_{k_t,n_t}
\end{equation}
$\BK-$\textbf{orderable} (respectively $\BK_\loc-$\textbf{orderable}) if it can be written as in the right-hand side of \eqref{eqn:e triangle} (respectively \eqref{eqn:e triangle bis}). Consider the assignment: 
$$
\Big( E_{k_1,n_1}...E_{k_t,n_t} \in \CS \Big) \stackrel{\Upsilon}\longrightarrow \Big( \text{the lattice path }P \Big)
$$
where $P$ starts at $(0,0)$ and is built out of the segments $(k_1,n_1),...,(k_t,n_t)$ in this order. This is clearly a one-to-one correspondence between products \eqref{eqn:product} and lattice paths starting at the origin and pointing in the right half plane. The products that appear in the right-hand side of \eqref{eqn:e triangle} and \eqref{eqn:e triangle bis} all correspond to convex paths. Given any lattice path $P$, its convexification $P^{\conv}$ is defined as the path built out of the same segments $(k,n)$ as $P$, but in non-decreasing order of slope $\frac nk$ (for segments of equal slope, their relative order may be chosen arbitrarily). The area $a(P)$ of the path $P$ is defined as the area of the polygon bounded by $P$ and $P^{\conv}$, and we note that it is always a natural number. It was shown in \cite{BS} that if:
\begin{equation}
\label{eqn:product bis}
\tilde{e} = E_{k,n}E_{k',n'}
\end{equation}
is $\BK$-orderable whenever  $a(\Upsilon(\tilde{e})) \leq \delta$, then any product \eqref{eqn:product} such that $a(\Upsilon(e)) \leq \delta$ is also $\BK$-orderable. The same proof works if we replace the ring $\BK$ with $\BK_\loc$. Therefore, to prove \eqref{eqn:e triangle bis} by induction on $\delta \in \BN$, it suffices to prove the following: 

\begin{center}
\textbf{Assume that any }$e$\textbf{ as in \eqref{eqn:product} with } $a(\Upsilon(e)) < \delta$ \textbf{ is } $\BK_\loc-$\textbf{orderable}, \newline
\text{ } \ \qquad \textbf{then any }$\tilde{e}$\textbf{ as in \eqref{eqn:product bis} with } $a(\Upsilon(\tilde{e})) = \delta$ \textbf{ is } $\BK_\loc-$\textbf{orderable} \\
\end{center}

\noindent Let us now prove the claim in boldface letters above. Choose any $\tilde{e}$ as in \eqref{eqn:product bis} such that $a(P) = \delta$, where $P$ is the path with segments $v = (k,n)$ and $v' = (k',n')$. From now on, the phrase \emph{the triangle determined by vectors $v$ and $v'$} will refer to the triangle with a vertex at $(0,0)$ and with edges given by drawing the vectors $v$ and $v'$ in this order. If this triangle has no lattice points inside or on one of the edges, then \eqref{eqn:product bis} is $\BK$-orderable by \eqref{eqn:e triangle 1}--\eqref{eqn:e triangle 3}. Otherwise, there will exist a lattice point $v_0 = (k_0,n_0)$ inside the triangle determined by the vectors $v$ and $v'$, but not on the edges corresponding to the vectors $v$ and $v'$. In the following argument, we will assume that $k_0 < k$, as the case $k_0 > k$ is dealt with by switching the roles of $v$ and $v'$ throughout \footnote{One also needs to exclude the situation in which the only lattice point $v_0 = (k_0,n_0)$ with the properties above satisfies $k_0 = k$, but this implies $k|n$, $k'|n'$ and can be resolved with ease (\cite{BS})}. Let us choose $v_0$ such that the area of the triangle determined by $v-v_0$ and $v_0$ is minimal. This implies that the latter triangle respects the hypothesis of relation \eqref{eqn:e triangle 3}, or its equivalent form \eqref{eqn:e triangle 4}, and so we have:
\begin{equation}
\label{eqn:rel}
E_v \in \frac {(-1)^{s-1}}{[q]_s} \left(\frac {[E_{v-v_0},E_{v_0}]}{\Delta} + \sum_{v_1,...,v_t \text{ divide } v}^{t > 1, \ v_1+...+v_t = v} \BK \cdot E_{v_1}...E_{v_t} \right)
\end{equation}
where $s = \gcd(k,n)$, and we write $E_v$ instead of $E_{k,n}$ if $v = (k,n)$. Taking the commutator of \eqref{eqn:rel} with $E_{v'}$ yields the following, in virtue of the Jacobi identity:
$$
[E_v,E_{v'}] \in \frac {(-1)^{s-1}}{[q]_s}  \left( \frac {[[E_{v-v_0}, E_{v_0}],E_{v'}]}{\Delta} + \sum_{v_1,...,v_t \text{ divide } v}^{t > 1, \ v_1+...+v_t = v} \BK \cdot [E_{v_1}...E_{v_t},E_{v'}] \right) \subset 
$$
$$
\subset \frac {(-1)^{s-1}}{[q]_s}  \left( \frac {[[E_{v'},E_{v_0}],E_{v-v_0}]}{\Delta} +  \frac {[[E_{v-v_0},E_{v'}],E_{v_0}]}{\Delta} + \right.
$$
\begin{equation}
\label{eqn:bs}
\left. + \sum_{v_1,...,v_t \text{ divide } v}^{t > 1, \ v_1+...+v_t = v} \sum_{s=1}^t \BK \cdot E_{v_1}...E_{v_{s-1}}[E_{v_s},E_{v'}]E_{v_{s+1}}...E_{v_t} \right)
\end{equation}
We claim that all summands in the right-hand side of \eqref{eqn:bs} are $\BK_\loc-$orderable. Indeed, by our choice of the vector $v_0$, all commutators that appear in the right-hand side correspond to paths whose area is $<\delta$. Therefore, by the induction hypothesis in boldface letters, we may express such a commutator as a sum over convex paths, and it was shown in \cite{BS} that all paths obtained in this manner in the right-hand side of expression \eqref{eqn:bs} will still have area $<\delta$. The key geometric statement here, proved in \loccitt, is that if one takes two consecutive segments which violate convexity in a path $P'$, and one replaces them by an arbitrary convex path $P_0$ between the same endpoints, the resulting path $P''$ has $a(P'') < a(P')$. \\

\noindent We conclude that the right-hand side of \eqref{eqn:bs} can be written as a sum over convex paths $P$ of the elements $\Upsilon^{-1}(P)$. Because every commutator brings down a factor of $\Delta$ (as follows from relation \eqref{eqn:e triangle} when the area of the triangle determined by $(k,n)$ and $(k',n')$ is $<\delta$, which we may assume as part of our induction hypothesis), we see that the coefficient of any $\Upsilon^{-1}(P)$ in the right-hand side of \eqref{eqn:bs} lies in: 
$$
\frac {(-1)^{s-1} \Delta}{[q]_s} \cdot \BK_\loc = \Delta \cdot \BK_\loc
$$
Now assume, for the purpose of contradiction, that a certain $p_{n,n_1,...,n_t,n'}^{k,k_1,...,k_t,k'}(q_1,q_2)$ that appears in the right-hand side of \eqref{eqn:e triangle bis} lies in $\BK_\loc \backslash \BK$, i.e. has a pole when $q$ is a non-trivial root of unity. Then the pole will remain when we change the right-hand side of \eqref{eqn:e triangle bis} from the basis $E_{k_1,n_1}...E_{k_t,n_t}$ to $P_{k_1,n_1}...P_{k_t,n_t}$, because the matrix transforming elementary symmetric polynomials $E_{k_i,n_i}$ into power-sum functions $P_{k_i,n_i}$ is invertible and has rational coefficients. From \eqref{eqn:shuffle gen}, it is easy to see that $[E_{k,n}, E_{k',n'}] \in \CS$ is a rational function of the form:
\begin{equation}
\label{eqn:good shuffle}
R = \frac {r(z_1,...,z_k)}{\prod_{1 \leq i \neq j \leq k} (z_i q - z_j)}
\end{equation}
where $r \in \kk[z_1^{\pm 1}, ..., z_k^{\pm 1}]^\sym$. By \eqref{eqn:anna}, we may express any such element as:
$$
R =  \sum_{\frac {n_1}{k_1} \leq ... \leq \frac {n_t}{k_t}}^{k_i \in\BN, \ n_i \in \BZ} \gamma_{n_1,...,n_t}^{k_1,...,k_t} \cdot P_{k_1,n_1} ... P_{k_t,n_t}
$$
where $\gamma_{n_1,...,n_t}^{k_1,...,k_t} \in \BF$. Then all that remains to prove is that the coefficients $\gamma_{n_1,...,n_t}^{k_1,...,k_t}$ do not have any poles when $q$ is a non-trivial root of unity. There exists a pairing: 
\begin{equation}
\label{eqn:pair}
\langle \cdot, \cdot \rangle : \CS \otimes \CS \rightarrow \BF
\end{equation}
for which the basis vectors $P_{k_1,n_1}...P_{k_t,n_t}$ are orthogonal (see \cite[Proposition 5.4]{Shuf}) and the coefficients we wish to express are given by:
\begin{equation}
\label{eqn:mal}
\gamma_{n_1,...,n_t}^{k_1,...,k_t} = \frac {\langle R, P_{k_1,n_1}...P_{k_t,n_t} \rangle}{\langle P_{k_1,n_1}...P_{k_t,n_t}, P_{k_1,n_1}...P_{k_t,n_t} \rangle}
\end{equation}
The goal is to show that the right-hand side of \eqref{eqn:mal} does not have poles when $q$ is a non-trivial root of unity. According to formula (7.15) of \cite{W}, we have:
\begin{equation}
\label{eqn:coeff 1}
\langle P_{k_1,n_1}...P_{k_t,n_t}, P_{k_1,n_1}...P_{k_t,n_t} \rangle = \text{integer}\prod_{i=1}^t \frac {(1-q_1^{s_i})(1-q_2^{s_i})(1-q^{-1})^{k_i}}{(1-q_1)^{k_i}(1-q_2)^{k_i}(1-q^{-s_i})}
\end{equation}
where $s_i = \gcd(k_i,n_i)$. According to formula (2.8) of \cite{W}, we have for all $k,n$ with greatest common divisor $s$, the equality:
\begin{equation}
\label{eqn:p1}
P_{k,n} = \sym \left[ \frac {\prod_{i=1}^k z_i^{\left \lfloor \frac {in}k \right \rfloor - \left \lfloor \frac {(i-1)n}k \right \rfloor}}{\prod_{i=1}^{k-1} \left(1 - \frac {qz_{i+1}}{z_i} \right)} \sum_{t=0}^{s-1} q^{t} \frac {z_{a(s-1)+1}...z_{a(s-t)+1}}{{z_{a(s-1)} ...z_{a(s-t)}}} \prod_{i < j} \zeta \left( \frac {z_i}{z_j} \right) \right]
\end{equation}
where $a = \frac ks$. We claim that the right-hand side of \eqref{eqn:p1} equals:
\begin{equation}
\label{eqn:p2}
P_{k,n} = c \cdot \sym \left[ \frac {\prod_{i=1}^k z_i^{\left \lfloor \frac {in}k \right \rfloor - \left \lfloor \frac {(i-1)n}k \right \rfloor}}{\prod_{i=1}^{k-1} \left(1 - \frac {z_{i+1}}{q_2z_i} \right)} \sum_{t=0}^{s-1} \frac {z_{a(s-1)+1}...z_{a(s-t)+1}}{{q_2^t z_{a(s-1)} ...z_{a(s-t)}}} \prod_{i < j} \zeta \left( \frac {z_i}{z_j} \right) \right]
\end{equation}
where the coefficient $c$ is given by:
\begin{equation}
\label{eqn:coeff 2}
c = \frac {(1-q_2^s)(1-q^{-1})^k}{(1-q_2)^k (1-q^{-s})}
\end{equation}
Indeed, let us denote $p_1 = \text{RHS of \eqref{eqn:p1}}$ and $p_2 = \text{RHS of \eqref{eqn:p2}}$, and we must prove that $p_1$ and $p_2$ are equal: both $p_1$ and $p_2$ are shuffle elements in $k$ variables of homogeneous degree $n$ whose coproduct is given by formula (5.4) of \cite{Shuf} (to recall the coproduct on the shuffle algebra, as well as the computation of the coproduct of $p_1$ and $p_2$, we refer to Proposition 6.4 of \loccitt). Therefore, Lemma 5.5 of \loccit implies that $p_1$ and $p_2$ are equal up to a constant multiple. To prove that this constant multiple is 1, we apply the linear map: 
$$
R(z_1,..,z_k) \stackrel{\ph}\longrightarrow R(1, q_1^{-1},...,q_1^{-k+1})
$$
to $p_1$ and $p_2$. Because $\zeta(q_1^{-1}) = 0$, only one of the $k!$ summands that make up the symmetrizations defining $p_1$ and $p_2$ is not annihilated by the linear map $\ph$, namely the summand corresponding to the identity permutation. Then the fact that $\ph(p_1) = \ph(p_2)$ is easy to check, thus implying the equivalence of formulas \eqref{eqn:p1} and \eqref{eqn:p2}. Combining \eqref{eqn:mal}, \eqref{eqn:coeff 1}, \eqref{eqn:p2}, \eqref{eqn:coeff 2}, all that remains to prove is: \\

\begin{claim}
\label{claim:pair 1}

For any shuffle element $R$ as in \eqref{eqn:good shuffle} and any $\rho \in \kk[z_1^{\pm 1},...,z_k^{\pm 1}]$ let:
\begin{equation}
\label{eqn:claim pair}
P = \esym \left[ \frac {\rho(z_1,...,z_k)}{\prod_{i=1}^{k-1} \left(1 - \frac {z_{i+1}}{q_2z_i} \right)} \prod_{i < j} \zeta \left( \frac {z_i}{z_j} \right) \right]
\end{equation}
Then the quantity $\langle R, P \rangle$ has no poles when $q$ is a root of unity (and $q_1$ is generic).\\

\end{claim}

\noindent At this step, we must recall that the pairing \eqref{eqn:pair} was defined in \cite{Shuf} by:
\begin{multline} 
\label{eqn:pair formula}
\left \langle R, \sym \left[ z_1^{d_1}... z_k^{d_k} \prod_{1 \leq i < j \leq k} \zeta \left( \frac {z_i}{z_j} \right) \right] \right \rangle = \\
= \int_{|z_1| \gg ... \gg |z_k|} \frac {R(z_1,...,z_k) z_1^{-d_1}...z_k^{-d_k}}{\prod_{1 \leq i < j \leq k} \zeta \left( \frac {z_i}{z_j} \right)} \prod_{i=1}^k \frac {dz_i}{2\pi i z_i}
\end{multline}
for any $d_1,...,d_k \in \BZ$, and $\BF$-linearity in the second argument. This is sufficient to completely determine the pairing, as the elements $\text{Sym}[...]$ that feature in the left-hand side of \eqref{eqn:pair formula}, as $d_1,...,d_k$ range over $\BZ$, span $\CS \otimes_{\BK} \BF$ (as shown in \cite{Shuf}). \\

\begin{claim}
\label{claim:pair 2}

For any $R(z_1,...,z_k), R'(z_1,...,z_k) \in \CS$, we have:
\begin{equation}
\label{eqn:is pairing}
\langle R, R' \rangle = \frac 1{k!} \int_{|z_1| = ... = |z_k|} \frac {R(z_1,...,z_k) R' \left( \frac 1{z_1},..., \frac 1{z_k} \right)}{\prod_{1 \leq i \neq j \leq k} \left[ \zeta \left( \frac {z_i}{z_j} \right) \frac {z_i - p z_j}{z_i - q z_j} \right]} \prod_{i=1}^k \frac {dz_i}{2\pi i z_i} \Big|_{p \mapsto q}
\end{equation}
The integral must be computed by residues under the assumptions $|q_1|, |q_2| > 1 > |p|$, and only after one evaluates the integral, one must specialize $p \mapsto q$. \\

\end{claim}

\begin{proof} \emph{of Claim \ref{claim:pair 2}:} It suffices to show that formula \eqref{eqn:pair formula} matches \eqref{eqn:is pairing} when:
$$
R' = \sym \left[ z_1^{d_1}... z_k^{d_k} \prod_{1 \leq i < j \leq k} \zeta \left( \frac {z_i}{z_j} \right) \right]
$$
for arbitrary $d_1,...,d_k \in \BZ$. In this case, we have:
$$
\text{RHS of \eqref{eqn:is pairing}} = \int_{|z_1| = ... = |z_k|} \frac {R(z_1,...,z_k) z_1^{-d_1}...z_k^{-d_k}}{\prod_{1 \leq i < j \leq k} \zeta \left( \frac {z_i}{z_j} \right) \prod_{1\leq i \neq j \leq k} \frac {z_i - p z_j}{z_i - q z_j}} \prod_{i=1}^k \frac {dz_i}{2\pi i z_i} \Big|_{p \mapsto q} =
$$
\begin{equation}
\label{eqn:jor}
= \int_{|z_1| \gg ... \gg |z_k|} \frac {R(z_1,...,z_k) z_1^{-d_1}...z_k^{-d_k}}{\prod_{1 \leq i < j \leq k} \zeta \left( \frac {z_i}{z_j} \right) \prod_{1\leq i \neq j \leq k} \frac {z_i - p z_j}{z_i - q z_j}} \prod_{i=1}^k \frac {dz_i}{2\pi i z_i} \Big|_{p \mapsto q}
\end{equation}
The second equality above is due to the fact we do not pick up any poles as we move the contours from $|z_1| = ... = |z_k|$ to $|z_1| \gg ... \gg |z_k|$, as a consequence of the assumption $|q_1|, |q_2| > 1 > |p|$.  However, the integral over $|z_1| \gg ... \gg |z_k|$ is a Laurent polynomial in $q_1,q_2$ and $p$, and thus one would not change the value of the integral if one removed the fraction $\frac {z_i - p z_j}{z_i - q z_j}$ and the symbol $|_{p \mapsto q}$. With this in mind, \eqref{eqn:jor} matches the right-hand side of \eqref{eqn:pair formula}, as we needed to prove. 

\end{proof}

\begin{proof} \emph{of Claim \ref{claim:pair 1}:} Let us re-run the argument that proved Claim \ref{claim:pair 2} with $R'$ replaced by $P$ of \eqref{eqn:claim pair}. Because of the extra factors $1 - \frac {z_{i+1}}{q_2z_i}$ in the denominator of $P$, equality \eqref{eqn:jor} does not hold as stated anymore. Instead, we have:
$$
\langle R, P \rangle = \int_{|z_1| = ... = |z_k|} \frac {R(z_1,...,z_k) \rho \left( \frac 1{z_1},..., \frac 1{z_k} \right)}{\prod_{i=1}^{k-1} \left( 1 - \frac {z_i}{q_2z_{i+1}} \right) \prod_{i < j} \zeta \left( \frac {z_i}{z_j} \right) \prod_{i\neq j} \frac {z_i - p z_j}{z_i - q z_j}} \prod_{i=1}^k \frac {dz_i}{2\pi i z_i} \Big|_{p \mapsto q} = 
$$ 
\begin{multline}
= \sum^{\text{composition}}_{k = n_1+...+n_t} \left[ \int_{|w_1| \gg ... \gg |w_t|} \prod_{s=1}^t \frac {dw_s}{2\pi i w_s} \right. \\ \left. \frac {R(z_1,...,z_k) \rho \left( \frac 1{z_1},..., \frac 1{z_k} \right)}{\prod_{s=1}^{t-1} \left( 1 - \frac {z_{n_1+...+n_s}}{q_2z_{n_1+...+n_s+1}} \right) \prod_{i < j} \zeta \left( \frac {z_i}{z_j} \right) \prod_{i\neq j} \frac {z_i - p z_j}{z_i - q z_j}} \Big|_{z_{n_1+...+n_{s-1}+a} \mapsto \frac {w_s}{q_2^{a-1}}} \right] \Big|_{p \mapsto q} \label{eqn:ah}
\end{multline}
because as we move the contours from $|z_1| = ... = |z_k|$ toward $|z_1| \gg ... \gg |z_k|$, we can now pick up residues from the poles $z_{i+1} = q_2^{-1} z_i$. To prove Claim \ref{claim:pair 1}, it suffices to prove that none of the summands in the right-hand side of \eqref{eqn:ah} has a pole when $q$ is a root of unity. Corresponding to any composition $k = n_1+...+n_t$, the integrand \eqref{eqn:ah} is of the form:
\begin{equation}
\label{eqn:tina}
\frac {\text{Laurent polynomial in }z_1,...,z_k}{\prod_{s=1}^{t-1} \left( 1 - \frac {z_{n_1+...+n_s}}{q_2 z_{n_1+...+n_s+1}} \right) \prod_{i < j} (z_iq_1-z_j)(z_iq_2-z_j) \prod_{i \neq j} (z_i - p z_j)}
\end{equation}
The fact that the contribution of \eqref{eqn:tina} to \eqref{eqn:ah} does not produce any poles when $q$ is a root of unity is a consequence of the following observations. \\

\begin{enumerate}[leftmargin=*]

\item as we specialize the variables in \eqref{eqn:tina} to $z_{n_1+...+n_{s-1}+a} = w_s q_2^{-a+1}$, the linear factors in the denominator will all be of the form $1-q_1q_2^a$, $1-q_2^a$, $1-pq_2^a$ (for various integers $a$; these factors are all regular when $q$ is a root of unity) or $w_i - w_j c$ for $i > j$ and various $c \in q_1^{\BZ} q_2^{\BZ}p^{\BZ}$. \\

\item the integral as $|w_1| \gg ... \gg |w_t|$ of any rational function of the form:
$$
\frac {\text{Laurent polynomial in }w_1,...,w_t}{\prod_{i > j, \text{various scalars }c}(w_i - w_j c)}
$$
is a polynomial in the coefficients of the numerator and the various $c$'s that appear in the denominator. This statement is simply the $t$-fold iteration of the elementary fact that the residue at $w = \infty$ of: 
$$
\frac {\text{Laurent polynomial in }w}{(w - \alpha_1)... (w - \alpha_n)}
$$
is a polynomial in the coefficients of the numerator and $\alpha_1,...,\alpha_n$. 

\end{enumerate} 

\end{proof} \end{proof}

\section{$W$--algebras}
\label{sec:w}

\medskip

\subsection{}  

One of the main results of \cite{W} was to realize the deformed $W$--algebra inside a double shuffle algebra, a construction which we will now recall. Starting from the shuffle algebra $\CS$ of \eqref{eqn:shuffle}, one constructs the $\ring$-algebra:
\begin{equation}
\label{eqn:def double}
\CA = \CA^\leftarrow \otimes \CA^{\text{diag}} \otimes \CA^\rightarrow
\end{equation}
where:
\begin{align*} 
&\CA^\leftarrow = \CS \quad \ \ \text{with generators denoted by }\{E_{-k,n}\}^{k \in \BN}_{n \in \BZ} \\
&\CA^\rightarrow = \CS^{\op} \quad \text{with generators denoted by } \{E_{k,n}\}^{k \in \BN}_{n \in \BZ} \\
&\CA^{\text{diag}} = \BK[c^{\pm 1},E_{0,k}]_{k \in \BZ \backslash 0}
\end{align*}
Both algebras $\CS$ and $\CA$ are generated over $\BK$ by symbols $E_{k,n}$, but the former has $k \in \BN$ and the latter has $k \in \BZ$. The relations between these generators in the algebra $\CA$ are modeled after the relations \eqref{eqn:zero triangle} and \eqref{eqn:small triangle}, but with certain small modifications (originally defined in \cite{BS}, but see \cite{W} for our conventions):
\begin{equation}
\label{eqn:shuf heis 0}
[P_{k,n}, P_{k',n'}] = \delta_{k+k'}^0 s \frac {(1-q_1^s)(1-q_2^s)}{1-q^{-s}} \cdot (1-c^k) 
\end{equation}
if $kn' = k'n$ and $k<0$, where $s = \gcd(k,n)$, and:
\begin{equation}
\label{eqn:small triangle double}
[P_{k,n}, P_{k',n'}] = \frac {(1-q_1^s)(1-q_2^s)}{1-q^{-1}} \cdot c^* Q_{k+k',n+n'} 
\end{equation}
if $kn' < k'n$ and the triangle with vertices $(0, 0)$, $(k, n)$, $(k + k', n + n')$
contains no lattice points inside or on one of the edges, and $s$ denotes $\gcd(k,n)\gcd(k'n')$. The particular power $c^*$ in formula \eqref{eqn:small triangle double} can be found in (2.21) of \cite{W}, but it will not be relevant to us. Throughout the present paper, we will set $c = q^r$ for a natural number $r$, in order to cancel the denominator of \eqref{eqn:shuf heis 0}. Since the $Q_{k,n}$ are still defined by \eqref{eqn:q}, they are also multiples of $1-q$, and this cancels the denominator of \eqref{eqn:small triangle double}. \\

\noindent Since the generators $P_{k,n}$ and $Q_{k,n}$ are connected with the generators $E_{k,n}$ by formulas \eqref{eqn:e} and \eqref{eqn:q}, respectively, one may convert relations \eqref{eqn:shuf heis 0} and \eqref{eqn:small triangle double} into commutation relations involving the $E$'s. We will show how to do so for the former of these relations, and leave the latter as an exercise to the interested reader (it will differ from \eqref{eqn:e triangle 1}--\eqref{eqn:e triangle 3} by some powers of $c$). \\

\begin{proposition}
\label{prop:zero}

Assume $s \in - \BN$, $k \in \BN$ and $\gcd(k,n) = 1$. Then \eqref{eqn:shuf heis 0} implies:
\begin{multline}
\label{eqn:shuf heis}
[E_{ks,ns}, E_{ks',ns'}] \Big|_{c \mapsto q^r}  = \\ = \begin{cases} 0 & \text{if } s'<0  \\
\Delta \sum_{i=1}^{\min(-s,s')}  \gamma_i E_{k(s'-i),n(s'-i)} E_{k(s+i),n(s+i)}  & \text{if } s'>0 \end{cases}
\end{multline}
for some $\gamma_i \in \BK$. \\ 

\end{proposition}

\begin{proof} Formula \eqref{eqn:shuf heis} is easy when $s' < 0$, since if $P_{k,n}, P_{2k,2n},P_{3k,3n}...$ all commute, then formula \eqref{eqn:e} implies that $E_{k,n}, E_{2k,2n},E_{3k,3n}...$ also all commute. On the other hand, when $s' > 0$ relation \eqref{eqn:shuf heis 0} reads:
\begin{align*}
&[P_{ks,ns}, P_{ks',ns'}] \Big|_{c \mapsto q^r}  = \delta_{s+s'}^0 s (1-q_1^s)(1-q_2^s)(1+q^{-s}+...+q^{-s(kr-1)}) \Rightarrow \\
&\Rightarrow \left[  \sum_{s=-\infty}^{-1} \frac {P_{ks,ns}}{sx^{-s}}, \sum_{s'=1}^{\infty} \frac {P_{ks',ns'}}{-s'y^{s'}} \right] = \sum_{i = 1}^\infty \frac {(1-q_1^i)(1-q_2^i)(1+q^{-i}+...+q^{-i(kr-1)})}{i x^i y^i} 
\end{align*}
We leave the following claim as an easy exercise: if $[P,P'] = \gamma$ with $\gamma$ central, then $\exp(P)\exp(P') = \exp(\gamma)\exp(P')\exp(P)$. With this in mind, we obtain:
\begin{equation}
\label{eqn:dean}
\sum_{s=-\infty}^{-1} \frac {E_{ks,ns}}{(-x)^{-s}} \sum_{s'=1}^{\infty} \frac {E_{ks',ns'}}{(-y)^{s'}} \Big|_{c \mapsto q^r} = \phi(xy) \sum_{s'=1}^{\infty} \frac {E_{ks',ns'}}{(-y)^{s'}} \sum_{s=-\infty}^{-1} \frac {E_{ks,ns}}{(-x)^{-s}} \Big|_{c \mapsto q^r}
\end{equation}
where:
$$
\phi(z) = \exp \left( \sum_{i = 1}^\infty \frac {(1-q_1^i)(1-q_2^i)(1+q^{-i}+...+q^{-i(kr-1)})}{i z^i} \right) 
$$
lies in $1 + \Delta \ring[[z^{-1}]]$. Taking the coefficient of $(-x)^s (-y)^{-s'}$ in \eqref{eqn:dean} yields \eqref{eqn:shuf heis}.

\end{proof}

\begin{proposition}
\label{prop:hecke}

Recall that $P_{0,k} \in \CA^{\emph{diag}}$ are to $E_{0,k} \in \CA^{\emph{diag}}$ as power-sum functions are to elementary symmetric polynomials. Then $\forall k \in \BZ \backslash 0$, we have:
\begin{equation}
\label{eqn:hecke right}
[P_{0,k}, R(z_1,...,z_n)] = (\emph{sign } k) (1-q_1^{|k|})(1-q_2^{|k|}) (z_1^k + ... + z_n^k) R(z_1,...,z_n)
\end{equation}
for all $R(z_1,...,z_n) \in \CS^{\emph{op}} \cong \CA^\rightarrow$, and:
\begin{equation}
\label{eqn:hecke left}
[P_{0,k}, R(z_1,...,z_n)] = - (\emph{sign } k) (1-q_1^{|k|})(1-q_2^{|k|}) (z_1^k + ... + z_n^k) R(z_1,...,z_n)
\end{equation}
for all $R(z_1,...,z_n) \in \CS \cong \CA^\leftarrow$. \\

\end{proposition}

\begin{proof} Since the algebra $\CA$ is a free $\ring$-module, it is enough to prove the required formulas over $\field$. According to Theorem 2.5 of \cite{Shuf}, the shuffle algebra is generated by $\{z_1^{k'}\}_{k' \in \BZ}$ over $\field$, and so the required formulas follow from the particular case $n=1$. In this case, the required relations boil down to:
\begin{align*}
&[P_{0,k}, E_{1,k'}] = (\sgn k) (1-q_1^{|k|})(1-q_2^{|k|}) E_{1,k+k'} \\
&[P_{0,k}, E_{-1,k'}] = - (\sgn k) (1-q_1^{|k|})(1-q_2^{|k|}) E_{-1,k+k'}
\end{align*}
which are just particular cases of \eqref{eqn:small triangle double}. 

\end{proof}

\subsection{}

We will now switch from the notation $E_{k,n}$ to $E_{n,k}$, as it will be more suitable for our study of the subalgebras of $\CA$ that we will introduce in Subsection \ref{sub:completion}. Formulas \eqref{eqn:anna} and \eqref{eqn:def double} allow us to find a $\BK$-basis of $\CA$:
\begin{equation}
\label{eqn:anna double}
\CA \Big|_{c \mapsto q^r} = \bigoplus_{(n_1,k_1) \ccur ... \ccur (n_t,k_t)} \BK \cdot E_{n_1,k_1} ... E_{n_t,k_t} 
\end{equation}
where the sum goes over all collections $(n_1,k_1),...,(n_t,k_t) \in \BZ^2 \backslash (0,0)$ ordered clockwise.  Here and below, we say that two lattice points $(n,k)$ and $(n',k')$ are ordered clockwise, denoted by:
\begin{equation}
\label{eqn:leq 1}
(n,k) \ccur (n',k')
\end{equation}
if one can reach the latter from the former by turning clockwise around the origin, without crossing the negative $y$ axis. If we wish to exclude the situation when $(n,k)$ and $(n',k')$ lie on the same ray through the origin, then we will use the notation:
\begin{equation}
\label{eqn:leq 2}
(n,k) \cur (n',k')
\end{equation}
instead. Formulas \eqref{eqn:leq 1} and \eqref{eqn:leq 2} are simply inequalities $\leq$ and $<$ on the slopes, appropriately defined, of the lattice points $(n,k)$ and $(n',k')$. Just like one can deduce \eqref{eqn:anna} from Theorem \ref{thm:comm}, one can deduce \eqref{eqn:anna double} from the following Theorem: \\

\begin{theorem}
\label{thm:comm double}

For any lattice points $(n',k') \curvearrowright (n,k)$, we have:
\begin{equation}
\label{eqn:e triangle double}
[E_{n,k}, E_{n',k'}] \Big|_{c \mapsto q^r} = \Delta \mathop{\sum_{(n',k') \ccur (n_1,k_1) \ccur ... }^{n_i \in \BZ, \ \sum n_i = n+n'}}^{k_i \in \BZ, \ \sum k_i = k+k'}_{... \ccur (n_t, k_t) \ccur (n,k)} p^{n,n_1,...,n_t,n'}_{k,k_1,...,k_t,k'}(q_1,q_2) \cdot E_{n_1,k_1}... E_{n_t,k_t} 
\end{equation}
for some $p^{n,n_1,...,n_t,n'}_{k,k_1,...,k_t,k'}(q_1,q_2) \in \BK$. All products which appear with non-zero coefficient in the right-hand side of \eqref{eqn:e triangle} have the property that $(n',k') \cur (n_i,k_i) \cur (n,k)$ for some $i$. The equations \eqref{eqn:e triangle double} generate the ideal of relations between $E_{n,k} \in \CA$. \\

\end{theorem}

\noindent We specialize $c = q^r$ in \eqref{eqn:e triangle double} for two reasons: firstly, so that formula \eqref{eqn:e triangle double} matches formula \eqref{eqn:shuf heis}, which requires $c = q^r$ in order to clear denominators. Secondly, formula \eqref{eqn:small triangle double} and its equivalent with $P$'s replaced with $E$'s, feature certain powers of $c$ in the right-hand side, which only become elements of $\BK$ upon specialization $c = q^r$. \\

\begin{proof} The generators $E_{n,k}$ (alternatively $P_{n,k}$) of the algebra $\CA$ are permuted by $SL_2(\BZ)$ acting on the indices. As observed in \cite{BS}, the relations \eqref{eqn:shuf heis 0} and \eqref{eqn:small triangle double} are not quite invariant under this $SL_2(\BZ)$ due to the various powers of $c$ that appear in the right-hand sides of these formulas, but they are invariant under the universal cover of $SL_2(\BZ)$. As a consequence of this fact, the subalgebras:
\begin{equation}
\label{eqn:half}
\CA^{< \frac ba} = \BK \langle E_{n,k} \rangle_{nb < ka} \subset \CA, \qquad \qquad \CA^{> \frac ba} = \BK \langle E_{n,k} \rangle_{nb > ka} \subset \CA 
\end{equation}
are all isomorphic to the shuffle algebra $\CS$ and its opposite $\CS^{\op}$, respectively (which by definition are isomorphic to $\CA^\leftarrow = \CA^{<\infty}$ and $\CA^\rightarrow = \CA^{>\infty}$, respectively). Any pair of vectors $(n,k)$, $(n',k')$ either lie on the same line passing through the origin, in which case \eqref{eqn:e triangle double} reduces to \eqref{eqn:shuf heis}, or they lie in the half-plane defined by the line of some slope $b/a \in \BQ$, in which case \eqref{eqn:e triangle double} is an equality in the subalgebra: 
$$
\CA^{<\frac ba} \cong \CS \qquad \text{or} \qquad \CA^{> \frac ba} \cong \CS^{\op}
$$
Such an equality holds due to \eqref{eqn:e triangle}. The final statement of the Theorem, concerning the fact that relations \eqref{eqn:e triangle double} generate the ideal of relations between the generators $E_{n,k} \in \CA$, was proved in \cite[Proposition 2.7]{W} following the ideas of \cite{BS}; the main idea is that one can use relations \eqref{eqn:e triangle double} to write any product of $E_{n,k}$'s as a linear combination of products of $E_{n,k}$'s in clockwise order; the fact that such clockwise products form a linear basis of $\CA$ follows by combining \cite{BS} and \cite{Shuf}. 

\end{proof}

\noindent Recall the elements $E_{d_\bullet} \in \CS$ from \eqref{eqn:shuffle gen}. Since the negative (respectively positive) half of the algebra $\CA$ is isomorphic to $\CS$ (respectively $\CS^\op$), we will write:
\begin{align}
&E_{d_\bullet} \in \CA^\leftarrow \subset \CA \label{eqn:gen left} \\ 
&F_{d_\bullet} \in \CA^\rightarrow \subset \CA \label{eqn:gen right}
\end{align} 
for the corresponding elements of $\CA$. Combining Proposition \ref{prop:gen} with \eqref{eqn:anna double} yields:
\begin{align}
&E_{d_\bullet} = \mathop{\sum^{n_1+...+n_t = n}_{\frac {k_1}{n_1} \leq ... \leq \frac {k_t}{n_t}}}^{n_i \in \BN, k_i \in \BZ} s_{k_1,...,k_t,d_\bullet}^{n_1,...,n_t}(q_1,q_2) \cdot E_{-n_1,k_1}... E_{-n_t,k_t} \label{eqn:left} \\
&F_{d_\bullet} = \mathop{\sum^{n_1+...+n_t = n}_{\frac {k_1}{n_1} \geq ... \geq \frac {k_t}{n_t}}}^{n_i \in \BN, k_i \in \BZ} s_{k_1,...,k_t,d_\bullet}^{n_1,...,n_t}(q_1,q_2) \cdot  E_{n_1,k_1}... E_{n_t,k_t} \label{eqn:right}
\end{align} 
for any $d_\bullet = (d_1,...,d_n)$, where $s_{k_1,...,k_t,d_\bullet}^{n_1,...,n_t}(q_1,q_2) \in \BK$ are uniquely determined. \\

\subsection{}
\label{sub:completion}

In this paper, we will mostly be concerned with the top half of the algebra $\CA$:
$$
\CA^\uparrow := \BK\text{-subalgebra generated by } \langle E_{n,k} \rangle^{k \in \BN}_{n \in \BZ} \subset \CA
$$
which coincides with $\CA^{<0}$ of \eqref{eqn:half}. Therefore, there exists an isomorphism:
\begin{equation}
\label{eqn:iso half}
\CA^\uparrow \cong \CS, \qquad \qquad E_{n,1} \mapsto z_1^n
\end{equation}
We will often extend the subalgebra $\CA^\uparrow$ by adding the elements $E_{n,0}$ on the $x$-axis:
$$
\CA^{\uparrow\ext} := \BK\text{-subalgebra generated by } \left \langle c^{\pm 1}, E_{n,k} \right \rangle^{k \in \BN \sqcup 0}_{n \in \BZ} \subset \CA
$$
As a consequence of \eqref{eqn:anna double}, we have: 
\begin{align}
\CA^\uparrow = &\bigoplus_{-\infty < \frac {n_1}{k_1} \leq ... \leq \frac {n_t}{k_t} < \infty }  \ring \cdot E_{n_1,k_1}... E_{n_t,k_t} \label{eqn:basis monomial 0} \\
\CA^{\uparrow\ext} \Big|_{c \mapsto q^r} = &\bigoplus_{-\infty \leq \frac {n_1}{k_1} \leq ... \leq \frac {n_t}{k_t} \leq \infty }  \ring \cdot E_{n_1,k_1}... E_{n_t,k_t}  \label{eqn:basis monomial 1}
\end{align}
(note that we do not need to specialize $c = q^r$ in \eqref{eqn:basis monomial 0} because the commutation relations \eqref{eqn:small triangle double} do not involve any powers of $c$ if the indices $(n,k)$ and $(n',k')$ are both in the upper half plane, see \cite{W}). The algebra $\CA^\uparrow$ is $\BZ \times \BN$ graded:
$$
\deg E_{n,k} = (n,k)
$$
and the graded pieces $\CA^\uparrow_{n,k}$ have infinite rank over $\BK$. However, the subspaces:
\begin{equation}
\label{eqn:basis monomial 2}
\CA^{\uparrow, \leq \mu}_{n,k} =  \mathop{\sum_{- \mu \leq \frac {n_1}{k_1} \leq ... \leq \frac {n_t}{k_t} \leq \mu}^{n_i \in \BZ, \ \sum n_i = n}}^{k_i \in \BN, \ \sum k_i = k} \ring \cdot E_{n_1,k_1}... E_{n_t,k_t} 
\end{equation}
are finite rank free $\BK$-modules. Consider the $\BK$-linear map $\CA_{n,k}^{\uparrow, \leq \mu+1} \twoheadrightarrow \CA_{n,k}^{\uparrow, \leq \mu}$ which sends every product $E_{n_1,k_1}... E_{n_t,k_t}$ of \eqref{eqn:basis monomial 0} either to 0 or to itself, and define:
$$
\wCA^\uparrow_{n,k} = \lim_{\leftarrow, \mu} \CA^{\uparrow, \leq \mu}_{n,k} \qquad \qquad \wCA^\uparrow = \bigoplus_{n \in \BZ}^{k \in \BN} \wCA^\uparrow_{n,k}
$$
In more practical terms, we may think of $\wCA^\uparrow$ as consisting of infinite $\ring$-linear combinations of basis monomials \eqref{eqn:basis monomial 0} for bounded $n_1+...+n_t$ and $k_1+...+k_t$:
\begin{equation}
\label{eqn:basis monomial 3}
\wCA^\uparrow = \underset{-\infty < \frac {n_1}{k_1} \leq ... \leq \frac {n_t}{k_t} < \infty}{\widehat{\bigoplus}}  \ring \cdot E_{n_1,k_1}... E_{n_t,k_t} 
\end{equation}
Similarly, we define $\wCA^{\uparrow\ext} \supset \wCA^\uparrow$ by allowing $k_i = 0$, and the analogue of \eqref{eqn:basis monomial 3} is:
\begin{equation}
\label{eqn:basis monomial 4}
\wCA^{\uparrow\ext} \Big|_{c \mapsto q^r} = \underset{-\infty \leq \frac {n_1}{k_1} \leq ... \leq \frac {n_t}{k_t} \leq \infty}{\widehat{\bigoplus}}   \ring \cdot E_{n_1,k_1}... E_{n_t,k_t} 
\end{equation}

\medskip

\begin{proposition}
\label{prop:completion}

$\wCA^\uparrow$ and $\wCA^{\uparrow\eext}$ are closed under multiplication, and thus algebras. \\

\end{proposition}

\begin{proof} We will prove the statement for $\wCA^\uparrow$, as the case of $\wCA^{\uparrow\ext}$ is analogous. From now on, we will consider paths $v$ in the upper half plane that start at the origin, and are built up of steps $\{(n_1,k_1),...,(n_t,k_t)\} \subset \BZ \times \BN$. Such a path is called convex if:
\begin{equation}
\label{eqn:inequality}
\frac {n_1}{k_1} \leq ... \leq \frac {n_t}{k_t}
\end{equation}
which corresponds to $(n_1,k_1) \ccur ... \ccur (n_t,k_t)$. The \textbf{size} of the path is the lattice point $(n,k)$ with $n = \sum n_i$ and $k = \sum k_i$, where the path ends. We will say that a path $v$ lies \textbf{to the left} of a path $v'$ if they have the same size, and $v' \subset v + (\BR_{\geq 0}, 0)$. Given paths $v, v'$ of sizes $(n,k), (n',k')$, their \textbf{concatenation} $v \sqcup v'$ is the path of size $(n+n',k+k')$ obtained by tracing the steps of $v$, followed by the steps of $v'$. Finally, the \textbf{convexification} $v^{\text{conv}}$ of any path $v$ refers to the convex path one obtains by rearranging the constituent steps of $v$ so that the inequality \eqref{eqn:inequality} is satisfied (the order of steps of the same slope is immaterial). \\

\noindent As in Section \ref{sec:shuf}, there is a one-to-one correspondence between paths and basis vectors of $\CA^\uparrow$, given by:
$$
v \leadsto E_v = E_{n_1,k_1}... E_{n_t,k_t}
$$
Formulas \eqref{eqn:basis monomial 0} and \eqref{eqn:basis monomial 3} say that all elements of the algebras $\CA^\uparrow$ and $\wCA^\uparrow$ are linear combinations (finite in the former case, infinite in the latter case) of the elements $E_v$ corresponding to convex paths. Let us recall from \cite{BS} that formula \eqref{eqn:e triangle double} (see also the proof of Theorem \ref{thm:comm}) implies that we can ``convexify" any path $v$, i.e. write $E_v$ as a linear combination $\sum_{v_0 \text{ convex}} c_v^{v_0} \cdot E_{v_0}$ for convex paths $v_0$. The main thing we will need to take from their argument is that the coefficients $c_v^{v_0}$ are non-zero only if the path $v_0$ lies to the left of $v$. More precisely, we will show that: \\


\begin{claim}
\label{claim:1}

For any convex paths $v,v'$ of sizes $(n,k), (n',k')$, we have:
\begin{equation}
\label{eqn:ragnar}
E_v \cdot E_{v'} = E_{v \sqcup v'} \in E_{(v \sqcup v')^\emph{conv}} + \Delta \sum^{v_0 \text{ convex path of size } (n+n',k+k')}_{\text{located to the left of }v \sqcup v'} \ring \cdot E_{v_0}
\end{equation}

\end{claim}

\medskip

\noindent Indeed, the Claim implies Proposition \ref{prop:completion}, because it establishes the following fact: given an infinite sum of $E_v$'s (resp. $E_{v'}$'s) over paths $v$ (resp. $v'$) of fixed size $(n,k)$ (resp. $(n',k')$), then any given convex path $v_0$ appears in the right-hand side of \eqref{eqn:ragnar} with non-zero coefficient only for finitely many $v$ and $v'$. This means that the product of infinite sums of $E_v$'s and $E_{v'}$'s is a well-defined infinite sum. \\

\begin{proof} \emph{of Claim \ref{claim:1}:} Let $v = \{(n_1,k_1),...,(n_t,k_t)\}$, $v' = \{(n'_1,k'_1),...,(n'_{t'},k'_{t'})\}$. If:
$$
\frac {n_t}{k_t} \leq \frac {n'_1}{k_1'}
$$
then $v \sqcup v'$ is already a convex path, and the claim holds trivially. If the opposite inequality holds, then we may apply \eqref{eqn:e triangle double} to obtain:
\begin{equation}
\label{eqn:path}
E_v E_{v'} = E_{v \backslash (n_t,k_t)} E_{n_t,k_t} E_{n_1',k_1'} E_{v' \backslash (n_1',k_1')} \in
\end{equation} 
$$
\in E_{v \backslash (n_t,k_t)} E_{n_1',k_1'} E_{n_t,k_t}  E_{v' \backslash (n_1',k_1')} + \Delta \sum_{v_0 \text{ convex path}} \ring \cdot E_{v \backslash (n_t,k_t)} E_{v_0} E_{v' \backslash (n_1',k_1')}
$$
where the sum goes over convex paths $v_0$ of size $(n_t+n_1', k_t+k_1')$ which stay to the left of the path spanned by the two vectors $(n_t,k_t), (n_1',k_1')$. Consider the paths:
\begin{align*}
&\tilde{v} \text{ obtained by concatenating } v \backslash (n_t,k_t), (n_1',k_1'), (n_t,k_t), v' \backslash (n_1',k_1') \\
&\tilde{v}_0 \text{ obtained by concatenating } v \backslash (n_t,k_t), v_0, v' \backslash (n_1',k_1')
\end{align*}
for any convex path $v_0$ that appears in the right-hand side of \eqref{eqn:path}. All such paths $\tilde{v}$ and $\tilde{v}_0$ are strictly to the left of $v \sqcup v'$. If any of these paths fails to be convex, then we choose two consecutive vectors in the given path which spoil convexity:
$$
(\tilde{n}, \tilde{k}) \text{ and } (\tilde{n}', \tilde{k}') \quad \text{such that} \quad \frac {\tilde{n}}{\tilde{k}} > \frac {\tilde{n}'}{\tilde{k}'}
$$
and repeat the argument of \eqref{eqn:path}. The reason why this recursive procedure will end after finitely many steps is that all our paths have fixed size, and the slopes of all the vectors in the paths obtained at every step cannot be greater (respectively smaller) than the greatest (respectively smallest) of the slopes of the vectors in the paths $v$ and $v'$ (cf. \eqref{eqn:e triangle double}). Finally, note that $E_{\tilde{v}}$ is the only summand in the right-hand side of \eqref{eqn:path} which does not have a prefactor $\Delta$ in front. Since $\tilde{v}$ consists of the same vectors as $v \sqcup v'$, but in some other order, at the end of the recursive procedure this summand will have transformed into $E_{(v \sqcup v')^{\text{conv}}}$ modulo $\Delta$.

\end{proof} \end{proof}


\begin{definition}
\label{def:completion}

A representation $F$ of $\CA$ is called \textbf{good} if it has a grading:
\begin{equation}
\label{eqn:good}
F = \bigoplus_{n = n_0}^\infty F_n
\end{equation}
for some $n_0 \in \BZ$, such that every $E_{n,k} \in \CA$ acts on $F$ by decreasing the degree by $n$. \\

\end{definition}

\begin{proposition}
\label{prop:completion acts}

The action of $\CA$ on any good module extends to $\wCA^\uparrow$ and $\wCA^{\uparrow\emph{ext}}$. \\

\end{proposition}

\begin{proof} Let $i_n$ and $\text{pr}_n$ denote the inclusion and projection, respectively, to the $n$-th direct summand of  \eqref{eqn:good}. Then we note that:
$$
\text{pr}_n \circ E_{n_1,k_1}...E_{n_t,k_t} \circ i_{n'} = 0
$$ 
as soon as $n_1 < n_0-n$ or $n_t > n'-n_0$. Since any element of $\wCA^\uparrow$ is a finite linear combination of monomials $E_{n_1,k_1}...E_{n_t,k_t}$ modulo those monomials which satisfy either $n_1 < n_0-n$ or $n_t > n'-n_0$ for fixed $n,n'$, the Proposition follows. 

\end{proof}

\subsection{}


The main reason we have introduced the completion $\wCA^\uparrow$ of $\CA^\uparrow$ is that it contains the following elements, studied in \cite{W}:
\begin{equation}
\label{eqn:w}
\{W_{n,k}\}_{n \in \BZ}^{k \in \BN} \in \wCA^\uparrow 
\end{equation}
given by:
\begin{equation}
\label{eqn:def w}
W_{n,k} = \mathop{\sum_{- \infty < \frac {n_1}{k_1} < ... < \frac {n_t}{k_t} < \infty}^{n_i \in \BZ, \ \sum n_i = n}}^{k_i \in \BN, \ \sum k_i = k} E_{n_1,k_1}... E_{n_t,k_t} \cdot q^{\alpha(v)} 
\end{equation}
where the integer $\alpha(v)$ associated to the convex path $v = \{(n_1,k_1),...,(n_t,k_t)\}$ is:
$$
\alpha(v) = \sum_{1 \leq i < j \leq t} k_i n_j + \sum_{i=1}^t \frac {k_in_i+k_i-n_i -\gcd(n_i,k_i)}2
$$
The relations between the elements \eqref{eqn:def w} are best phrased via the currents:
\begin{equation}
\label{eqn:def w current}
W_k(x) = \sum_{n \in \BZ} \frac {W_{n,k}}{x^n}
\end{equation}
which were shown in \cite{W} to match, up to a renormalization, the relations in the quantum deformed $W$--algebra of type $\fgl_r$ as $r \rightarrow \infty$ (which was defined in \cite{AKOS, FF}):
\begin{equation}
\label{eqn:w rel}
W_k(x) W_{k'}(y) \cdot f_{kk'}\left(\frac yx \right) - W_{k'}(y) W_k(x) \cdot f_{k'k}\left(\frac xy \right) \ = \
\end{equation}
$$
= \sum_{i=\max(0,k'-k)+1}^{k'} \delta\left(\frac {y}{xq^i} \right) \left[ W_{k'-i}\left(\frac y{q^i}\right) W_{k+i}(y) f_{k'-i,k+i} (q^i) \right] \theta(\min(i,k-k'+i)) - 
$$
$$
- \sum^{k}_{i = \max(0,k-k')+1} \delta\left(\frac {x}{yq^i} \right) \left[ W_{k-i}\left(\frac x{q^i}\right) W_{k'+i}(x) f_{k-i,k'+i} (q^i) \right] \theta(\min(i,k'-k+i))
$$
for all $k, k' \geq 0$ (one sets $W_0(x) = 1$). In formula \eqref{eqn:w rel}, we write $\delta(z) = \sum_{n \in \BZ} z^n$,
\begin{equation}
\label{eqn:def theta}
\theta(s) = \frac {\Delta}{1-q}\cdot \zeta(q)...\zeta(q^{s-1})
\end{equation}
while for all $k,k' \geq 0$, we define the following power series in $z$:
\begin{equation}
\label{eqn:f}
f_{kk'} (z) =  \exp \left[\sum_{n=1}^\infty \frac {z^n}n \cdot \frac {(1-q_1^n)(1-q_2^n)(q^{\max(0,k-k')n}-q^{kn})}{1-q^n}  \right]
\end{equation}

\subsection{}

Relation \eqref{eqn:w rel} is quadratic in the generating series $W_k(x)$, but one must carefully interpret it to obtain an infinite family of relations between the coefficients $W_{n,k}$. The idea is to equate the coefficients of $x^{-n}y^{-n'}$ in the left and right-hand sides of \eqref{eqn:w rel}, for all $n, n' \in \BZ$. In the left-hand side, this is achieved by expanding:
\begin{equation}
\label{eqn:extract 1}
W_k(x) W_{k'}(y) f_{kk'}\left(\frac yx\right) \qquad \text{in non-negative powers of } \frac yx
\end{equation}
\begin{equation}
\label{eqn:extract 2}
W_{k'}(y) W_k(x) f_{k'k}\left(\frac xy \right) \qquad \text{in non-negative powers of } \frac xy
\end{equation}
This is the only reasonable choice one can make in order for the expansion of either term to be an infinite sum of the form:
$$
\sum_{a \geq 0} \text{coefficient} \cdot W_{n-a,k}W_{n'+a,k'} \qquad \text{or} \qquad \sum_{a \geq 0} \text{coefficient} \cdot W_{n'-a,k'}W_{n+a,k}
$$
which are acceptable expressions in the completion $\wCA^\uparrow$ of $\CA^\uparrow$. What is not immediately clear, and will be discussed in the proof of Theorem \ref{thm:w algebra} below, is how to make sense of the product of $W$--algebra currents in the right-hand side of \eqref{eqn:w rel}. The only exception is when $k'=1$, in which case the relation unambiguously reads:
\begin{multline} 
W_k(x) W_1(y) \zeta \left(\frac x{y q^k} \right) - W_1(y) W_k(x) \zeta \left(\frac y{qx} \right) = \\ = \frac {\Delta}{1-q} \left[ \delta \left( \frac y{xq} \right) W_{k+1}(y) - \delta \left( \frac x{yq^k} \right) W_{k+1}(x) \right] \label{eqn:series}
\end{multline} 
Extracting the coefficients of $x^{-n}y^{-n'}$ according to the rules \eqref{eqn:extract 1} and \eqref{eqn:extract 2} yields:
$$
[W_{n,k}, W_{n',1}] + \Delta \sum_{a=1}^{\infty} \frac {1-q^a}{1-q}  \left( q^{a(k-1)} W_{n-a,k}W_{n'+a,1} - W_{n'-a,1} W_{n+a,k} \right) =
$$
\begin{equation}
\label{eqn:example}
= \Delta \frac {q^{-n} - q^{-kn'}}{1-q} \cdot W_{n+n',k+1}
\end{equation}
What is surprising about formula \eqref{eqn:example} is that all the coefficients lie in $\BK$ instead of $\BF$, even though $\theta(s)$ could a priori have produced poles of the form $1-q^n$ in \eqref{eqn:w rel}. In fact, this is a general phenomenon, as evidenced by the result below: \\

\begin{theorem}
\label{thm:w algebra}

Relations \eqref{eqn:w rel} are equivalent with the following family of equalities, which hold for all $k,k' > 0$ and $n,n' \in \BZ$:
\begin{equation}
\label{eqn:comm w}
[W_{n,k}, W_{n',k'}] = 
\end{equation}
$$
= \Delta \left( \mathop{\sum_{0 \leq \min(l,l') \leq \min(k,k') \text{ and } \frac ml \leq \frac {m'}{l'}, \text{ as well as}}}_{\frac ml < \min(\frac nk, \frac {n'}{k'}) \leq \max(\frac nk, \frac {n'}{k'}) < \frac {m'}{l'} \text{ if } \min(l,l') = \min(k,k')} c_{n,n',k,k'}^{m,m',l,l'} (q_1,q_2) \cdot W_{m,l} W_{m',l'} \right)
$$
for certain $c_{n,n',k,k'}^{m,m',l,l'} (q_1,q_2) \in \ring$ which we will compute algorithmically. \\



\end{theorem}

\noindent Note that one can only have $m+m' = n+n'$ and $k+k' = l+l'$ in the right-hand side of \eqref{eqn:comm w} for degree reasons, and that $W_{m,0} = \delta_m^0$ and $W_{m',0} = \delta_{m'}^0$ are allowed. \\

\subsection{} Theorem \ref{thm:w algebra} will be proved at the end of the present Section. Let us define:
\begin{equation}
\label{eqn:def a infty}
\CA_\infty = \ring \Big \langle W_{n,k} \Big \rangle^{k \in \BN}_{n\in \BZ} \Big/ \text{relations \eqref{eqn:comm w}}
\end{equation}
and let us note that Theorem \ref{thm:w algebra} states that there is a well-defined homomorphism $\CA_\infty \rightarrow \wCA^\uparrow$ given by \eqref{eqn:def w}. After tensoring with $\field$, this embedding is completely determined by sending $W_{n,1} \mapsto E_{n,1}$, because \eqref{eqn:example} implies that any $W_{n,k}$ can be obtained as a sum of products of $W_{n,1}$, upon inverting $\Delta$ and $[q]_s$ for all $s \in \BN$. \\

\begin{proposition}
\label{prop:w and heis}

The generators of $\CA_\infty \rightarrow \wCA^\uparrow$ interact with:
\begin{equation}
\label{eqn:gardel}
p_{-n} := P_{-n,0} \in \CA^{\uparrow\emph{ext}} \quad \text{and} \quad p_n := \left(\frac cq\right)^n P_{n,0} \in \CA^{\uparrow\emph{ext}}
\end{equation}
by the formulas:
\begin{equation}
\label{eqn:bonus w rel 1}
[W_k(x), p_{-n}] =  - \frac {(1-q_1^n)(1-q_2^n)(1 - q^{kn})}{1-q^n}  \cdot \frac 1{x^n} W_k(x) 
\end{equation}
\begin{equation}
\label{eqn:bonus w rel 2}
[W_k(x), p_n] = \frac {(1-q_1^n)(1-q_2^n)(q^{-kn} - 1)c^n}{1-q^n}  \cdot x^{n} W_k(x)
\end{equation}
\begin{equation}
\label{eqn:bonus w rel 3}
[p_{-n}, p_n] = n (1-q_1^n)(1-q_2^n)\frac {1-c^n}{1-q^n}
\end{equation}
for all $n \in \BN$. \\

\end{proposition}

\noindent Define:
\begin{equation}
\label{eqn:def a ext infty}
\CA^{\ext}_\infty = \ring \Big \langle W_{n,k},p_n \Big \rangle^{k \in \BN}_{n\in \BZ} \Big/ \text{relations \eqref{eqn:comm w},\eqref{eqn:bonus w rel 1},\eqref{eqn:bonus w rel 2},\eqref{eqn:bonus w rel 3}}
\end{equation}
Throughout this paper, we will always specialize the central charge $c$ to $q^r$ for some $r \in \BN$, so the structure constants of the algebra $\CA^{\ext}_\infty$ will all lie in $\ring$ (if $c \neq q^r$, then \eqref{eqn:bonus w rel 3} would contradict this fact). By Proposition \ref{prop:w and heis}, the homomorphism: 
$$
\CA_\infty \rightarrow \wCA^\uparrow \qquad \text{extends to} \qquad \CA_\infty^{\ext} \rightarrow \wCA^{\uparrow,\ext}
$$

\medskip

\begin{proof} \emph{of Proposition \ref{prop:w and heis}:} We will prove \eqref{eqn:bonus w rel 1}, as \eqref{eqn:bonus w rel 2} is analogous and \eqref{eqn:bonus w rel 3} is a trivial application of \eqref{eqn:shuf heis 0}. We will do so by induction on $k$ (the case when $k=1$ is simply \eqref{eqn:small triangle double}), and start by commuting relation \eqref{eqn:series} with $p_{-n}$:
\begin{align*} &\frac {\Delta}{1-q} \left[ \delta \left( \frac y{xq} \right)[W_{k+1}(y), p_{-n}] - \delta \left( \frac x{yq^k} \right) [W_{k+1}(x), p_{-n}] \right] = \\ &= 
[W_k(x) W_1(y), p_{-n}] \zeta \left( \frac x{yq^k} \right) - [W_1(y) W_k(x), p_{-n}] \zeta \left( \frac y{xq} \right) \stackrel{\text{Leibniz rule}}= \\ &= W_k(x) [W_1(y), p_{-n}] \zeta \left(\frac x{yq^k} \right) + [W_k(x), p_{-n}]W_1(y) \zeta \left(\frac x{yq^k} \right) \\ &- W_1(y) [W_k(x), p_{-n}] \zeta \left( \frac y{xq} \right) - [W_1(y), p_{-n}]W_k(x) \zeta \left( \frac y{xq} \right) \stackrel{\text{induction hypothesis of \eqref{eqn:bonus w rel 1}}}= \\ & -(1-q_1^{n})(1-q_2^{n}) \left( \frac 1{y^{n}} + \frac {1-q^{kn}}{(1-q^{n})x^{n}}\right) \left[ W_k(x) W_1(y)\zeta \left(\frac x{yq^k}\right) - W_1(y) W_k(x) \zeta \left(\frac y{xq}\right)\right] \\ &= -\frac {\Delta(1-q_1^{n})(1-q_2^{n})}{1-q} \left( \frac 1{y^{n}} + \frac {1-q^{kn}}{(1-q^{n})x^{n}}\right) \left[ \delta \left( \frac y{xq} \right) W_{k+1}(y) - \delta \left( \frac x{yq^k} \right) W_{k+1}(x) \right]
\end{align*}
If we multiply both sides of the equation above with $1 - yq^k/x$, the second $\delta$ function vanishes in both sides of the equation above, and we are left with:
\begin{multline*} 
\frac {\Delta(1-q^{k+1})}{1-q} \delta \left( \frac y{xq} \right) [W_{k+1}(y), p_{-n}] = \\ =  - \frac {\Delta(1-q^{k+1})}{1-q} \frac {(1-q_1^{n})(1-q_2^{n})}{y^{n}} \left[1 + \frac {1-q^{kn}}{(1-q^{n})q^{-n}} \right] \delta \left( \frac y{xq} \right) W_{k+1}(y)
\end{multline*}
Extracting the coefficient of $x^0$ from the equation above yields precisely \eqref{eqn:bonus w rel 1} with $k$ replaced by $k+1$, thus establishing the induction step.

\end{proof}

\subsection{} Finally, we realize the deformed $W$--algebra of type $\fgl_r$ as a quotient of $\CA_\infty$. \\

\begin{definition}
\label{def:w algebra r}

For arbitrary $r \in \BN$, define:
\begin{equation}
\label{eqn:def a r}
\CA_r = \CA_\infty \Big / \text{ relation \eqref{eqn:w vanish}}
\end{equation}
\begin{equation}
\label{eqn:def a ext r}
\CA^{\eext}_r = \CA^{\eext}_\infty \Big / \text{ relations \eqref{eqn:w vanish}, \eqref{eqn:w exp} and } c = q^r
\end{equation}
where:
\begin{equation}
\label{eqn:w vanish}
W_k(x) = 0, \qquad \forall k>r 
\end{equation}
\begin{equation}
\label{eqn:w exp}
W_r(x) = u \exp \left[ \sum_{n=1}^\infty \frac {p_{-n}}{nx^{-n}} \right] \exp \left[ \sum_{n=1}^\infty \frac {p_n}{nx^n} \right] 
\end{equation}
The parameter $u$ in \eqref{eqn:w exp} will not be crucial (it will be identified with a line bundle in the next Section) and so we will not mention it explicitly in our notation. \\

\end{definition}

\noindent The algebra $\CA_r$ was shown to be isomorphic to the deformed $W$--algebra of type $\fgl_r$ in \cite{W}. Definition \ref{def:w algebra r} explains the appeal of presenting this algebra in terms of the generators $W_{n,k}$: to show that $\CA_r$ acts on a certain module, it is enough to show that the module is a good representation of $\CA$, and then check that relations \eqref{eqn:w vanish} hold (if one wants an action of $\CA_r^\ext$, then one also has to check \eqref{eqn:w exp}). \\

\begin{proof} \emph{of Theorem \ref{thm:w algebra}:} By \eqref{eqn:def w}, the element $W_{n,k} \in \wCA^\uparrow$ is a certain infinite $\ring$-linear combination of the elements:
$$
E_v = E_{n_1,k_1}... E_{n_t,k_t}
$$
as $v = \{(n_1,k_1),...,(n_t,k_t)\}$ runs over convex paths of size $(n,k)$. Moreover, the leading order coefficient, i.e. the coefficient of $E_{n,k}$ itself, is a power of $q$. By successive applications of Claim \ref{claim:1}, this implies that there exists $* \in \BZ$ such that:
\begin{equation}
\label{eqn:mercan}
W_v := W_{n_1,k_1}...W_{n_t,k_t} = q^{*} E_v + \Delta \sum^{\text{convex paths }v'}_{\text{to the left of }v} E_{v'} \cdot \text{element of }\ring
\end{equation}
Therefore, the $\BK$-basis $\{W_v\}_{v\text{ convex}}$ is upper triangular in terms of $\{E_v\}_{v\text{ convex}}$, with respect to the partial ordering on convex paths of the same size, where one path is to the left of another. Since $\{E_v\}_{v\text{ convex}}$ form a basis of $\wCA^\uparrow$ over $\ring$, therefore:
\begin{equation}
\label{eqn:statement}
\wCA^\uparrow = \mathop{\underset{k \in \BN}{\widehat{\bigoplus}}}_{n \in \BZ} \left[ \mathop{\bigoplus_{- \infty < \frac {n_1}{k_1} \leq ... \leq \frac {n_t}{k_t} < \infty}^{n_i \in \BZ, \sum n_i = n}}^{k_i \in \BN, \sum k_i = k} \ring \cdot W_{n_1,k_1}... W_{n_t,k_t} \right]
\end{equation}
(if $\frac {n_i}{k_i} = \frac {n_{i+1}}{k_{i+1}}$ for some $i$, then we make the convention that $k_i \leq k_{i+1}$ in the sum above, and henceforth). Combining this statement with Claim \ref{claim:1}, we infer that:
\begin{equation}
\label{eqn:comm w 1}
[W_{n,k}, W_{n',k'}] = \Delta \mathop{\sum_{ - \infty < \frac {n_1}{k_1} \leq ... \leq \frac {n_t}{k_t} < \infty }^{n_i \in \BZ, \ \sum n_i = n+n'}}^{k_i \in \BN, \ \sum k_i = k+k'} r_{n,n_1,...,n_t,n'}^{k,k_1,...,k_t,k'}(q_1,q_2) W_{n_1,k_1}... W_{n_t,k_t} 
\end{equation}
for certain Laurent polynomials $r_{n,n_1,...,n_t,n'}^{k,k_1,...,k_t,k'}(q_1,q_2) \in \ring$. To prove \eqref{eqn:comm w}, we must show that the only terms in the right-hand side of \eqref{eqn:comm w 1} with non-zero coefficient are products of exactly two generators $W_{n_1,k_1}W_{n_2,k_2}$, with $\min(k_1,k_2) \leq \min(k,k')$ (we allow $\min(k_1,k_2) = 0$ in the preceding statement, recalling that $W_{m,0} = \delta_m^0$). To do so, we must properly interpret \eqref{eqn:w rel}. It was shown in \cite[(2.57)]{W} that:
$$
R_{kk'}(x,y) = W_k(x) W_{k'}(y) f_{kk'} \left(\frac yx \right)
$$
is a linear combination of the basis elements $\{E_v\}_{v\text{ convex path}}$, whose coefficients are rational functions in $x$ and $y$. These rational functions have poles at:
\begin{equation}
\label{eqn:poles 1}
y - x q^{-i} \qquad \text{for } i \in \{ \max(0,k-k')+1,...,k\}\\
\end{equation}
\begin{equation}
\label{eqn:poles 2}
y - x q^i \qquad \ \ \text{for } i \in \{ \max(0,k'-k)+1,...,k'\}
\end{equation}
and satisfy the following symmetry relation:
\begin{equation}
\label{eqn:symmetry}
R_{kk'} (x,y) = R_{k'k} (y,x)
\end{equation}
and the evaluation properties:
\begin{equation}
\label{eqn:residue 1}
\underset{y = x q^{-i}}{\res} \frac {R_{kk'} (x,y)}y = R_{k-i,k'+i} (xq^{-i},x) \theta(\min(i,k'-k+i)) \quad \text{for } i \text{ as in \eqref{eqn:poles 1}}
\end{equation}
\begin{equation}
\label{eqn:residue 2}
\underset{y = x q^i}{\res} \frac {R_{kk'} (x,y)}y = - R_{k'-i,k+i}(x,xq^i) \theta(\min(i,k-k'+i)) \quad \text{for } i \text{ as in \eqref{eqn:poles 2}} 
\end{equation}
Combining \eqref{eqn:symmetry}, \eqref{eqn:residue 1} and \eqref{eqn:residue 2}, we interpret formula \eqref{eqn:w rel} as saying that:
$$
\Big[ R_{kk'}(x,y) \text{ expanded in } |y| \ll |x| \Big] - \Big[ R_{kk'}(x,y) \text{ expanded in } |y| \gg |x| \Big] = 
$$
\begin{equation}
\label{eqn:residue equality}
= - \sum_{\alpha \notin \{0, \infty\}} \delta \left( \frac y{x\alpha} \right) \underset{y = x \alpha}{\res} \frac {R_{kk'}(x,y)}y
\end{equation}
which is simply a reformulation of the residue theorem for rational functions. Note that the interpretation given above was known to \cite{AKOS} (see \cite{O} for a review), although in the somewhat different context of operator product expansions corresponding to the free field realization of the deformed $W$--algebra. \\


\noindent Formula \eqref{eqn:residue equality} shows how to make sense of the relation \eqref{eqn:w rel}, and how to deduce it from \eqref{eqn:symmetry}, \eqref{eqn:residue 1} and \eqref{eqn:residue 2}. However, the down-side is that the right-hand side of \eqref{eqn:residue equality} is not readily expressed in terms of the generators $W_{n,k}$ which appear once we expand the left-hand side (unless we are in the simple situation of \eqref{eqn:example}). To fix this issue, let us recall the normal-ordered integrals considered in \cite[\S 2.2]{O}:
\begin{equation}
\label{eqn:normal order 1}
:R_{kk'}(x,y): \ = \ :W_k(x)W_{k'}(y) f_{kk'} \left(\frac yx \right): \ = 
\end{equation}
$$
= \oint_{|z| \gg |x|,|y|} W_{k'}(z)W_{k}(x) f_{k'k} \left(\frac xz \right) \frac {Dz}{1-\frac yz} -  \oint_{|z| \ll |x|,|y|} W_k(x)W_{k'}(z) f_{kk'} \left(\frac zx \right) \frac {Dz}{1-\frac yz}
$$
where $Dz = \frac {dz}{2\pi i z}$. On one hand, by explicitly computing the expansions, we obtain:
\begin{equation}
\label{eqn:normal order 2}
:R_{kk'}(x,y): \ = 
\end{equation}
$$
= \sum_{a,b \in \BN \sqcup \{0\}}^{c \in \BZ} \left[ x^{a+c} y^b f_{k'k}^a W_{-a-b,k'} W_{-c,k} + x^{-a-c} y^{-b-1} f_{kk'}^a W_{c,k} W_{a+b+1,k'} \right]
$$
where we consider the power series expansion of \eqref{eqn:f}:
$$
f_{kk'}(x) = \sum_{a=0}^\infty f_{kk'}^a x^a \in 1 +  x \Delta \ring[[x]]
$$
On the other hand, as observed in \eqref{eqn:symmetry}, the two integrands in equation \eqref{eqn:normal order 1} represent the same rational function. Therefore, the difference between the two integrals is given by the sum of the residues in the variable $z$, similar to \eqref{eqn:residue equality}:
$$
:R_{kk'}(x,y): \ = R_{kk'}(x,y) + \sum_{\alpha \notin \{0, \infty\}} \underset{z = x \alpha}{\res} \frac {R_{kk'}(x,z)}{x\alpha - y}
$$
(the first term in the right-hand side is the residue at $z=y$). These residues are prescribed by \eqref{eqn:residue 1} and \eqref{eqn:residue 2}, so we obtain:
\begin{equation}
\label{eqn:normal order 3}
R_{kk'}(x,y) = \ :R_{kk'}(x,y): +
\end{equation}
\begin{align*}
&+ \sum_{i = \max(0,k-k')+1}^{k} \frac {R_{k-i,k'+i}(xq^{-i},x)}{1-\frac {yq^i}{x}} \theta(\min(i,k'-k+i)) - \\
&- \sum_{i = \max(0,k'-k)+1}^{k'} \frac {R_{k'-i,k+i}(x,xq^i)}{1-\frac y{xq^i}} \theta(\min(i,k-k'+i))
\end{align*}
One can iterate relation \eqref{eqn:normal order 3} to express the right-hand side as:
\begin{equation}
\label{eqn:normal order 4}
R_{kk'}(x,y) =  \ :R_{kk'}(x,y): + \sum^{0 \leq l < \min(k,k')}_{k+k' = l+l', a,b\in \BZ} \frac {:R_{ll'}(xq^a, xq^b): g_{ll'}^{ab}(q_1,q_2)}{\prod_{\text{various }c} \left(1-\frac {y}{xq^c} \right)}
\end{equation}
for some $g_{ll'}^{ab}(q_1,q_2) \in \field$. We conjecture that $g_{ll'}^{ab}(q_1,q_2) \in \ring$ for all $a,b,l,l'$, but we will not prove this result. Instead, we observe that relation \eqref{eqn:normal order 4} gives the correct interpretation of \eqref{eqn:w rel}. Indeed, expanding \eqref{eqn:normal order 4} for $|y| \ll |x|$ (respectively $|y| \gg |x|$ gives us the first (respectively second) term in the left-hand side of \eqref{eqn:w rel}, and taking the difference of these two expansions gives us:
$$
W_k(x) W_{k'}(y) f_{kk'} \left(\frac yx \right) - W_{k'}(y) W_k(x)  f_{k'k} \left(\frac xy \right) = 
$$
$$
= \sum^{0 \leq l < \min(k,k')}_{k+k' = l+l', a,b\in \BZ} :R_{ll'}(xq^a, xq^b): \left[ \sum_c \delta \left( \frac {y}{xq^c} \right) \frac 1{\prod_{c'\neq c} (1- q^{c-c'})} \right] g_{ll'}^{ab}(q_1,q_2)
$$
The normal-ordered products in the right-hand side are quadratic expressions in the $W$--algebra generators according to \eqref{eqn:normal order 2}, and so taking the coefficient of $x^{-n} y^{-n'}$ for any $n,n' \in \BZ$ in the above expression gives us:
\begin{equation}
\label{eqn:normal order 5}
[W_{n,k}, W_{n',k'}] + \sum_{a=1}^\infty \left[ W_{n-a,k}W_{n'+a,k'} f_{kk'}^a - W_{n'-a,k'}W_{n+a,k} f_{k'k}^a \right] = 
\end{equation}
$$
= \sum^{0 \leq \min(l,l') < \min(k,k')}_{k+k' = l+l'} \left( \sum_{d \leq 0} W_{d,l'} W_{n+n'-d,l} \cdot \text{coeff} + \sum_{d > 0} W_{n+n'-d,l} W_{d,l'} \cdot \text{coeff} \right)
$$
where the coefficients denoted by ``coeff" lie in $\field$. For any fixed $n+n'$ and $k+k'$, iterating formula \eqref{eqn:normal order 5} allows us to prove \eqref{eqn:comm w} by induction on $\min(k,k')$, but only yields the fact that the coefficients $c_{n,n',k,k'}^{m,m',l,l'}$ lie in $\BF$. However, \eqref{eqn:comm w 1} and the fact that $\{W_v\}_{v \text{ convex}}$ form a basis of $\wCA^\uparrow$ imply that $c_{n,n',k,k'}^{m,m',l,l'}$ lie in $\BK$, as required. 

\end{proof}

\section{The moduli space of sheaves}
\label{sec:mod}

\medskip

\subsection{} 

Having constructed shuffle and deformed $W$--algebras in the previous Sections, let us now construct the modules on which they are expected to act. Consider a smooth projective surface $S$ over an algebraically closed field (denoted by $\BC$) and an ample divisor $H \subset S$. The Hilbert polynomial of a coherent sheaf $\CF$ on $S$ is:
$$
P_\CF(n) := \chi(S, \CF \otimes \CO(nH)) = a n^2 + b n + c
$$
where $a,b,c$ are rational numbers that one can compute from the Grothendieck-Hirzebruch-Riemann-Roch theorem. One can find formulas for these numbers in the Appendix to \cite{Univ}, but the only thing we will need in the present paper is that they can be expressed in terms of $S,H$ and the rank and Chern classes $r, c_1, c_2$ of $\CF$. The reduced Hilbert polynomial is defined as:
$$
p_\CF(n) = \frac {P_\CF(n)}{a}
$$ 
A torsion free coherent sheaf $\CF$ on $S$ is called stable (respectively semistable) if for all proper subsheaves $\CG \subset \CF$ we have:
$$
p_\CG(n) < p_\CF(n) \qquad (\text{respectively } p_\CG(n) \leq p_\CF(n))
$$
for $n \gg 0$. Since the reduced Hilbert polynomials are monic and quadratic, stability (respectively semistability) is determined by checking the respective inequalities for the linear term and constant term coefficients. Note that stability depends on the polarization $H$, but we will fix a choice throughout this paper. \\

\begin{definition} (see \cite{HL}): Let $\CM_{(r,c_1,c_2)}$ denote the quasiprojective variety which corepresents the moduli functor of stable sheaves on $S$ with the invariants $r,c_1,c_2$. \\

\end{definition}

\subsection{} We will denote the moduli space by $\CM$ when the particular invariants will not be important to us. We impose the following assumption throughout this paper:
\begin{equation}
\label{eqn:assumption a}
\textbf{Assumption A:} \qquad \gcd(r, c_1 \cdot H) = 1
\end{equation}
This assumption has two important consequences: firstly, any semistable sheaf is stable. Secondly, there exists a universal sheaf:
\begin{equation}
\label{eqn:universal}
\xymatrix{
\CU \ar@{.>}[d] \\
\CM \times S}
\end{equation}
which is flat over $\CM$, and its fiber over any closed point $\{\CF\} \times S$ is isomorphic to $\CF$ as a coherent sheaf over $S$. This leads to the following fact (see \cite{HL}): \\

\begin{proposition} 

Under Assumption A, $\CM_{(r,c_1,c_2)}$ is a projective variety, which also represents the moduli functor of stable sheaves on $S$ with the invariants $r,c_1,c_2$. \\

\end{proposition}

\begin{remark}
\label{rem:choice}

Note that the universal sheaf $\CU$ is only determined up to tensoring with a line bundle pulled back from $\CM$. This means that one has the freedom to choose such a line bundle on any component $\CM_{(r,c_1,c_2)}$ of the moduli space, and this represents the ambiguity in choosing the universal sheaf on the whole of $\CM$. Throughout the present paper, we will fix such a choice, requiring only that the chosen universal sheaves on $\CM_{(r,c_1,c_2)}$ and $\CM_{(r,c_1,c_2+1)}$ be compatible with each other as described in the Appendix to \cite{Univ}. This compatibility is precisely what allows us to define Hecke correspondences in the following Subsection. \\

\end{remark}

\begin{remark}
\label{rem:twisted}

Since many of the constructions in the present paper are local on the moduli space of stable sheaves, one can replace the category of coherent sheaves by that of twisted coherent sheaves (see \cite{C} for an overview). In this case, there exists a twisted universal sheaf on $\CM \times S$, even without imposing Asssumption A. The interested reader may try to generalize the results in the present paper to that setting, but we do not expect any fundamentally new features to arise. \\

\end{remark}

\noindent Because the universal sheaf $\CU$ is flat over $\CM$, it inherits certain properties from the stable sheaves it parametrizes, such as having projective dimension 1 (indeed, any torsion free sheaf on a smooth projective surface has projective dimension 1, see Example 1.1.16 of \cite{HL}): \\

\begin{proposition}
\label{prop:length 1}

There exists a short exact sequence:
\begin{equation}
\label{eqn:length 1}
0 \rightarrow \CW \rightarrow \CV \rightarrow \CU \rightarrow 0
\end{equation}
with $\CW$ and $\CV$ locally free sheaves on $\CM \times S$ (see \cite[Proposition 2.2]{Univ} for a proof). \\

\end{proposition}


\subsection{}
\label{sub:z1}

Consider the nested moduli space:
\begin{equation}
\label{eqn:z}
\bfZ_1 = \Big\{ (\CF,\CF') \text{ s.t. } \CF \supset_x \CF' \text{ for some } x\in S \Big\} \subset \CM \times \CM'
\end{equation}
where the notation:
$$
\CF \supset_x \CF' \text{ means that } \CF \supset \CF' \text{ and } \CF/\CF' \cong \BC_x
$$
Consider the natural projection maps:
\begin{equation}
\label{eqn:projection maps}
\xymatrix{
& \bfZ_1 \ar[ld]_{p_-} \ar[d]^{p_S} \ar[rd]^{p_+} & \\
\CM & S & \CM'} \qquad \xymatrix{
& (\CF \supset_x \CF') \ar[ld]_{p_-} \ar[d]^{p_S} \ar[rd]^{p_+} & \\
\CF & x & \CF'} 
\end{equation}
Moreover, we have the tautological line bundle on $\bfZ_1$:
\begin{equation}
\label{eqn:tautological line}
\xymatrix{
\CL \ar@{.>}[d] \\
\bfZ_1} \qquad \qquad \CL|_{(\CF,\CF')} = \CF_x / \CF'_x = \Gamma(S,\CF/\CF')
\end{equation}
The space $\bfZ_1$ was denoted by $\bfZ$ in \cite{Univ}, but in the present paper we have changed the notation because we will shortly encounter a related space, denoted by $\bfZ_2$. \\


\noindent Remark \ref{rem:choice} refers to fixing choices of the universal sheaves $\CU$ on all components of the moduli space $\CM$, so that these choices are compatible with $\bfZ_1$. Specifically, this entails the existence of a short exact sequence:
\begin{equation}
\label{eqn:ses}
0 \rightarrow (p_+ \times \text{Id})^* \left( \CU \right) \rightarrow (p_- \times \text{Id})^* \left( \CU \right) \rightarrow \Gamma_*(\CL) \rightarrow 0
\end{equation}
on $\bfZ_1 \times S$, where $\bfZ_1 \xrightarrow{\Gamma} \bfZ_1 \times S$ denotes the graph of $p_S$. We will often abuse notation and denote the first two terms in \eqref{eqn:ses} by $\CU'$ and $\CU$, thought of as sheaves on $\bfZ_1 \times S$. \\

\subsection{} Recall that a dg scheme $X$ (defined over a base scheme $Y$, which will be clear from context) is locally defined by a dg algebra:
$$
\CO_X = \left[ ... \xrightarrow{d} \CO_X^{(2)} \xrightarrow{d} \CO_X^{(1)} \xrightarrow{d} \CO_X^{(0)} \right]
$$
The scheme $\bar{X}$ whose coordinate ring is locally given by:
$$
\CO_{\bar{X}} = \CO_X^{(0)}/\text{Im }d
$$
will be called the \textbf{support} of the dg scheme $X$. This means that there exists a map of dg schemes $\bar{X} \rightarrow X$, and we will therefore say that $X$ is supported on $\bar{X}$. \\

\begin{example}
\label{ex:zero}

If $V$ is a locally free sheaf on the scheme $Y$, and we are given a section $\sigma \in \Gamma(V)$, then the derived zero locus of $\sigma$ is defined as the dg scheme:
$$
Z = \emph{Spec}_Y \left[ ... \xrightarrow{\sigma^\vee} \wedge^2(V^\vee) \xrightarrow{\sigma^\vee} V^\vee \xrightarrow{\sigma^\vee} \CO_Y \right] 
$$
Alternatively, we will say that $Z$ is cut out (in a derived sense) by the vanishing of $\sigma$. Meanwhile, the scheme-theoretic (or non-derived) zero locus of $\sigma$ is the scheme:
$$
\bar{Z} = \emph{Spec}_Y \left( \CO_Y/\emph{Im }\sigma^\vee \right)
$$
Clearly, $Z$ is supported on $\bar{Z}$. \\

\end{example}

\subsection{} We will now define a dg scheme $\fZ_1$ supported on the scheme $\bfZ_1$. \\


\begin{definition}
\label{def:desc}

Let $\emph{Sym}$ denote the symmetric algebra. Consider the dg scheme:
\begin{equation}
\label{eqn:def z1 proj}
\fZ_1 = \BP_{\CM \times S}(\CU) = \emph{Proj} \left( \emph{Sym}^\bullet_{\CM \times S} (\CU) \right)
\end{equation}
Since $\CU$ is not a vector bundle, but a coherent sheaf of projective dimension 1 as in Proposition \ref{prop:length 1}, relation \eqref{eqn:def z1 proj} is taken to mean that:
\begin{equation}
\label{eqn:desc 1}
\xymatrix{\fZ_1 \ar[dr] \ar@{^{(}->}[r]^-{\iota_-} 
& \BP_{\CM \times S}(\CV) \ar@{->>}[d]^{\rho_-} \\
& \CM \times S}
\end{equation}
The vertical arrow is a projective bundle, and the horizontal arrow is the dg subscheme cut out by the vanishing of the following map of vector bundles on $\BP_{\CM \times S}(\CV)$:
\begin{equation}
\label{eqn:koszul 1}
\sigma^\vee : \rho_-^*(\CW) \rightarrow \rho_-^*(\CV) \twoheadrightarrow \CO(1)
\end{equation}
Define $\CL = \iota_-^*(\CO(1))$, where $\CO(1)$ is the tautological line bundle on $\BP_{\CM \times S}(\CV)$. \\

\end{definition}

\noindent Similarly, the non-derived zero locus of the section \eqref{eqn:koszul 1} is the scheme $\bfZ_1$, which will therefore be the support of $\fZ_1$. The line bundles denoted by $\CL$ in \eqref{eqn:tautological line} and Definition \ref{def:desc} are compatible under the natural map $\bfZ_1 \rightarrow \fZ_1$. \\


\subsection{} Define the map $p_- \times p_S : \fZ_1 \rightarrow \CM \times S$ as the diagonal arrow in \eqref{eqn:desc 1}. The following Proposition claims that there also exists a map $p_+ \times p_S : \fZ_1 \rightarrow \CM' \times S$. In order to define it, let us write $\CU' = \CV' / \CW'$ for the analogue of \eqref{eqn:length 1} on $\CM' \times S$. We will write $\omega_S$ both for the canonical line bundle on $S$, and for its pull-back to $\CM' \times S$. \\

\begin{proposition}
\label{prop:desc}

We have a diagram:
\begin{equation}
\label{eqn:desc 2}
\xymatrix{\fZ_1 \ar[rd] \ar@{^{(}->}[r]^-{\iota_+} 
& \BP_{\CM' \times S} \left(\CW'^\vee\otimes \omega_S \right) \ar@{->>}[d]^{\rho_+} \\
& \CM' \times S}
\end{equation}
where the vertical arrow is a projective bundle, and the horizontal arrow $\iota_+$ is the dg subscheme cut out by the vanishing of the following map of vector bundles:
\begin{equation}
\label{eqn:koszul 2}
\sigma' : \CO(-1) \hookrightarrow \rho_+^*(\CW' \otimes \omega_S^{-1}) \longrightarrow \rho^*_+(\CV' \otimes \omega_S^{-1})
\end{equation}
on $\BP_{\CM' \times S}(\CW'^\vee \otimes \omega_S)$. Moreover, we have $\CL \cong \iota_+^*(\CO(-1))$ in the notation of \eqref{eqn:desc 2}. \\

\end{proposition}

\noindent We refer the reader to \cite[Proposition 2.8]{Univ} for the proof of Proposition \ref{prop:desc}, which states that the derived zero loci of the sections \eqref{eqn:koszul 1} and \eqref{eqn:koszul 2} are isomorphic dg schemes. Combining Definition \ref{def:desc} and Proposition \ref{prop:desc}, we conclude that there exist maps as in the diagram below, which are compatible with those of \eqref{eqn:projection maps} under the support map $\bfZ_1 \rightarrow \fZ_1$ (we will use the same symbols, namely $p_-, p_S,p_+$, for these maps):
\begin{equation}
\label{eqn:projection maps dg}
\xymatrix{
& \fZ_1 \ar[ld]_{p_-} \ar[d]^{p_S} \ar[rd]^{p_+} & \\
\CM & S & \CM'}
\end{equation}

\medskip

\subsection{} 
\label{sub:fine}

Let us consider the following closed subscheme of $\CM \times \CM' \times \CM''$:
\begin{equation}
\label{eqn:z2}
\bfZ_2 = \Big\{ (\CF,\CF',\CF'') \text{ s.t. } \CF \supset_{x_1} \CF' \supset_{x_2} \CF'' \text{ for some } x_1,x_2 \in S \Big\} 
\end{equation}
Moreover, we define the dg scheme $\fZ_2$ that completes the derived fiber square below:
\begin{equation}
\label{eqn:fiber 0}
\xymatrix{
& & S \times S & & \\
& & \fZ_2 \ar[rd]^{\pi_+} \ar[ld]_{\pi_-} \ar[u]^{p_S^1 \times p_S^2} & & \\
& \fZ_1 \ar[ld]_{p_-} \ar[rd]^{p_+} & & \fZ_1 \ar[ld]_{p_-} \ar[rd]^{p_+} & \\
\CM & & \CM' & & \CM'' } 
\end{equation}
Since $\fZ_1$ is supported on $\bfZ_1$, it is easy to see that $\fZ_2$ is supported on $\bfZ_2$. The map ${p_S^1 \times p_S^2}$ in \eqref{eqn:fiber 0} records the support points $(x_1,x_2)$ of the quotients in \eqref{eqn:z2}. Consider the derived restriction of $\fZ_2$ to the diagonal $\Delta : S \hookrightarrow S \times S$, i.e. the derived fiber product of ${p_S^1 \times p_S^2}$ and $\Delta$. Definition \ref{def:desc} and Proposition \ref{prop:desc} imply:
\begin{equation}
\label{eqn:z2 fib diag}
\xymatrix{\fZ_2|_\Delta \ar[rd] \ar@{^{(}->}[r]
& \BP_{\CM' \times S} \left(\CV'\right) \times_{\CM' \times S} \BP_{\CM' \times S} \left( {\CW'}^\vee \otimes \omega_S \right) \ar@{->>}[d]^\rho \\
& \CM' \times S}
\end{equation}
where the embedding is defined as the derived zero locus of the section:
\begin{equation}
\label{eqn:section}
\CO \xrightarrow{\sigma \oplus \sigma'} \rho^*({\CW'}^\vee) \otimes \CO_1(1) \bigoplus \rho^*(\CV' \otimes \omega_S^{-1}) \otimes \CO_2(1) \ \ =: \ \ \widetilde{\CE}
\end{equation}
The notation $\CO_1(1)$, $\CO_2(1)$ stands for the tautological line bundles on the two projectivizations that appear in \eqref{eqn:z2 fib diag}, and $\sigma$, $\sigma'$ are the sections in \eqref{eqn:koszul 1}, \eqref{eqn:koszul 2}. \\

\begin{proposition}
\label{prop:vanish}

The composition of the section $\sigma \oplus \sigma'$ with the map:
$$
\rho^*({\CW'}^\vee) \otimes \CO_1(1) \bigoplus \rho^*(\CV' \otimes \omega_S^{-1}) \otimes \CO_2(1) \xrightarrow{\emph{taut}_2 \ominus \emph{taut}_1} \CO_1(1) \otimes \CO_2(1) \otimes \rho^*(\omega_S^{-1})
$$
vanishes. Above, $\emph{taut}_{1,2}$ refers to the tautological surjective homomorphisms: 
\begin{equation}
\label{eqn:taut hom}
\rho^*(\CV') \xrightarrow{\emph{taut}_1} \CO_1(1) \qquad \text{and} \qquad \rho^*({\CW'}^\vee \otimes \omega_S) \xrightarrow{\emph{taut}_2} \CO_2(1)
\end{equation}
on $\BP_{\CM' \times S}(\CV')$ and $\BP_{\CM' \times S}({\CW'}^\vee \otimes \omega_S)$, respectively (we abuse notation by also writing $\emph{taut}_{1,2}$ for the homomorphisms \eqref{eqn:taut hom} tensored by arbitrary line bundles). \\

\end{proposition} 

\begin{proof} Explicitly, the Proposition states that the following compositions coincide:
$$
\underbrace{\CO_1(-1) \xrightarrow{\text{taut}_1^\vee} \rho^*({\CV'}^\vee) \xrightarrow{j^\vee} \rho^*({\CW'}^\vee)}_{\sigma} \xrightarrow{\text{taut}_2} \CO_2(1) \otimes \rho^*(\omega_S^{-1})
$$
$$
\underbrace{\CO_2(-1) \xrightarrow{\text{taut}_2^\vee} \rho^*(\CW' \otimes \omega_S^{-1}) \xrightarrow{j} \rho^*(\CV' \otimes \omega_S^{-1})}_{\sigma'} \xrightarrow{\text{taut}_1} \CO_1(1) \otimes \rho^*(\omega_S^{-1})
$$
up to tensoring with a suitable line bundle. The map $j$ is induced by the map of vector bundles $\CW' \rightarrow \CV'$ that features in \eqref{eqn:length 1}, with $\CM'$ instead of $\CM$, and $j^\vee$ denotes its dual. Then the equality of the two compositions above is a consequence of the standard linear algebra fact that:
$$
\lambda \cdot A \cdot v = v^T \cdot A^T \cdot \lambda^T
$$
where $\BC \xrightarrow{v} \BC^m$ is a vector, $\BC^m \xrightarrow{A} \BC^n$ is a matrix, and $\BC^n \xrightarrow{\lambda} \BC$ is a covector.
	
\end{proof} 

\begin{definition}
\label{def:fine}

Define the derived scheme:
\begin{equation}
\label{eqn:fine}
\xymatrix{\fZ_2^\bullet \ar[rd] \ar@{^{(}->}[r]^-{\iota}
& \BP_{\CM' \times S} \left(\CV'\right) \times_{\CM' \times S} \BP_{\CM' \times S} \left( {\CW'}^\vee \otimes \omega_S \right) \ar@{->>}[d]^\rho \\
& \CM' \times S}
\end{equation}
by the requirement that the embedding be the derived zero locus of the section:
\begin{equation}
\label{eqn:section fine}
\CO \xrightarrow{\sigma \oplus \sigma'} \CE
\end{equation}
given by the formula \eqref{eqn:section}, but mapping into the vector bundle:
\begin{equation}
\label{eqn:ker bundle}
\CE = \emph{Ker} \left[ \widetilde{\CE} \xrightarrow{\emph{taut}_2 \ominus \emph{taut}_1} \CO_1(1) \otimes \CO_2(1) \otimes \rho^*(\omega_S^{-1}) \right]
\end{equation}
(the maps $\emph{taut}_{1,2}$ are defined in Proposition \ref{prop:vanish}). \\

\end{definition}

\subsection{} 
\label{sub:fine 2}

The closed subscheme $\bfZ_2^\bullet \hookrightarrow \bfZ_2$ consisting of flags \eqref{eqn:z2} with $x_1 = x_2$ is the support of $\fZ_2^\bullet$. Moreover, we have a natural map of dg schemes:
\begin{equation}
\label{eqn:map}
\fZ_2^\bullet \rightarrow \fZ_2|_\Delta
\end{equation}
arising from the fact that the former (respectively latter) dg scheme is the derived zero locus of the section \eqref{eqn:section fine} (respectively \eqref{eqn:section}). Moreover, the map of dg schemes \eqref{eqn:map} is compatible with the equality of their supports $\bfZ_2^\bullet = \bfZ_2|_\Delta$. The line bundles $\CO_1(1)$, $\CO_2(1)$ on the two projective bundles $\BP(\CV') \times \BP({\CW'}^\vee \otimes \omega_S)$ in \eqref{eqn:fine} restrict to the line bundles $\CL_1, \CL_2^{-1}$ on either of the dg schemes $\fZ_2|_\Delta$ and $\fZ_2^\bullet$. In terms of the support schemes $\bfZ_2|_\Delta = \bfZ_2^\bullet$, these line bundles are given by:
\begin{align*} 
&\CL_1|_{(\CF \supset_x \CF' \supset_x \CF'')} = \CF'_x/\CF''_x \\
&\CL_2|_{(\CF \supset_x \CF' \supset_x \CF'')} = \CF_x/\CF'_x 
\end{align*}
Finally, note that we have natural forgetful maps:
\begin{align} 
&\fZ_2^\bullet \xrightarrow{\pi_-} \fZ_1 \label{eqn:pi-} \\
&\fZ_2^\bullet \xrightarrow{\pi_+} \fZ_1 \label{eqn:pi+}
\end{align} 
which lie above the following forgetful maps between their supports:
\begin{align*} 
&\bfZ_2^\bullet \xrightarrow{\pi_-} \bfZ_1, \qquad (\CF \supset_x \CF' \supset_x \CF'') \mapsto (\CF \supset_x \CF') \\
&\bfZ_2^\bullet \xrightarrow{\pi_+} \bfZ_1, \qquad (\CF \supset_x \CF' \supset_x \CF'') \mapsto (\CF' \supset_x \CF'')
\end{align*}
At the level of dg structures, the map \eqref{eqn:pi-} arises from the commutativity of the following diagram of projective bundles on $\CM' \times S$, coupled with \eqref{eqn:desc 2} and \eqref{eqn:fine}:
$$
\xymatrix{ & & \CE \ar@{.>}[dl] \ar[d]^{\text{projection}} \\
\fZ_2^\bullet \ar[d]_{\pi_-} \ar@{^{(}->}[r]^-{\iota}
& \BP_{\CM' \times S} \left(\CV'\right) \times_{\CM' \times S} \BP_{\CM' \times S} \left( {\CW'}^\vee \otimes \omega_S \right) \ar[d]^{\text{projection}} & \CV' \otimes \omega_S^{-1} \otimes \CO(1) \ar@{.>}[dl] \\
\fZ_1 \ar@{^{(}->}[r]^-{\iota_+} & \BP_{\CM' \times S} \left(\CW'^\vee\otimes \omega_S \right) & }
$$
In other words, the homomorphism between the dg structure sheaves of $\fZ_1$ and $\fZ_2^\bullet$ is induced by the duals of the maps labeled ``projection" in the formula above. An analogous diagram defines the map \eqref{eqn:pi+}. At the level of $K$--theory, the map $\pi_-: \fZ_2^\bullet \rightarrow \fZ_1$ is the projectivization of the virtual vector bundle:
\begin{equation}
\label{eqn:proj-}
[\CV'] - [\CW'] + [\CL \otimes \omega_S]
\end{equation}
where the locally free sheaves $\CW',\CV'$ on $\fZ_1$ correspond to the sheaf denoted $\CF'$ in: 
$$
\bfZ_1 = \text{supp } \fZ_1 = \{\CF \supset \CF'\}
$$
Similarly, the map $\pi_+: \fZ_2^\bullet \rightarrow \fZ_1$ is the projectivization of the virtual vector bundle:
\begin{equation}
\label{eqn:proj+}
\left[{\CW'}^\vee \otimes \omega_S \right] - \left[{\CV'}^\vee \otimes \omega_S \right] + [\CL^{-1} \otimes \omega_S] \qquad \text{where} \quad \text{supp } \fZ_1 = \{\CF' \supset \CF''\}  
\end{equation}
where the locally free sheaves $\CW',\CV'$ on $\fZ_1$ correspond to the sheaf denoted $\CF'$ in:
$$
\bfZ_1 = \text{supp } \fZ_1 = \{\CF' \supset \CF''\}
$$

\subsection{}
\label{sub:smooth}

For the remainder of the Section, we will operate under the following Assumptions, which also featured in the Heisenberg algebra construction of \cite{Ba}:
\begin{equation}
\label{eqn:assumption s}
\textbf{Assumption S}: \text{either} \quad \begin{cases} \omega_S \cong \CO_S \quad \ \ \text{or} \\ c_1(\omega_S) \cdot H < 0 \end{cases}
\end{equation}
Since any endomorphism of a $H$-stable sheaf $\CF$ is either 0 or an isomorphism, we conclude that $\Hom(\CF,\CF) = \BC$. Meanwhile, Serre duality implies that:
$$
\Ext^2(\CF,\CF) \cong \Hom(\CF,\CF \otimes \omega_S)^\vee
$$
The Hom space in the right-hand side is canonically isomorphic to either $\BC$ or 0, depending on which of the two options of \eqref{eqn:assumption s} holds. Since the tangent space to $\CM$ at the point $\CF$ is isomorphic to $\Ext^1(\CF,\CF)$, the descriptions of $\Hom(\CF,\CF)$ and $\Ext^2(\CF,\CF)$ above imply that $\CM$ is smooth (see \cite{HL} for details). \\

\begin{proposition}
\label{prop:z1 smooth}

Under Assumption S, the section \eqref{eqn:koszul 1} is regular, hence the dg scheme $\fZ_1$ coincides with its support $\bfZ_1$ (and the latter scheme is smooth). \\ 

\end{proposition}

\noindent Proposition \ref{prop:z1 smooth} was proved in \cite[Proposition 2.10]{Univ}. The method of attack was to estimate the dimensions of the tangent spaces to $\bfZ_1$. Since $\bfZ_1$ parametrizes nested sheaves, it is well-known that its tangent spaces are given by the formula:
\begin{equation}
\label{eqn:tan 1}
\Tan_{(\CF \supset \CF')}\bfZ_1 = \Ker \left(\Ext^1(\CF,\CF) \oplus \Ext^1(\CF',\CF') \xrightarrow{(h,-v)} \Ext^1(\CF',\CF) \right)
\end{equation}
where the maps $h$ and $v$ are given as in the following diagram, all of whose rows and columns are exact (we write $\BC_x$ for the length 1 quotient $\CF/\CF'$):

$$
\xymatrix{\Tan_x S \ar[r] & \Ext^1(\CF,\BC_x) \ar[r] & \Ext^1(\CF',\BC_x) \ar@{->>}[r] & \BC \\
\Ext^1(\BC_x, \CF) \ar@{^{(}->}[r] \ar[u] & \Ext^1(\CF,\CF) \ar[r]^-h \ar@{->>}[u] & \Ext^1(\CF',\CF) \ar[r] \ar[u] & \Ext^2(\BC_x, \CF) \ar@{->>}[u] \\
\Ext^1(\BC_x, \CF') \ar[r] \ar[u] & \Ext^1(\CF,\CF') \ar[r]_{h'} \ar[u]^{v'} & \Ext^1(\CF',\CF') \ar[u]_-{v} \ar@{->>}[r] & \Ext^2(\BC_x,\CF') \ar[u] \\
\BC \ar@{^{(}->}[r] \ar@{_{(}->}[u] & \Hom(\CF,\BC_x) \ar[r] \ar[u] & \Hom(\CF', \BC_x) \ar[r] \ar@{_{(}->}[u] & \Tan_x S \ar[u]}
$$
The vector spaces in the corners of the diagram are: 
$$
\Tan_x S \cong \Ext^1(\BC_x,\BC_x)
$$
and:
$$
\BC \cong \Hom(\BC_x, \BC_x) \cong \Ext^2(\BC_x, \BC_x)
$$
It was shown in \cite[Proposition 2.10]{Univ} that the differential of $\bfZ_1 \stackrel{p_S}\rightarrow S$ is surjective:
\begin{equation}
\label{eqn:surj 1}
p_{S*} : \Tan_{(\CF \supset_x \CF')}\bfZ_1 \twoheadrightarrow \Tan_x S
\end{equation}
and that the kernel of this map consists of pairs of Exts that come from $\Ext^1(\CF,\CF')$:
\begin{equation}
\label{eqn:ker 1}
(v',h') : \Ext^1(\CF,\CF') \stackrel{\sim}\longrightarrow \text{Ker }p_{S*} \subset \text{right-hand side of } \eqref{eqn:tan 1}
\end{equation}
We will use \eqref{eqn:surj 1} and \eqref{eqn:ker 1} to prove the following analogue of Proposition \ref{prop:z1 smooth}. \\


\begin{proposition}
\label{prop:z2 smooth}

Under Assumption S, the section \eqref{eqn:section fine} is regular, hence the dg scheme $\fZ_2^\bullet$ coincides with its support $\bfZ_2^\bullet$ (and the latter scheme is smooth). \\ 

\end{proposition}

\begin{proof} By analogy with \eqref{eqn:z2 fib diag}, the definition of $\fZ_2$ as the derived fiber product $\fZ_1 \times_{\CM'} \fZ_1$ implies that we have the following commutative triangle of maps:
\begin{equation}
\label{eqn:exp diag}
\xymatrix{\fZ_2 \ar[rd] \ar@{^{(}->}[r]
& \BP_{\CM' \times S \times S} \left( {\CW'_1}^\vee \otimes \omega_S \right) \times_{\CM' \times S \times S} \BP_{\CM' \times S \times S} \left(\CV'_2\right) \ar@{->>}[d]^{\tilde{\rho}} \\
& \CM' \times S \times S}
\end{equation}
where $\CW'_1$ (respectively $\CV'_2$) refers to the pull-back of the vector bundle $\CW'$ (respectively $\CV'$) from $\CM \times S$ to $\CM \times S \times S$ via the first (respectively second) projection of $S \times S$ onto $S$. The embedding in \eqref{eqn:exp diag} is the derived zero locus of the section:
\begin{equation}
\label{eqn:exp section}
\CO \xrightarrow{\sigma'_1 \oplus \sigma_2} \tilde{\rho}^*(\CV'_1 \otimes \omega_S^{-1}) \otimes \CO_2(1) \bigoplus \tilde{\rho}^*({\CW'_2}^\vee) \otimes \CO_1(1) 
\end{equation}
and it is easy to see that the scheme-theoretic zero locus of the section \eqref{eqn:exp section} is $\bfZ_2$ of \eqref{eqn:z2}. Because of the Grothendieck-Hirzebruch-Riemann-Roch theorem, the dimension of the smooth variety $\CM'$ at any point $\CF'$ of second Chern class $c_2$ is:
\begin{equation}
\label{eqn:exp dim 0}
\dim \Ext^1(\CF',\CF') = 1 + \e - \dim \chi(\CF',\CF') = 1 + \e + \text{const} + 2rc_2 
\end{equation}
where the integer labeled ``const" depends only on $S, H, r, c_1$, and the number $\e$ is equal to the dimension of $\Ext^2(\CF',\CF')$, which is 1 or 0 depending on which of the two cases in Assumption S holds. The dimension of the derived scheme $\fZ_2$ is the dimension of the projective bundle $\BP \times \BP$ in \eqref{eqn:exp diag} minus the rank of the vector bundle in the right-hand side of \eqref{eqn:exp section}. Since $\text{rank } \CV' - \text{rank } \CW' = r$, we have:
\begin{multline}
\label{eqn:exp dim}
\text{dim }\fZ_2 = \dim \CM' + \underbrace{2 + 2}_{\dim S \times S}  + \\ + (\text{rank } \CW_1' + \text{rank } \CV_2' - 2 - \text{rank } \CV_1' - \text{rank } \CW_2')  = 1 + \e + \text{const} + 2rc_2 + 2 
\end{multline}
As for the scheme-theoretic zero locus of the section \eqref{eqn:exp section}, the computation above implies that the dimensions of all of its local rings satisfy:
\begin{equation}
\label{eqn:actual dim}
\text{dim }\bfZ_2 \geq \text{RHS of \eqref{eqn:exp dim}} = 1 + \e + \text{const} + 2rc_2 + 2 
\end{equation}
By a well-known deformation-theory argument, the tangent space to a point of $\bfZ_2$ parametrizing a flag $(\CF \supset_{x_1} \CF' \supset_{x_2} \CF'')$ equals the space of triples of extensions:
\begin{equation}
\label{eqn:triple}
\xymatrix{0 \ar[r] & \CF \ar[r]  & \CS  \ar[r] & \CF \ar[r] & 0 \\
0 \ar[r] & \CF' \ar[r] \ar@{^{(}->}[u] & \CS' \ar[r] \ar@{^{(}->}[u] & \CF' \ar@{^{(}->}[u] \ar[r] & 0 \\
0 \ar[r] & \CF'' \ar[r] \ar@{^{(}->}[u] & \CS'' \ar[r] \ar@{^{(}->}[u] & \CF'' \ar@{^{(}->}[u] \ar[r] & 0}
\end{equation}
By a simple diagram chase, one can see that this space is:
\begin{equation}
\label{eqn:exp tan}
\text{Tan}_{(\CF \supset_{x_1} \CF' \supset_{x_2} \CF'')}\bfZ_2 = 
\end{equation}
$$
= \text{Ker} \Big[ \Ext^1(\CF,\CF) \oplus \Ext^1(\CF',\CF') \oplus \Ext^1(\CF'',\CF'') \stackrel{\lambda}\longrightarrow \Ext^1(\CF',\CF) \oplus \Ext^1(\CF'',\CF') \Big]
$$
where the map $\lambda$ is the alternating sum of the fourn+atural maps induced by the inclusions $\CF \supset \CF' \supset \CF''$ on the Ext spaces in question (i.e. the maps that feature in \eqref{eqn:tan 1}, two for each of the pairs $\CF \supset \CF'$ and $\CF' \supset \CF''$). By analogy with \eqref{eqn:exp dim 0}, we have the following formulas for the dimensions of the following Ext spaces:
\begin{align*}
\dim \Ext^1(\CF,\CF) &= 1 + \e + \text{const} + 2rc_2 - 2r \\
\dim \Ext^1(\CF'',\CF'') &= 1 + \e + \text{const} + 2rc_2 + 2r \\
\dim \Ext^1(\CF',\CF) &= 1 + 0 + \text{const} + 2rc_2 - r \\
\dim \Ext^1(\CF'',\CF') &= 1 + 0 + \text{const} + 2rc_2 + r \\
\dim \Ext^1(\CF,\CF') &= 0 + \e + \text{const} + 2rc_2 - r \\ 
\dim \Ext^1(\CF',\CF'') &= 0 + \e + \text{const} + 2rc_2 + r 
\end{align*}
(the first two summands in each RHS represents the dimension of the corresponding $\Hom$ and $\Ext^2$ spaces, subject to Assumption S). Therefore, the tangent space \eqref{eqn:exp tan} has expected dimension \eqref{eqn:exp dim} if and only if $\lambda$ has cokernel of dimension $2-2\e$. \\

\begin{claim}
\label{claim:tan}

The map $\lambda$ of \eqref{eqn:exp tan} has cokernel of dimension:
\begin{equation}
\label{eqn:claim}
\begin{cases} 3-2\e & \text{if } x_1 = x_2 \text{ and } \CF/\CF'' \cong \BC_{x_1} \oplus \BC_{x_2} \\ 2-2\e & \text{otherwise} \end{cases}
\end{equation}
In the latter case, $\bfZ_2$ is smooth of expected dimension at the point $(\CF \supset \CF' \supset \CF'')$, while in the former case its tangent space has dimension 1 greater than expected. \\

\end{claim}

\begin{proof}

From the long exact sequences corresponding to the $\Ext$ functor, we obtain:
\begin{align*} 
&\lambda \left(\Ext^1(\CF,\CF),0,0 \right) = \left( \text{Ker} \left[ \Ext^1(\CF',\CF) \stackrel{p_1}\rightarrow \Ext^2(\BC_{x_1},\CF) \right] ,0 \right) \\
&\lambda \left(0,0,\Ext^1(\CF'',\CF'') \right) = \left(0,\text{Ker} \left[ \Ext^1(\CF'',\CF') \stackrel{p_2}\rightarrow \Ext^1(\CF'',\BC_{x_2}) \right] \right)
\end{align*}
where $\BC_{x_1}$ and $\BC_{x_2}$ denote $\CF/\CF'$ and $\CF'/\CF''$, respectively. Therefore, the cokernel of the map $\lambda$ is isomorphic to the cokernel of the map:
\begin{equation}
\label{eqn:mu}
\Ext^1(\CF',\CF') \xrightarrow{(\mu_1,\mu_2)} \text{Im } p_1 \oplus \text{Im } p_2
\end{equation}
where the maps $\mu_1$ and $\mu_2$ are as in the diagrams below:
$$
\xymatrix{ \Hom(\CF',\BC_{x_1}) \ar[d] \ar[r] & \Ext^1(\CF',\CF') \ar@{.>}[dr]^{\mu_1} \ar@{->>}[d]_{r_1} \ar[r] & \Ext^1(\CF',\CF) \ar[d]^{p_1} \ar@{->>}[r] & \Ext^1(\CF',\BC_{x_1}) \ar@{->>}[d]  \\
\Ext^1(\BC_{x_1},\BC_{x_1}) \ar[r] & \Ext^2(\BC_{x_1},\CF') \ar[r]_{q_1} & \Ext^2(\BC_{x_1},\CF) \ar@{->>}[d]^{t_1} \ar@{->>}[r]_{s_1} & \Ext^2(\BC_{x_1},\BC_{x_1}) \\ & & \BC^\e &}
$$
$$
\xymatrix{ \Ext^1(\BC_{x_2}, \CF') \ar[r] \ar[d] & \Ext^1(\CF',\CF') \ar[r] \ar@{->>}[d]_{r_2} \ar@{.>}[dr]^{\mu_2} & \Ext^1(\CF'',\CF') \ar[d]^{p_2} \ar[r] & \Ext^2(\BC_{x_2}, \CF') \ar@{->>}[d] \\
\Ext^1(\BC_{x_2}, \BC_{x_2}) \ar[r] & \Ext^1(\CF',\BC_{x_2}) \ar[r]_{q_2} & \Ext^1(\CF'',\BC_{x_2}) \ar@{->>}[d]^{t_2} \ar@{->>}[r]_{s_2} & \Ext^2(\BC_{x_2}, \BC_{x_2}) \\ & & \BC^\e &}
$$
In each of the above diagrams, the space in the bottom right is $\cong \BC$. \\

\begin{itemize}[leftmargin=4mm]

\item If $\e = 1$, it is easy to see that $\text{Ker } s_i = \text{Ker }t_i$ for both $i \in \{1,2\}$. Therefore, we may rewrite \eqref{eqn:mu} by saying that $\text{Coker }\lambda$ is isomorphic to the cokernel of:
\begin{equation}
\label{eqn:mu 1}
\Ext^1(\CF',\CF') \xrightarrow{(\mu_1, \mu_2)} \text{Im } q_1 \oplus \text{Im } q_2
\end{equation}
We want to show that the cokernel of \eqref{eqn:mu 1} has dimension 1 (respectively 0) in the first (respectively second) case of \eqref{eqn:claim}. Since $\mu_i = q_i \circ r_i$, we will solve this problem by noting that the following square commutes:
\begin{equation}
\label{eqn:boy}
\xymatrix{\Ext^1(\CF',\CF') \ar@{->>}[r]^-{r_1} \ar@{->>}[d]_-{r_2} & \Ext^2(\BC_{x_1},\CF') \ar@{->>}[d]^-{y_2} \\
\Ext^1(\CF',\BC_{x_2}) \ar@{->>}[r]_-{y_1} & \Ext^2(\BC_{x_1},\BC_{x_2})}
\end{equation}
where the horizontal maps are induced by the sequence $0\rightarrow \CF' \rightarrow \CF \rightarrow \BC_{x_1} \rightarrow 0$, and the vertical maps are induced by the sequence $0\rightarrow \CF'' \rightarrow \CF' \rightarrow \BC_{x_2} \rightarrow 0$. By applying Serre duality to the square \eqref{eqn:boy}, it is easy to show that: 
\begin{equation}
\label{eqn:key}
\text{Im }(r_1,r_2) = \text{Ker }(y_2,-y_1)
\end{equation} 
Therefore, the row and column of the following diagram are exact:
$$
\xymatrix{& \Ext^1(\BC_{x_1}, \BC_{x_1}) \ar@{.>}[dr]^\rho \oplus \Ext^1(\BC_{x_2}, \BC_{x_2}) \ar[d] & \\
\Ext^1(\CF',\CF') \ar[r]^-{(r_1,r_2)} \ar@{.>}[dr]_{(\mu_1,\mu_2)} & \Ext^2(\BC_{x_1},\CF') \oplus \Ext^1(\CF',\BC_{x_2}) \ar@{->>}[d]_{(q_1,q_2)}  \ar[r]^-{(y_2,-y_1)} & \Ext^2(\BC_{x_1}, \BC_{x_2}) \\
& \text{Im }q_1 \oplus \text{Im }q_2 & }
$$
If $x_1 \neq x_2$, then $\Ext^2(\BC_{x_1}, \BC_{x_2}) = 0$, and the surjectivity of $(r_1,r_2)$ implies the surjectivity of $(\mu_1,\mu_2)$, as required. If $x_1 = x_2$, then $\Ext^2(\BC_{x_1}, \BC_{x_2}) \cong \BC$, and the cokernel of $(\mu_1,\mu_2)$ is isomorphic to the cokernel of $\rho$. The latter is 1 or 0 dimensional depending on whether $\CF/\CF''$ is  isomorphic to $\BC_{x_1} \oplus \BC_{x_2}$ or not. \\

\item If $\e = 0$, then $\text{Im }q_i$ has codimension 1 inside $\text{Im }p_i$, and so we may rewrite \eqref{eqn:mu} by saying that $\text{Coker }\lambda$ is isomorphic to the cokernel of:
\begin{equation}
\label{eqn:mu 2}
\Ext^1(\CF',\CF') \xrightarrow{(\mu_1, \mu_2)} \text{Im } q_1 \oplus \text{Im } q_2 \oplus \BC^2
\end{equation}
(although the right-hand side of \eqref{eqn:mu 2} should more appropriately be called ``a vector space containing $\text{Im } q_1 \oplus \text{Im } q_2$ as a codimension 2 subspace"). One can repeat the argument following equation \eqref{eqn:mu 1} in the previous bullet to give us the desired values \eqref{eqn:claim} for the dimension of the cokernel of the map \eqref{eqn:mu 2}. 

\end{itemize}

\end{proof}

\noindent The dg scheme $\fZ_2^\bullet$ defined as the derived zero locus of the section \eqref{eqn:section fine} has:
\begin{equation}
\label{eqn:exp dim fine}
\text{dim }\fZ_2^\bullet = 1 + \e + \text{const} + 2rc_2 + 1 
\end{equation}
Note that this is 1 less than the dimension of $\fZ_2$ from \eqref{eqn:exp dim}: two dimensions fewer because we require the support points $x_1$ and $x_2$ to be the same, but one dimension more because the section \eqref{eqn:section fine} maps to a vector bundle of rank 1 less than \eqref{eqn:exp section}. To conclude the proof of Proposition \ref{prop:z2 smooth}, it is therefore enough to show that the tangent spaces to the support scheme:
\begin{equation}
\label{eqn:z2 pt}
\bfZ_2^\bullet = \Big\{ (\CF,\CF',\CF'') \text{ s.t. } \CF \supset_{x} \CF' \supset_{x} \CF'' \text{ for some } x \in S \Big\} 
\end{equation}
have dimension less than or equal to \eqref{eqn:exp dim fine}. To this end, note that the composition of the linear maps:
\begin{equation}
\label{eqn:pss}
\Tan_{(\CF \supset_{x_1} \CF' \supset_{x_2} \CF'')} \bfZ_2^\bullet \hookrightarrow \Tan_{(\CF \supset_{x_1} \CF' \supset_{x_2} \CF'')} \bfZ_2 \xrightarrow{\pss} \Tan_{x_1} S \oplus \Tan_{x_2} S
\end{equation}
\footnote{We must have $x_1 = x_2$ in order for $(\CF \supset_{x_1} \CF' \supset_{x_2} \CF'')$ to lie in $\bfZ_2^\bullet$, but we will use the notation $x_1$ and $x_2$ in order to not confuse the length 1 quotients $\CF/\CF' \cong \BC_{x_1}$ and $\CF'/\CF'' \cong \BC_{x_2}$} lands in the 2-dimensional diagonal subspace $\Tan_x S \hookrightarrow \Tan_{x_1} S \oplus \Tan_{x_2}S$. Because of the last sentence in the statement of Claim \ref{claim:tan}, in order to obtain the desired estimate on the dimensions of the tangent spaces to $\bfZ_2^\bullet$, it suffices to prove that:
\begin{align} 
&\text{the map } \pss \text{ is surjective  if } \CF/\CF'' \cong \BC_{x_1} \oplus \BC_{x_2} \label{eqn:want 1} \\
&\text{the map } \pss \text{ has 1-dimensional cokernel if } \CF/\CF'' \not\cong \BC_{x_1} \oplus \BC_{x_2} \label{eqn:want 2}
\end{align}Let us first tackle the first of these, i.e. the case \eqref{eqn:want 1}. It suffices to show that for any tangent vector $v \in \Tan_{x_1} S = \Tan_{x_2} S$, there exist tangent vectors to $\bfZ_2$ which map to either $(v,0)$ or $(0,v)$. We will only prove the former statement, as the latter is analogous and we leave it to the interested reader. By \eqref{eqn:surj 1}, there exists:
$$
(w,w') \in \text{Ker} \left(\Ext^1(\CF,\CF) \oplus \Ext^1(\CF',\CF') \longrightarrow \Ext^1(\CF',\CF) \right)
$$
which maps to $v$ under $p^1_{S*}$. By \eqref{eqn:ker 1}, we may change $w$ and $w'$ by the image of one and the same element of $\Ext^1(\CF,\CF')$ without changing the vector $v$. The task is to complete the datum above to a triple:
$$
(w,w',w'') \in \Tan_{(\CF \supset_{x_1} \CF' \supset_{x_2} \CF'')} \bfZ_2 \subset \Ext^1(\CF,\CF) \oplus \Ext^1(\CF',\CF') \oplus \Ext^1(\CF'',\CF'')
$$
However, in order to ensure that $(p_{S*}^1 \times p_{S*}^2)(w,w',w'') = (v,0)$, both $w'$ and $w''$ must come from one and the same element of $\Ext^1(\CF',\CF'')$. This is equivalent to requiring that in the diagram below with exact row and columns:
\begin{equation}
\label{eqn:volo}
\xymatrix{& \Ext^1(\CF,\CF') \ar@{->>}[r] \ar[d] & \Ext^1(\CF,\BC_{x_2}) \ar[d] \\
\Ext^1(\CF',\CF'') \ar@{->>}[d] \ar[r] & \Ext^1(\CF',\CF') \ar@{->>}[d]_{r_1} \ar@{->>}[r]^{r_2} \ar@{.>}[rd]^\phi & \Ext^1(\CF', \BC_{x_2}) \ar@{->>}[d]^{y_1} \\
\Ext^2(\BC_{x_1},\CF'') \ar[r] & \Ext^2(\BC_{x_1},\CF') \ar@{->>}[r]_{y_2} & \Ext^2(\BC_{x_1}, \BC_{x_2})}
\end{equation}
we may choose $w' \in \Ext^1(\CF',\CF')$ to map to 0 in $\Ext^1(\CF',\BC_{x_2})$. But remember that $w'$ may only be changed by adding the image of an arbitrary element in $\Ext^1(\CF,\CF')$, and so chasing through \eqref{eqn:volo} implies that it is enough to show that $\phi(w') = 0$. To this end, recall that the extensions $w$ and $w'$ have the property that they map to one and the same element in $\Ext^1(\CF',\CF)$. Consider the following diagram:
\begin{equation}
\label{eqn:dya}
\xymatrix{ & \Ext^1(\CF,\CF) \ar[d] \\
\Ext^1(\CF',\CF') \ar[r] \ar[d] \ar@{.>}[rd]^{\psi} & \Ext^1(\CF',\CF) \ar[d]  \\
\Ext^2(\BC_{x_1}, \CF') \ar[r] & \Ext^2(\BC_{x_1}, \CF)}  
\end{equation}
where the two vertical arrows on the right have composition equal to 0. Since $w$ and $w'$ map to the same element in $\Ext^1(\CF',\CF)$, this implies that $\psi(w') = 0$. In turn, this yields the desired property $\phi(w') = 0$, because if we let:
$$
\kappa : \Ext^2(\BC_{x_1},\CF) \rightarrow \Ext^2(\BC_{x_1}, \BC_{x_2})
$$
be the map induced by the composition $\CF \rightarrow \CF/\CF'' \cong \BC_{x_1} \oplus \BC_{x_2} \xrightarrow{\text{proj}_1} \BC_{x_1} \cong \BC_{x_2}$, then a straightforward diagram chase reveals the fact that $\phi = \kappa \circ \psi$. \\

\noindent Let us now treat the case \eqref{eqn:want 2}. We may ask for which $v \in \Tan_{x_1} S = \Tan_{x_2} S$ can we choose a tangent vector $(w,w',w'')$ to $\bfZ_2$ which maps to either $(v,0)$ or $(0,v)$ under the map \eqref{eqn:pss}. Chasing through the proof on the previous page shows that such a tangent vector $(w,w',w'')$ can be chosen if and only if $\phi(w') = 0$, where $\phi$ is the diagonal arrow in \eqref{eqn:volo}. Since the image of $\phi$ is one-dimensional, this places a single linear condition on the vector $w'$, and therefore there exists a one-dimensional family of $v$'s such that $(v,0)$ and $(0,v)$ lie in the image of $\pss$. \\

\noindent To prove \eqref{eqn:want 2}, we must show that the image of $\pss$ is three-dimensional, and so we must show that there exists at least one more vector in the image. In fact, we will show that vectors of the form $(v,v)$ are always in the image of $\pss$, for any $v \in \Tan_{x_1} S = \Tan_{x_2} S$. By \eqref{eqn:surj 1}, we may choose vectors:
\begin{align}
&\left(\begin{array}{c}
w  \\
w'  \\    
\end{array} \right) \in \text{Ker} \left( \begin{aligned}
\Ext^1(\CF, \CF) & &  \\
\oplus \qquad \ & \longrightarrow & \Ext^1(\CF',\CF) \\
\Ext^1(\CF', \CF') & &     
\end{aligned} \right) = \Tan_{(\CF\supset_{x_1} \CF')} \bfZ_1 \label{eqn:ww} \\
&\left(\begin{array}{c}
\bar{w}'  \\
w''  \\    
\end{array} \right) \in \text{Ker} \left( \begin{aligned}
\Ext^1(\CF', \CF') & &  \\
\oplus \qquad \ & \longrightarrow & \Ext^1(\CF'',\CF') \\
\Ext^1(\CF'', \CF'') & &     
\end{aligned} \right) = \Tan_{(\CF'\supset_{x_2} \CF'')} \bfZ_1 \label{eqn:www}
\end{align}
which map to the same vector $v \in \Tan_{x_1} S = \Tan_{x_2} S$ under the map $p_{S*}$. According to \eqref{eqn:ker 1}, we may modify the vectors $w'$ and $\bar{w}'$ by arbitrary elements coming from $\Ext^1(\CF,\CF')$ and $\Ext^1(\CF',\CF'')$, respectively, without modifying the images of $p_{S*}$. We need to show that we can perform these modifications so as to ensure $w' = \bar{w}'$, as then $(w,w',w'')$ will map to $(v,v)$ under $\pss$. The elements $w'$ and $\bar{w}'$ lie in the space $\Ext^1(\CF',\CF')$ in the middle of the diagram \eqref{eqn:volo}. Chasing through the aforementioned diagram, we see that we can make $w'$ equal to $\bar{w}'$ upon modification by elements coming from $\Ext^1(\CF,\CF')$ and $\Ext^1(\CF',\CF'')$, respectively, iff:
\begin{equation}
\label{eqn:finally}
\phi(w') = \phi(\bar{w}')
\end{equation}
The fact that equality \eqref{eqn:finally} holds is a direct consequence of the following claim: \\


\begin{claim}

The element $\phi(w')$ coincides with the image of $v$ under the map:
\begin{equation}
\label{eqn:of}
\emph{Ext}^1(\BC_{x_1},\BC_{x_1}) \longrightarrow \emph{Ext}^2(\BC_{x_1}, \BC_{x_2})
\end{equation}
induced by the short exact sequence: 
$$
0 \longrightarrow \CF'/\CF'' \cong \BC_{x_2} \longrightarrow \CF/\CF'' \longrightarrow \CF/\CF' \cong \BC_{x_1} \longrightarrow 0
$$
The analogous statement holds for $\phi(\bar{w}')$, which implies equality \eqref{eqn:finally}. \\

\end{claim}

\noindent We will prove the claim about $\phi(w')$, and leave the analogous result for $\phi(\bar{w}')$ as an exercise for the interested reader. Assume the pair $(w,w')$ of \eqref{eqn:ww} corresponds to a commutative diagram:
\begin{equation}
\label{eqn:double}
\xymatrix{0 \ar[r] & \CF \ar[r]^{\iota} & \CS  \ar[r]^{\pi} & \CF \ar[r] & 0 \\
0 \ar[r] & \CF' \ar[r]^{\iota'} \ar@{^{(}->}[u] & \CS' \ar[r]^{\pi'} \ar@{^{(}->}[u] & \CF' \ar@{^{(}->}[u] \ar[r] & 0}
\end{equation}
By a straightforward diagram chase, $\phi(w')$ is given by the extension:
\begin{equation}
\label{eqn:ofof}
0 \longrightarrow \CF'/\CF'' \cong \BC_{x_2} \stackrel{(1,0)}\longrightarrow \CF'/\CF'' \oplus_{\CF'} \CS' \stackrel{(0,\pi')}\longrightarrow \CF \longrightarrow \CF/\CF' \cong \BC_{x_1} \longrightarrow 0
\end{equation}
and the notation $\oplus_{\CF'}$ denotes the push-out with respect to the standard projection map $\CF' \twoheadrightarrow \CF'/\CF''$ and the map $\iota' : \CF' \rightarrow \CS'$. Because of the commutative diagram:
\begin{equation}
\label{eqn:ofofof}
\xymatrix{0 \ar[r] & \CF'/\CF'' \ar[r]^-{(1,0)} \ar@{=}[d] & \CF'/\CF'' \oplus_{\CF'} \CS' \ar[r]^-{(0,\pi')} \ar[d]_{(\iota',\text{projection})} & \CF \ar[r] \ar[d]^{\text{projection}} & \CF/\CF' \ar[r] \ar@{=}[d] & 0 \\
0 \ar[r] & \CF'/\CF'' \ar[r]^{\iota'} & \CS'/\CS''  \ar[r]^{\text{inclusion} \circ \pi'} & \CF/\CF'' \ar[r] & \CF/\CF' \ar[r] & 0}
\end{equation}
we infer that \eqref{eqn:ofof} is equal to the extension on the bottom row. However, the assumption that \eqref{eqn:ww} and \eqref{eqn:www} map to the same tangent vector $v \in \Tan_{x_1} S = \Tan_{x_2} S$ implies that $\CS'/\CS'' \cong \CS/\CS'$, hence the bottom row of \eqref{eqn:ofofof} is equal to:\footnote{A non-split length 2 sheaf $A$ supported at a point $x$ of a smooth surface is an extension $0 \rightarrow \BC_x \rightarrow A \rightarrow \BC_x \rightarrow 0 $, which is completely determined by how the two generators of $\fm_x/\fm_x^2$ act on $A$. If $A$ and $B$ are two such extensions, then we obtain two rank 1 maps $A \rightarrow B$ and $B \rightarrow A$ by projecting through the quotients $\BC_x$. It is easy to show that $0 \rightarrow \BC_x \rightarrow A \rightarrow B \rightarrow \BC_x \rightarrow 0$ and $0 \rightarrow \BC_x \rightarrow B \rightarrow A \rightarrow \BC_x \rightarrow 0$ are equal in the 1-dimensional extension group $\Ext^2(\BC_x, \BC_x)$}
\begin{equation}
\label{eqn:ofofofof}
0 \longrightarrow \CF'/\CF'' \longrightarrow \CF/\CF'' \xrightarrow{\iota \circ \text{projection}} \CS/\CS' \stackrel{\pi}\longrightarrow  \CF/\CF' \longrightarrow 0
\end{equation}
in the 1-dimensional extension group $\Ext^2(\BC_{x_1},\BC_{x_2})$. Since \eqref{eqn:ofofofof} is nothing but the map \eqref{eqn:of} applied to the short exact sequence: 
$$
0 \longrightarrow \CF/\CF' \cong \BC_{x_1} \longrightarrow \CS/\CS' \longrightarrow \CF/\CF' \cong \BC_{x_1} \longrightarrow 0
$$
and this sequence represents the vector $v \in \Tan_x S$, we are done.

\end{proof}

\begin{corollary}
	
The scheme $\bfZ_2$ is l.c.i. of dimension \eqref{eqn:exp dim}. \\
	
\end{corollary}

\begin{proof} By Claim \ref{claim:tan}, we may stratify $\bfZ_2$ into two locally closed subsets: the locus $A$ of points where the tangent space has expected dimension \eqref{eqn:exp dim}, and the locus $B$ of points where the tangent space has dimension 1 bigger than expected. It is clear that $\dim A$ is no greater than the right-hand side of \eqref{eqn:exp dim}, so it remains to show that the same is true of $\dim B$. However, recall that Claim \ref{claim:tan} implies:
$$
B \subset \bfZ_2^\bullet \quad \Rightarrow \quad \dim B \leq \dim \bfZ_2^\bullet = \text{RHS of \eqref{eqn:exp dim fine}} < \text{RHS of \eqref{eqn:exp dim}}
$$	

\end{proof}

\section{The double shuffle algebra action}
\label{sec:shuf acts}

\medskip

\subsection{} 
\label{sub:stacks}


Let $\ks = K_0(S)$ be the algebraic $K$-theory ring of the smooth surface $S$. The correspondences $\fZ_1$ of \eqref{eqn:def z1 proj} induce operators on the $K$--theory groups:
$$
\km = \bigoplus_{c_2 = \left \lceil \frac {r-1}{2r} c_1^2 \right \rceil}^\infty K_0(\CM_{(r,c_1,c_2)})
$$
explicitly given by the following formulas (with the notations in \eqref{eqn:projection maps dg}):
\begin{align}
&\km \stackrel{e_d}\longrightarrow \kmms, & &e_d =  (p_+ \times p_S)_*\left( \CL^d \cdot p_-^* \right) \label{eqn:def e} \\
&K_{\CM'} \stackrel{f_d}\longrightarrow \kms, & &f_d = \frac {\det \CU}{(-1)^r q^{r-1}} \cdot  (p_- \times p_S)_* \left( \CL^{d-r} \cdot p_+^* \right) \label{eqn:def f}
\end{align}
where $q$ denotes the $K$-theory class of the canonical line bundle of $S$ (and its pull-back to $\CM \times S$ via the second projection). Let us focus on the operators $\{e_d\}_{d\in \BZ}$. Compositions of $k$ such operators can be described via the following correspondence:
\begin{equation}
\label{eqn:long fine}
\fZ_k := \fZ_1 \underset{\CM_1}\times \fZ_1 \underset{\CM_2}\times ... \underset{\CM_{k-1}}\times \fZ_1 \longrightarrow \CM_0 \times ... \times \CM_k
\end{equation}
(the fiber product is derived) which is a dg scheme supported on:
\begin{equation}
\label{eqn:long flag}
\bfZ_k = \Big \{\CF_0,...,\CF_k \text{ sheaves}, \text{ points } x_1,...,x_k \in S \text{ s.t. } \CF_0 \subset_{x_1} ... \subset_{x_k} \CF_k \Big\} 
\end{equation}
(the notation above is that $\CM_i$ denotes the moduli space that parametrizes the stable sheaf denoted by $\CF_i$). Both $\fZ_k$ and $\bfZ_k$ are endowed with line bundles $\{\CL_i\}_{1 \leq i \leq k}$, which parametrize the length 1 quotients $\CF_i/\CF_{i-1}$. We have projection maps:
$$
\xymatrix{
& \fZ_k \ar[ld]_{p_+}\ar[d]^{p_{S^k}}  \ar[rd]^{p_-}  & \\
\CM_0 & S^k & \CM_k}  \qquad \xymatrix{
& \bfZ_k \ar[ld]_{\bp_+}\ar[d]^{\bp_{S^k}}  \ar[rd]^{\bp_-}  & \\
\CM_0 & S^k & \CM_k} 
$$
where $\bp_+$, $\bp_-$, $\bp_{S^k}$ send a flag of sheaves \eqref{eqn:long flag} to $\CF_0$, $\CF_k$, $(x_1,...,x_k)$, respectively, and they are compatible with $p_+, p_-, p_{S^k}$ under the support map. Then we have:
\begin{equation}
\label{eqn:composition 0}
e_{d_1} \circ ... \circ e_{d_k} = (p_+ \times p_{S^k})_* \left(\CL_1^{d_1}...\CL_k^{d_k} \cdot p_-^* \right) : \km \longrightarrow K_{\CM \times S^k}
\end{equation}
Certain quadratic relations (the case $k=2$) between the compositions \eqref{eqn:composition 0} were worked out in \cite{Univ}, but the full set of relations was conjectured in \loccit to match those in the so-called universal shuffle algebra. However, we do not know how to describe the full ideal of relations explicitly. The situation is simplified if we compose \eqref{eqn:composition 0} with restriction to the smallest diagonal $\Delta : S \hookrightarrow S^k$:
\begin{equation}
\label{eqn:composition}
e_{d_1} \circ ... \circ e_{d_k} \Big|_\Delta : \km \longrightarrow K_{\CM \times S}
\end{equation}
The idea that \eqref{eqn:composition} is the ``composition" of the operators $e_{d_1},...,e_{d_k} : \km \rightarrow \kms$ leads to the following notion. Consider the ring homomorphism:
\begin{equation}
\label{eqn:constants}
\kk = \BZ[q_1^{\pm 1}, q_2^{\pm 1}]^{\text{symmetric in }q_1,q_2} \xrightarrow{\phi} K_S
\end{equation}
given by sending $q_1$ and $q_2$ to the Chern roots of $\Omega_S^1$. Let $S \stackrel{\Delta}\hookrightarrow S \times S$ be the diagonal. \\

\begin{definition}
\label{def:shuf acts}

An ``action" of the double shuffle algebra $\CA$ from \eqref{eqn:def double} on $\km$ refers to an abelian group homomorphism:
\begin{equation}
\label{eqn:action}
\CA \Big|_{c \mapsto q^r} \stackrel{\Phi}\longrightarrow \emph{Hom}(\km, \kms)
\end{equation}
satisfying the following properties, for any $x,y \in \emph{Hom}(\km, \kms)$ and $\gamma \in \BK$: \\

\begin{enumerate}
	
\item $\Phi(1) = \emph{proj}_1^*$, where $\emph{proj}_1 : \CM\times S \rightarrow \CM$ is the natural projection map \\

\item $\Phi(\gamma x)$ coincides with the composition:
\begin{equation}
\label{eqn:const}
\km  \xrightarrow{\Phi(x)} \kms \xrightarrow{\emph{Id}_{\CM} \times \emph{multipilication by }\phi(\gamma)} \kms
\end{equation}

\item $\Phi(xy)$ coincides with the composition:
\begin{equation}
\label{eqn:hom}
\km \xrightarrow{\Phi(y)} \kms \xrightarrow{\Phi(x) \times \emph{Id}_S} \kmss \xrightarrow{\emph{Id}_\CM \times \Delta^*} \kms
\end{equation}

\item $\Delta_* \Phi \left( \frac {[x,y]}{(1-q_1)(1-q_2)} \right)$ coincides with the difference of compositions:
\begin{align*}
&\km \xrightarrow{\Phi(y)} \kms \xrightarrow{\Phi(x) \times \emph{Id}_S} \kmss \\
&\km \xrightarrow{\Phi(x)} \kms \xrightarrow{\Phi(y) \times \emph{Id}_S} \kmss \xrightarrow{\emph{Id}_{\CM} \times \emph{swap}^*} \kmss
\end{align*}
where $\emph{swap} : S \times S \rightarrow S \times S$ is the permutation map. \\

\end{enumerate}

\noindent Although in the formulas above we use the language of operators for conciseness, we postulate that $\Phi(x)$ is induced by a correspondence in $K_{\CM \times \CM \times S}
$ for any $x \in \CA$. For various algebraic spaces $X,Y,Z$, given an operator $f : K_X \rightarrow K_Y$, the notation:
\begin{equation}
\label{eqn:f times id}
f \times \emph{Id}_Z : K_{X \times Z} \rightarrow K_{Y \times Z}
\end{equation}
should be interpreted as follows: consider the correspondence $\Gamma \in K_{X \times Y}$ which represents $f$, and then define \eqref{eqn:f times id} as the operator induced by the correspondence:
$$
\emph{proj}^*_{13}(\Gamma) \cdot \emph{proj}^*_{24}(\Delta_Z) \in K_{X \times Z \times Y \times Z}
$$

\medskip

\end{definition}

\noindent Conjecture \ref{conj:shuf acts} states that there exists a homomorphism \eqref{eqn:action} satisfying properties \emph{(1)--(4)} above, and we will now show how to construct $\Phi(E_{n,k})$, $\forall (n,k) \in \BZ^2 \backslash (0,0)$. \\

\begin{remark}
\label{rem:sv}

There are general reasons why one expects the shuffle algebra to act on $\km$: as shown by Schiffmann-Vasserot in \cite{SV}, their $K$--theoretic Hall algebra $\CH$ naturally acts on groups akin to $\km$. There is a map $\Upsilon: \CH \rightarrow \CS_\ebig$ that arises from equivariant localization on the stack of finite length sheaves on $S$, and it is natural to conjecture that the image $\Upsilon(\CH)$ also acts on the groups $\km$ (see \cite{Min} for a detailed and rigorous treatment in a setup very close to ours, which results in a similar shuffle algebra). However, it is not clear how to prove that the map $\Upsilon$ is injective, or less ambitiously, that the kernel of this map acts trivially on $\km$. \\

\end{remark}

\subsection{} 
\label{sub:explicit action}

As suggested in \cite{Univ}, we expect that the ``action" $\Phi$ should be given by sending:
\begin{align}
&\Phi \left( z_1^d \in \CA^\leftarrow \right) = e_d \label{eqn:sending 1} \\
&\Phi \left( z_1^d \in \CA^\rightarrow \right) = f_d \label{eqn:sending 2}
\end{align} 
$\forall d \in \BZ$, where we implicitly use the isomorphism $\CA^\leftarrow \cong \CS$ and $\CA^\rightarrow \cong \CS^{\op}$, and:
$$
\Phi \left( E_{0, \pm k} \in \CA^{\text{diag}} \right) = \text{operator of multiplication by } \oint z^{\pm k} \wedge^\bullet \left(\frac {\CU^{\pm 1}}{z^{\pm 1} q^{-\delta_\pm^-}} \right)^{\pm 1}
$$ 
$\forall k \in \BN$. For example, when the sign is $\pm = +$ in the formula above, the contour integral just picks up the coefficient of $z^{- k}$ in the expansion of the rational function: 
$$
\wedge^\bullet \left(\frac {\CU}z \right)  = \frac {\wedge^\bullet \left(\frac {\CV}z \right)}{\wedge^\bullet \left(\frac {\CW}z \right)}
$$
where $\CV$, $\CW$ are the vector bundles of \eqref{eqn:length 1}. Let us focus on the operators \eqref{eqn:sending 1}. By property \emph{(3)} of Definition \ref{def:shuf acts}, the composition \eqref{eqn:composition} should correspond to:
$$
\Phi \left(\sym \left[ z_1^{d_1} ... z_k^{d_k} \prod_{1\leq i < j \leq k} \zeta \left( \frac {z_i}{z_j} \right) \right] \in \CA^\leftarrow \right) = e_{d_1} \circ ... \circ e_{d_k} \Big|_\Delta
$$
However, to completely determine the action \eqref{eqn:action}, we must conjecture how the generators $E_{d_\bullet}$ of \eqref{eqn:shuffle gen} act. To do this, let us introduce some new correspondences: \\

\begin{definition}
\label{def:long fine}

Consider the following dg scheme, obtained by chains of derived fiber products of the dg scheme \eqref{eqn:fine} via the maps \eqref{eqn:pi-} and \eqref{eqn:pi+}:
$$
\fZ_k^{\bullet} = \fZ_2^{\bullet} \underset{\fZ_1}\times \fZ_2^{\bullet} \underset{\fZ_1}\times ... \underset{\fZ_1}\times \fZ_2^{\bullet} \longrightarrow \CM_0 \times ... \times \CM_k
$$
which will be supported on the scheme:
\begin{equation}
\label{eqn:long flag fine}
\bfZ_k^{\bullet} = \left \{\CF_0,...,\CF_k \text{ sheaves}, x \in S \text{ such that } \CF_0 \subset_{x} ... \subset_{x} \CF_k \right\} 
\end{equation}
There are line bundles $\CL_1,...,\CL_k$ on $\fZ_k^{\bullet}$ that correspond to the line bundles on $\bfZ_k^\bullet$ that parametrize the quotients $\CF_1/\CF_0,..., \CF_k/\CF_{k-1}$, as well as projection maps:
\begin{equation}
\label{eqn:diag fine}
\xymatrix{
& \fZ_k^{\bullet} \ar[ld]_{p_+} \ar[d]^{p_S} \ar[rd]^{p_-} & \\
\CM_0 & S & \CM_k} \qquad \xymatrix{
& \bfZ_k^{\bullet} \ar[ld]_{\bp_+} \ar[d]^{\bp_S} \ar[rd]^{\bp_-} & \\
\CM_0 & S & \CM_k} 
\end{equation}
For any vector of integers $d_\bullet = (d_1,...,d_k)$, define the operators:
\begin{align}
&K_{\CM} \stackrel{e_{d_\bullet}}\longrightarrow K_{\CM \times S} \label{eqn:def e fine 0} \\
&K_{\CM} \stackrel{f_{d_\bullet}}\longrightarrow K_{\CM \times S} \label{eqn:def f fine 0}
\end{align}
by the following formulas:
\begin{align}
&e_{d_\bullet} = (p_+ \times p_S)_*\left( \CL_1^{d_1}... \CL_k^{d_k} \cdot p_-^* \right) \label{eqn:def e fine} \\
&f_{d_\bullet} = \frac {(\det \CU)^{\otimes k}}{(-1)^{kr} q^{k(r-1)}} \cdot (p_- \times p_S)_* \left( \CL_1^{d_1 - r}... \CL_k^{d_k - r} \cdot p_+^* \right) \label{eqn:def f fine}
\end{align}

\end{definition}

\noindent We make Conjecture \ref{conj:shuf acts} more precise by stipulating that the elements $E_{d_\bullet}, F_{d_\bullet} \in \CA$ from \eqref{eqn:gen left} and \eqref{eqn:gen right} should act on $K$--theory groups via the formulas:
\begin{align}
&\Phi \left( E_{d_\bullet} \right) = e_{d_\bullet} \label{eqn:generators act 1} \\
&\Phi \left( F_{d_\bullet} \right) = f_{d_\bullet} \label{eqn:generators act 2}
\end{align}
for all $d_\bullet = (d_1,...,d_k) \in \BZ^k$. Note that $E_{d_\bullet}$ and $F_{d_\bullet}$ generate the algebras $\CA^\leftarrow$ and $\CA^\rightarrow$, because these algebras are isomorphic to $\CS$ and $\CS^{\op}$, respectively. \\

\subsection{} Consider the following particular cases of the operators \eqref{eqn:def e fine} and \eqref{eqn:def f fine}:
\begin{align}
&e_{-k,n} = q^{\gcd(k,n)-1} e_{(d_1,...,d_k)} \label{eqn:important 1} \\
&e_{k,n} = q^{\gcd(k,n)-1} f_{(d_1,...,d_k)} \label{eqn:important 2}
\end{align}
for all $k>0$ and $n \in \BZ$, where $d_i = \left \lceil {\frac {ni}k} \right \rceil - \left \lceil {\frac {n(i-1)}k} \right \rceil + \delta_i^k - \delta_i^1$, together with:
\begin{equation}
\label{eqn:important 3}
e_{0,\pm k} = \text{operator of multiplication by } \oint z^{\pm k} \wedge^\bullet \left(\frac {\CU^{\pm 1}}{z^{\pm 1} q^{-\delta_\pm^-}} \right)^{\pm 1}
\end{equation}
By \eqref{eqn:generators act 1} and \eqref{eqn:generators act 2}, the operators $e_{n,k}$ should correspond to the generators $E_{n,k}$ of the shuffle algebra from \eqref{eqn:shuf gen}. However, recall from Theorem \ref{thm:comm double} that the relations in the algebra $\CA|_{c \mapsto q^r}$ are generated by the following commutation relations:
\begin{equation}
\label{eqn:comm 0}
[E_{n,k}, E_{n',k'}] = \Delta \sum_{v \text{ convex path}} p^{n,n_1,...,n_t,n'}_{k,k_1,...,k_t,k'} (q_1,q_2) \cdot E_{n_1,k_1}... E_{n_t,k_t} 
\end{equation}
for some Laurent polynomials $p^{n,n_1,...,n_t,n'}_{k,k_1,...,k_t,k'} (q_1,q_2) \in \BK$, where the sum in the right-hand side of \eqref{eqn:comm 0} goes over all $\sum n_i = n+n', \sum k_i = k+k'$ such that the path:\footnote{The notation below means that $v$ is the broken line between $(0,0)$ and $(n_1+...+n_t,k_1+...+k_t)$, formed out of the line segments $(n_1,k_1)$,...,$(n_t,k_t)$ in this order}
$$
v = \overrightarrow{(n_1,k_1),...,(n_t,k_t)} 
$$ 
is sandwiched between the paths $v_1 = \overrightarrow{(n+n',k+k')}$ and $v_2 = \overrightarrow{(n',k'),(n,k)}$, and has the same convexity as the path $v_2$. Conjecture \ref{conj:shuf acts} implies that the analogous commutation relations hold between the operators \eqref{eqn:important 1}--\eqref{eqn:important 3}. \\


\begin{conjecture}
\label{conj:comm}


For any lattice points $(n,k)$ and $(n',k')$, we have:
\begin{equation}
\label{eqn:comm}
[e_{n,k}, e_{n',k'}] = \Delta_* \left( \sum_{v \emph{ convex path}} \phi(p^{n,n_1,...,n_t,n'}_{k,k_1,...,k_t,k'} (q_1,q_2)) \cdot e_{n_1,k_1}... e_{n_t,k_t} \Big|_\Delta \right)
\end{equation}
where $\phi : \BK \rightarrow K_S$ is the map \eqref{eqn:constants}. \\


\end{conjecture}

\noindent Let us explain the notation in \eqref{eqn:comm}. The LHS is an operator $\km \rightarrow \kmss$, with $e_{n,k}$ acting in the first factor of $S \times S$ and $e_{n',k'}$ acting in the second factor. Meanwhile, each summand in the RHS is the following composed operator:
\begin{multline*}
\km \xrightarrow{e_{n_t,k_t}} \kms \xrightarrow{e_{n_{t-1},k_{t-1}} \times \text{Id}_S} ... \xrightarrow{e_{n_1,k_1} \times \text{Id}^{k-1}_S} K_{\CM \times S \times ... \times S} \\  \xrightarrow{\text{Id}_\CM \times \Delta^*} \kms \xrightarrow{\text{Id}_{\CM} \times \text{multiplication by } \phi(p^{n,n_1,...,n_t,n'}_{k,k_1,...,k_t,k'} (q_1,q_2))} \kms 
\end{multline*}
and $\Delta_*$ in \eqref{eqn:comm} is shorthand for $(\text{Id}_{\CM} \times \Delta)_* : \kms \rightarrow \kmss$. \\



\begin{proposition}
\label{prop:equiv}

Conjectures \ref{conj:shuf acts} and \ref{conj:comm} are equivalent. \\

\end{proposition}

\begin{proof} It is clear that Conjecture \ref{conj:shuf acts} implies Conjecture \ref{conj:comm}, since formula \eqref{eqn:e triangle double}, together with properties \emph{(2), (3), (4)} of Definition \ref{def:shuf acts}, implies formula \eqref{eqn:comm}. Conversely, let us assume that \eqref{eqn:comm} holds and prove Conjecture \ref{conj:shuf acts}. By \eqref{eqn:anna double}, any element of $\CA|_{c \mapsto q^r}$ can be written as:
\begin{equation}
\label{eqn:x}
x = \sum^{n_1,...,n_t}_{k_1,...,k_t} p^{n_1,...,n_t}_{k_1,...,k_t}(q_1,q_2) \cdot E_{n_1,k_1}...E_{n_t,k_t}
\end{equation}
for some $p^{n_1,...,n_t}_{k_1,...,k_t}(q_1,q_2) \in \BK$, where the vectors $(n_1,k_1),...,(n_t,k_t)$ that appear in the right-hand side of \eqref{eqn:x} are ordered clockwise by slope. Properties \emph{(2)} and \emph{(3)} of Definition \ref{def:shuf acts} completely determine the action of $x$ on $\km$, since they force:
\begin{equation}
\label{eqn:phi}
\Phi(x) = \sum^{n_1,...,n_t}_{k_1,...,k_t} \phi(p^{n_1,...,n_t}_{k_1,...,k_t}(q_1,q_2)) \cdot e_{n_1,k_1}...e_{n_t,k_t} \Big|_\Delta
\end{equation}
Since $\phi$ is a ring homomorphism, $\Phi$ defined by \eqref{eqn:phi} satisfies property \emph{(2)} of Definition \ref{def:shuf acts}. As for property \emph{(3)}, note that restricting \eqref{eqn:comm} to the diagonal (together with the basic fact that $\Delta_*(\gamma)|_\Delta = \phi((1-q_1)(1-q_2))\gamma, \forall \gamma \in \ks$) implies:
\begin{multline} 
e_{n,k} e_{n',k'} |_\Delta - e_{n',k'} e_{n,k} |_\Delta = \\ = \sum_{v \text{ convex path}} \phi((1-q_1)(1-q_2)p^{n,n_1,...,n_t,n'}_{k,k_1,...,k_t,k'} (q_1,q_2)) \cdot e_{n_1,k_1}... e_{n_t,k_t} \Big|_\Delta \label{eqn:vlad} 
\end{multline} 
and relation \eqref{eqn:vlad} allows us to iteratively transform any product $e_{n_1,k_1}...e_{n_t,k_t}|_\Delta$ into a linear combination of such products where the vectors $(n_1,k_1),...,(n_t,k_t)$ are ordered by slope. Moreover, the specific coefficients that appear in the linear combination match the ones that transform the product $E_{n_1,k_1}...E_{n_t,k_t} \in \CA$ into a linear combination of such products where the vectors $(n_1,k_1),...,(n_t,k_t)$ are ordered by slope. This proves that $\Phi(x)\Phi(y)|_\Delta = \Phi(xy)$ for any $x,y$ of the form \eqref{eqn:x}, by showing that both $\Phi(x)\Phi(y)|_\Delta$ and $\Phi(xy)$ have the same expansion into products like in the right-hand side of \eqref{eqn:phi}. Thus property \emph{(3)} of Definition \ref{def:shuf acts} is established, and the only thing that we used was the associativity property:
\begin{equation}
\label{eqn:associativity}
\left( \left( X\circ Y \Big|_\Delta \right) \circ Z \right) \Big|_\Delta = \left(X \circ \left( Y\circ Z \Big|_\Delta \right)\right) \Big|_\Delta
\end{equation}
for any $\km \xrightarrow{X,Y,Z} \kms$. The fact that assignment \eqref{eqn:phi} satisfies property \emph{(4)} of Definition \ref{def:shuf acts} is proved analogously with the computation above, but instead of the associativity property \eqref{eqn:associativity} one uses the following version of the Leibniz rule:
\begin{equation}
\label{eqn:leibniz}
\left[X,Y \circ Z\Big|_\Delta \right]_{\text{red}} = [X,Y]_{\text{red}}\circ Z \Big|_\Delta + Y \circ [X,Z]_{\text{red}} \Big|_\Delta
\end{equation}
where if $[X,Y] = \Delta_*(A)$ we write $[X,Y]_{\text{red}} = A$ \footnote{Since projection is a left inverse for the diagonal $\Delta : S \rightarrow S \times S$, the map $\Delta_*$ is injective and thus the notation  $[X,Y]_{\text{red}}$ is unambiguous}. 

\end{proof}

\begin{remark} In Conjecture \ref{conj:comm}, it suffices to assume that $p^{n,n_1,...,n_t,n'}_{k,k_1,...,k_t,k'} (q_1,q_2)$ of \eqref{eqn:comm} are certain Laurent polynomials that do not depend on the surface $S$, because then one can deduce that they must coincide with the Laurent polynomials that appear in \eqref{eqn:comm 0}. The argument for this claim is to consider $S = \BA^2$ and replace usual $K$--theory by equivariant $K$--theory with respect to the standard torus action $\BC^* \times \BC^* \curvearrowright \BA^2$ ($q_1$ and $q_2$ correspond to the equivariant parameters) and replace $\CM$ by the moduli space of framed sheaves on $\BP^2$. If this is the case, we showed in \cite{Mod} that Conjecture \ref{conj:shuf acts} holds and the generators $E_{n,k}$ correspond to the operators $e_{n,k}$, thus implying that the coefficients in \eqref{eqn:comm 0} match those in \eqref{eqn:comm}. \\
	
\end{remark} 

\subsection{} 
\label{sub:ass b}

If one is willing to accept the equivalent Conjectures \ref{conj:shuf acts} and \ref{conj:comm}, one can skip to the next Section. The remainder of the current Section is devoted to the proof of both Conjectures subject to the following quite strong assumption: \\

\noindent \textbf{Assumption B: } \emph{The Kunneth decomposition $\kss \cong \ks \boxtimes \ks$ holds, and the class of the diagonal splits up as:}
\begin{equation}
\label{eqn:diag decomp}
[\CO_\Delta] = \sum_c l_c \boxtimes l^c \in \ks \boxtimes \ks
\end{equation}
\emph{Following \cite{CG}, this implies $\kms \cong \km \boxtimes \ks$, and so we may decompose the class of the universal sheaf as:}
\begin{equation}
\label{eqn:univ decomp}
[\CU] = \sum_c [\CT_c] \boxtimes l^c \in \km \boxtimes \ks
\end{equation}
\emph{Finally, we assume that $\{\CT_c\}$ which appear in \eqref{eqn:univ decomp} are actually locally free sheaves on $\CM$ whose exterior powers generate $K_{\CM_{(r,c_1,c_2)}}$ as a ring, for any $c_2 \in \BZ$}. \\

\noindent Assumption B holds for $S = \BP^2$, as shown in \cite[Example 3.12]{Univ}; we expect that the proof provided therein extends to toric surfaces $S$, although we will not prove this fact. Since we will only be concerned with $K$-theory classes in the remainder of this Section, we will generally abuse notation by writing $\CU$, $\CT_c$ for the $K$-theory classes of the coherent sheaves in question. Note that under assumption \eqref{eqn:diag decomp}, the collections $\{l_c\}$ and $\{l^c\}$ yield dual $\BZ$-bases of $\ks$, with respect to the bilinear form:
$$
\ks \otimes \ks \rightarrow \BZ, \qquad \left([\CV_1], [\CV_2] \right) = \chi \left(S, \CV_1 \otimes \CV_2 \right)
$$
The last statement of Assumption B says that any element of $\km$ is of the form:
\begin{equation}
\label{eqn:psi taut}
\Psi(...,\CT_c,...)
\end{equation}
where $\Psi$ is a Laurent polynomial, symmetric in the Chern roots of each vector bundle $\CT_c$ separately (if we dropped the condition that $\km$ is generated by the $\CT_c$'s, as we did in \cite{Univ}, then we would be only constructing an action of $\CA$ on the subalgebra $K'_{\CM} \subseteq \km$ generated by the $\CT_c$'s). Under Assumption B, \cite[Lemma 3.14]{Univ} shows that the operators \eqref{eqn:def e} and \eqref{eqn:def f} act by:
\begin{align}
&e_d \Psi(...,\CT_c,...) = \oi z^d \Psi(...,\CT_c + z l_c,...) \wedge^\bullet \left(\frac {zq}{\CU} \right) \label{eqn:one 0} \\
&f_d \Psi(...,\CT_c,...) = \i \frac {z^d}{q^{r-1}} \Psi(...,\CT_c - z l_c,...) \wedge^\bullet \left( - \frac {z}{\CU} \right) \label{eqn:two 0}
\end{align}
where:
\begin{equation}
\label{eqn:residue def} 
\oi F(z) = - \i F(z) = \text{Res}_{z = 0} \frac {F(z)}z - \text{Res}_{z = \infty} \frac {F(z)}z
\end{equation}
Iterating \eqref{eqn:one 0} and \eqref{eqn:two 0} yields:
\begin{multline}
\label{eqn:one 1}
\quad e_{d_1} ... e_{d_k} \Psi(...,\CT_c,...) = \\ = \int_{0 - \infty}^{z_1 \prec ... \prec z_k} \frac{z_1^{d_1} ... z_k^{d_k} \Psi(...,\CT_c + \sum_{i=1}^k z_i l_c^{(i)},...)}{\prod_{1 \leq i < j \leq k} \zeta_{ij} \left(\frac {z_j}{z_i} \right)} \prod_{i=1}^k \wedge^\bullet \left(\frac {z_iq^{(i)}}{\CU^{(i)}} \right)
\end{multline}
\begin{multline}
\label{eqn:two 1}
\quad f_{d_k}... f_{d_1} \Psi(...,\CT_c,...) = \\ = \int_{\infty - 0}^{z_k \prec ... \prec z_1} \frac{z_1^{d_1} ... z_k^{d_k} \Psi(...,\CT_c - \sum_{i=1}^k z_i l_c^{(i)},...)}{q^{k(r-1)} \prod_{1 \leq i < j \leq k} \zeta_{ij} \left(\frac {z_j}{z_i} \right)} \prod_{i=1}^k \wedge^\bullet \left(- \frac {z_i}{\CU^{(i)}} \right)
\end{multline}
where $z_i \prec z_j$ means that we first compute the residues at $0$ and $\infty$ in the variable $z_j$, and then in the variable $z_i$ (more graphically, we think of $z_j$ as being ``larger", in the sense of being closer to the singularities at $\infty$ and 0, than $z_i$). The right-hand sides of expressions \eqref{eqn:one 1} and \eqref{eqn:two 1} take values in:
\begin{equation}
\label{eqn:bili}
K_{\CM \times S \times ... \times S} \cong \km \boxtimes \ks \boxtimes ... \boxtimes \ks
\end{equation}
and $l_c^{(i)}, q^{(i)}, \CU^{(i)} $ refer to the elements of \eqref{eqn:bili} obtained by pulling back $l_c, q, \CU$ via the $i$-th projection $S \times ... \times S \rightarrow S$. Finally, in formulas \eqref{eqn:one 1} and \eqref{eqn:two 1} we let:
\begin{equation}
\label{eqn:zeta ij}
\zeta_{ij} (x) = \wedge^{\bullet}(-x \cdot \CO_{\Delta_{ij}} ) \in K_{\CM \times S \times ... \times S}(x)
\end{equation}
denote the pull-back of $\wedge^{\bullet}(-x \cdot \CO_{\Delta} ) \in \kss(x)$ via the $(i,j)$ projection. \\

\subsection{} Formulas \eqref{eqn:one 1}--\eqref{eqn:two 1} are somewhat complicated, but they become simpler if one restricts the right-hand side to the smallest diagonal $\Delta : S \hookrightarrow S \times ... \times S$:
\begin{multline}
\quad e_{d_1} ... e_{d_k} \Big|_\Delta \Psi(...,\CT_c,...) = \label{eqn:one 2} \\
= \int_{0 - \infty}^{z_1 \prec ... \prec z_k} \frac{z_1^{d_1} ... z_k^{d_k} \Psi(...,\CT_c + (z_1+...+z_k) l_c,...)}{\prod_{1 \leq i < j \leq k} \zeta \left(\frac {z_j}{z_i} \right)} \prod_{i=1}^k \wedge^\bullet \left(\frac {z_iq}{\CU} \right)
\end{multline}
\begin{multline}
\quad f_{d_k}... f_{d_1}  \Big|_\Delta \Psi(...,\CT_c,...) = \label{eqn:two 2} \\ = \int_{\infty - 0}^{z_k \prec ... \prec z_1} \frac{z_1^{d_1} ... z_k^{d_k} \Psi(...,\CT_c - (z_1+...+z_k) l_c,...)}{q^{k(r-1)} \prod_{1 \leq i < j \leq k} \zeta \left(\frac {z_j}{z_i} \right)} \prod_{i=1}^k \wedge^\bullet \left(- \frac {z_i}{\CU} \right)
\end{multline}
where:
\begin{equation}
\label{eqn:zeta geom}
\zeta(x) = \wedge^\bullet(-x \cdot \CO_\Delta) \Big |_\Delta = \frac {\wedge^\bullet(x \cdot \Omega_S^1)}{(1-x)(1-x q)} \in \ks(x)
\end{equation}
Note that the formula above matches \eqref{eqn:zeta univ}, as long as we set the Chern roots of $\Omega_S^1$ equal to $q_1$ and $q_2$. By analogy with \eqref{eqn:one 2} and \eqref{eqn:two 2}, we have the following: \\

\begin{proposition}
\label{prop:act}

The operators of Definition \ref{def:long fine} act on tautological classes by:
\begin{multline}
\quad e_{(d_1,...,d_k)} \Psi(...,\CT_c,...) = \label{eqn:one} \\ 
= \int_{0 - \infty}^{z_1 \prec ... \prec z_k} \frac{z_1^{d_1} ... z_k^{d_k} \Psi(...,\CT_c + (z_1+...+z_k) l_c,...)}{\left(1 - \frac {z_2 q}{z_1} \right) ... \left(1 - \frac {z_k q}{z_{k-1}} \right) \prod_{i < j} \zeta \left(\frac {z_j}{z_i} \right)} \prod_{i=1}^k \wedge^\bullet \left(\frac {z_iq}{\CU} \right)
\end{multline}
\begin{multline}
f_{(d_k,...,d_1)} \Psi(...,\CT_c,...) = \label{eqn:two} \\
= \int_{\infty - 0}^{z_k \prec ... \prec z_1} \frac{z_1^{d_1} ... z_k^{d_k} \Psi(...,\CT_c - (z_1+...+z_k) l_c,...)}{q^{k(r-1)}\left(1 - \frac {z_2 q}{z_1} \right) ... \left(1 - \frac {z_k q}{z_{k-1}} \right) \prod_{i < j} \zeta \left(\frac {z_j}{z_i} \right)} \prod_{i=1}^k \wedge^\bullet \left(- \frac {z_i}{\CU} \right)
\end{multline}
for all $d_1,...,d_k \in \BZ$. \\

\end{proposition}

\begin{proof} We will prove \eqref{eqn:one}, and leave the analogous formula \eqref{eqn:two} as an exercise to the interested reader. In virtue of the definition of the operator $e_{(d_1,...,d_k)}$ in \eqref{eqn:def e fine}, the left-hand side of \eqref{eqn:one} is given by:
\begin{equation}
\label{eqn:gaga}
(p_+ \times p_S)_* \left(\CL_1^{d_1}... \CL_k^{d_k} \cdot \Psi(..., p^*_- (\CT_c),...) \right) 
\end{equation}
where the maps $p_+$, $p_-$, $p_S$ are defined in \eqref{eqn:diag fine}. However, note the equality:
\begin{equation}
\label{eqn:intermediate}
(p_- \times \text{Id}_S)^*(\CU) = (p_+ \times \text{Id}_S)^*(\CU) + \Gamma_*(\CL_1) + ... + \Gamma_*(\CL_k)
\end{equation}
in $K_{\fZ_k^\bullet \times S}$, where $\Gamma$ denotes the graph of $p_S : \fZ_k^\bullet \rightarrow S$. From \eqref{eqn:univ decomp}, we have:
\begin{equation}
\label{eqn:tc}
\CT_c = \text{proj}_{1*} \left(\CU \otimes \text{proj}_2^* (l_c)\right)
\end{equation}
where $\text{proj}_1$ and $\text{proj}_2$ denote the projections from $\CM \times S$ onto the two factors. Tensoring relation \eqref{eqn:intermediate} with $l_c$ pulled back from the second factor of $\fZ_k^{\bullet} \times S$ and pushing forward to the first factor, we conclude that:
$$
p_-^*(\CT_c) = p_+^*(\CT_c) + (\CL_1 + ... + \CL_k) \cdot p_S^*(l_c)
$$
With this relation in hand, \eqref{eqn:gaga} equals:
\begin{equation}
\label{eqn:gaga 2}
(p_+ \times p_S)_* \left(\CL_1^{d_1}... \CL_k^{d_k} \cdot \Psi(..., p^*_+ (\CT_c) + (\CL_1+...+\CL_k)p_S^*(l_c),...) \right) 
\end{equation}
Now we need to recall the description of the map $p_+ \times p_S : \fZ_k^{\bullet} \rightarrow \CM_0 \times S$, where $\CM_0$ is the copy of the moduli space of stable sheaves that parametrizes the sheaf denoted by $\CF_0$ in \eqref{eqn:long flag fine}. By Definition \ref{def:long fine}, we have a derived fiber square:
\begin{equation}
\label{eqn:fiber square}
\xymatrix{\fZ_k^{\bullet} \ar[d] \ar[r]^-{p_+^{(k)}} & \fZ_{k-1}^{\bullet} \ar[d] \\
\fZ_2^\bullet \ar[r]^-{\pi_+} & \fZ_1}
\end{equation}
where $p_+^{(k)}$ and $\pi_+$ are the obvious maps of dg schemes which lie above the following maps between their support schemes:
$$
\bp_+^{(k)}(\CF_0 \subset_x ... \subset_x \CF_{k-2} \subset_x \CF_{k-1} \subset_x \CF_k) = (\CF_0 \subset_x ... \subset_x \CF_{k-2} \subset_x \CF_{k-1})
$$
and:
$$
\bpi_+(\CF_{k-2} \subset_x \CF_{k-1} \subset_x \CF_k) = (\CF_{k-2} \subset_x \CF_{k-1})
$$
Note that the line bundles $\CL_1,...,\CL_{k-1}$ on $\fZ_k^{\bullet}$ are pulled back via $p_+^{(k)}$, while the line bundle $\CL_k$ is pulled back from $\fZ_2^\bullet$ via the vertical map. We will use the notation $\CU_0,..., \CU_k$ for the universal sheaves on $\fZ_k^\bullet \times S$ that correspond to $\CF_0,...,\CF_k$. Since: 
$$
p_+ = p_+^{(1)} \circ ... \circ p_+^{(k)}
$$
in order to prove that \eqref{eqn:gaga 2} equals the right-hand side of \eqref{eqn:one}, it is enough to iterate the following equation $n$ times:
\begin{equation}
\label{eqn:vir}
p^{(k)}_{+*} \left( F(\CL_k) \right) = \oi F(z) \frac {\wedge^\bullet \left( \frac {zq}{\CU_{k-1}} \right)}{1 - \frac {z q}{\CL_{k-1}}} = \oi \frac {F(z)  \cdot \wedge^\bullet \left( \frac {zq}{\CU_0} \right)}{\left( 1 - \frac {z q}{\CL_{k-1}} \right) \prod_{i=1}^{k-1} \zeta \left( \frac z{\CL_i} \right) } 
\end{equation}
where $F(z)$ is any function of $z$ with coefficients in $\text{Im }p_+^{(k)*}$. The first equality above is a direct application of Proposition 5.19 of \cite{Univ}, which describes the push-forward of tautological classes down projectivizations of coherent sheaves of projective dimension 1. Indeed, by \eqref{eqn:proj+}, the bottom map $\pi_+$ of the derived fiber square \eqref{eqn:fiber square} is the projectivization of a  coherent sheaf of projective dimension 1, whose class in $K$--theory is  $-q[\CU_{k-1}^\vee] + q[\CL_{k-1}^{\vee}]$. By base change, so is the top map in \eqref{eqn:fiber square}. The second equality of \eqref{eqn:vir} is a consequence of $\CU_{k-1} = \CU_0 + \sum_{i=1}^{k-1} \CL_i \otimes (p_S \times \text{Id})^*(\CO_\Delta)$. 

\end{proof}

\subsection{}
\label{sub:proof ass b}

As observed in Remark 3.17 of \cite{Univ}, one can redefine the integrals \eqref{eqn:one 1} and \eqref{eqn:two 1} in such a way that $|z_1| = ... = |z_n|$. This is achieved by formally assuming that the Chern roots of $\Omega^1_S$ are complex numbers of absolute value either $<1$ or $>1$, evaluating an $n$-fold integral by residues under this assumption, and acknowledging that the result is a Laurent polynomial in $q_1$ and $q_2$ and thus a well-defined $K$--theory class. The same definition applies to \eqref{eqn:one 2} and \eqref{eqn:two 2} by restriction, and also to \eqref{eqn:one} and \eqref{eqn:two} since the factors $z_{i} - z_{i+1} q$ in the denominator of the latter expressions do not create any new poles (they are canceled by certain factors of $\zeta$). 

\tab 
This procedure of moving the contours does not lead to more helpful formulas for the operators \eqref{eqn:one 1}--\eqref{eqn:two}, but it does have the following important consequence: these operators vanish not only when the integrand vanishes as a function of $z_1,...,z_n$, but when the symmetrization of the integrand vanishes. In other words:
$$
\sum c_{d_1,...,d_n} \cdot \sym \left[ \frac {z_1^{d_1}...z_n^{d_n}}{\prod_{i=1}^{n-1} \left(1-\frac {z_{i+1}q}{z_i} \right)\prod_{i < j} \zeta \left( \frac {z_j}{z_i} \right)} \right] = 0 \quad \Longrightarrow 
$$
\begin{equation}
\label{eqn:principle}
\Longrightarrow \quad \sum \phi(c_{d_1,...,d_n}) \cdot e_{(d_1,...,d_n)} = 0
\end{equation}
for any collection of elements $c_{d_1,...,d_n} \in \BK$ (recall the map $\phi$ of \eqref{eqn:constants}), where $\sym$ denotes averaging with respect to all $n!$ permutations of the variables $z_1,...,z_n$. As a consequence of this fact, formulas \eqref{eqn:left} and \eqref{eqn:right} imply the following equalities:
\begin{align}
&e_{d_\bullet} = \mathop{\sum^{n_1+...+n_t = n}_{\frac {k_1}{n_1} \leq ... \leq \frac {k_t}{n_t}}}^{n_i \in \BN, k_i \in \BZ} \phi(s_{k_1,...,k_t,d_\bullet}^{n_1,...,n_t}(q_1,q_2)) \cdot  e_{-n_1,k_1}... e_{-n_t,k_t} \Big|_\Delta \label{eqn:geom left} \\
&f_{d_\bullet} = \mathop{\sum^{n_1+...+n_t = n}_{\frac {k_1}{n_1} \geq ... \geq \frac {k_t}{n_t}}}^{n_i \in \BN, k_i \in \BZ} \phi(s_{k_1,...,k_t,d_\bullet}^{n_1,...,n_t}(q_1,q_2)) \cdot e_{n_1,k_1}... e_{n_t,k_t}  \Big|_\Delta \label{eqn:geom right}
\end{align} 
as operators $\ks \rightarrow \kms$, for any $d_\bullet = (d_1,...,d_n) \in \BZ^n$. \\


\begin{proof} \emph{of Conjectures \ref{conj:shuf acts} and \ref{conj:comm} subject to Assumption B:} as shown in Proposition \ref{prop:equiv}, the two conjectures are equivalent, so we only need to prove relation \eqref{eqn:comm}. Looking at formulas \eqref{eqn:one} and \eqref{eqn:two}, it is clear that our task involves commuting two operators that are defined by successive contour integrals. In all significant cases, the commutator will be shown to involve a difference of the following kind:
\begin{equation}
\label{eqn:residue thm 1}
\int_{0 - \infty}^{z \prec w} F(z,w) - \int_{0 - \infty}^{w \prec z} F(z,w)
\end{equation}
where $F(z,w)$ is a rational function (valued in $K$-theory classes on $\CM \times S \times S$), whose poles are all of the form $z - w\alpha$ for various scalars $\alpha \notin \{0,\infty\}$. Then the key idea is that the difference \eqref{eqn:residue thm 1} equals (see also \eqref{eqn:residue equality}):
\begin{equation}
\label{eqn:residue thm 2}
- \sum_{\alpha \neq \{0,\infty\}} \int_{0 - \infty} \underset{z=w\alpha}{\text{Res}} \frac {F(z,w)}z
\end{equation}
To keep the subsequent notation simple, we will often ignore the homomorphism $\phi$ of \eqref{eqn:constants}, and write $q_1,q_2$ for the Chern roots of the cotangent bundle of $S$. Therefore, symmetric polynomials in $q_1,q_2$ will be elements of $K_S$. We may assume $n > 0$ without loss of generality, and then we have three cases to study: \\

\noindent \emph{Case 1:} $n' > 0$. By \eqref{eqn:important 2}, we have $e_{n,k} = f_{d_\bullet}$ and $e_{n',k'} = f_{d_\bullet'}$ for certain vectors of integers $d_\bullet, d_\bullet'$ (we ignore the powers of $q$ that appear in \eqref{eqn:important 2} and \eqref{eqn:two}, as they will play no role in the following argument). Then \eqref{eqn:two} implies the following:
\begin{align*} 
&\Psi(...,\CT_c,...) \stackrel{e_{n,k}}\leadsto \int_{\infty - 0}^{z_n \prec ... \prec z_1} \frac{z_1^{d_1} ... z_n^{d_n} \Psi(...,\CT_c - (z_1+...+z_n) l_c,...)}{\left(1 - \frac {z_2 q}{z_1} \right) ... \left(1 - \frac {z_n q}{z_{n-1}} \right) \prod_{i < j} \zeta \left(\frac {z_j}{z_i} \right)} \prod_{i=1}^n \wedge^\bullet \left(- \frac {z_i}{\CU} \right) \\
&\Psi(...,\CT_c,...) \stackrel{e_{n',k'}}\leadsto \int_{\infty - 0}^{w_{n'} \prec ... \prec w_1} \frac{w_1^{d'_1} ... w_{n'}^{d'_{n'}} \Psi(...,\CT_c - (w_1+...+w_{n'}) l_c,...)}{\left(1 - \frac {w_2 q}{w_1} \right) ... \left(1 - \frac {w_{n'} q}{w_{n'-1}} \right) \prod_{i < j} \zeta \left(\frac {w_j}{w_i} \right)} \prod_{i=1}^{n'} \wedge^\bullet \left(- \frac {w_{i'}}{\CU} \right) 
\end{align*} 
for any $\Psi$. When composing two operators given by integral formulas as above, we obtain the following formulas (we henceforth write $l^{(1)}, l^{(2)} \in \kss$ for the pull-backs of any class $l \in K_S$ via the first and second projections, respectively):
\begin{equation}
\label{eqn:composition nr 1} 
e_{n,k} e_{n',k'} \Psi(...,\CT_c,...) = \int_{\infty - 0}^{z_n \prec ... \prec z_1 \prec w_{n'} \prec ... \prec w_1} \frac {z_1^{d_1}...z_n^{d_n} w_1^{d_1'}... w_{n'}^{d_{n'}'}}{\prod^{1\leq i \leq n}_{1 \leq i' \leq n'} \zeta_{12} \left( \frac {z_i}{w_{i'}} \right)} 
\end{equation}
$$
\frac {\Psi(...,\CT_c - (z_1+...+z_n)l_c^{(1)} - (w_1+...+w_{n'})l_c^{(2)},...) \prod_{i=1}^n \wedge^\bullet \left(- \frac {z_i}{\CU^{(1)}} \right) \prod_{i'=1}^{n'} \wedge^\bullet \left(- \frac {w_{i'}}{\CU^{(2)}} \right)}{\left(1 - \frac {z_2 q^{(1)}}{z_1} \right) ... \left(1 - \frac {z_n q^{(1)}}{z_{n-1}} \right) \left(1 - \frac {w_2 q^{(2)}}{w_1} \right) ... \left(1 - \frac {w_{n'} q^{(2)}}{w_{n'-1}} \right) \prod_{i < j} \zeta^{(1)} \left(\frac {z_j}{z_i} \right) \prod_{i' < j'} \zeta^{(2)} \left(\frac {w_{j'}}{w_{i'}} \right)}
$$
\begin{equation}
\label{eqn:composition nr 2} 
e_{n',k'} e_{n,k}  \Psi(...,\CT_c,...) = \int_{\infty - 0}^{w_{n'} \prec ... \prec w_1 \prec z_n \prec ... \prec z_1} \frac {z_1^{d_1}...z_n^{d_n} w_1^{d_1'}... w_{n'}^{d_{n'}'}}{\prod^{1\leq i \leq n}_{1 \leq i' \leq n'} \zeta_{12} \left( \frac {w_{i'}}{z_i} \right) }
\end{equation}
$$
\frac {\Psi(...,\CT_c - (z_1+...+z_n)l_c^{(1)} - (w_1+...+w_{n'})l_c^{(2)},...) \prod_{i=1}^n \wedge^\bullet \left(- \frac {z_i}{\CU^{(1)}} \right) \prod_{i'=1}^{n'} \wedge^\bullet \left(- \frac {w_{i'}}{\CU^{(2)}} \right)}{\left(1 - \frac {z_2 q^{(1)}}{z_1} \right) ... \left(1 - \frac {z_n q^{(1)}}{z_{n-1}} \right) \left(1 - \frac {w_2 q^{(2)}}{w_1} \right) ... \left(1 - \frac {w_{n'} q^{(2)}}{w_{n'-1}} \right) \prod_{i < j} \zeta^{(1)} \left(\frac {z_j}{z_i} \right) \prod_{i' < j'} \zeta^{(2)} \left(\frac {w_{j'}}{w_{i'}} \right)}
$$
where $\zeta_{12} (x) = \wedge^{\bullet}(-x \cdot \CO_{\Delta}) \in K_{S \times S}(x)$ is defined in \eqref{eqn:zeta ij}, while $\zeta^{(1)}(x), \zeta^{(2)}(x)$ denote the pull-backs of the rational functions \eqref{eqn:zeta geom} from $K_S(x)$ to $K_{S \times S}(x)$ via the two projections. The two expressions in the displays above only differ in the order of the contours, and in the argument of the function $\zeta_{12}$. However, because:
\begin{equation}
\label{eqn:foil zeta}
\zeta_{12} (x) = 1 + \Delta_* \left[ \frac {x}{(1-x)(1-xq)} \right]
\end{equation}
(see (3.6) of \cite{Univ}) then $\frac 1{\zeta_{12}(x)} \in 1 + \Delta_* \left( K_S[[x^{\pm 1}]] \right)$ and we conclude that:
$$
[e_{n,k}, e_{n',k'}]  \Psi(...,\CT_c,...) = \Delta_* \left( \int_{\infty - 0}^{z_n \prec ... \prec z_1 \prec w_{n'} \prec ... \prec w_1}  \text {or } \int_{\infty - 0}^{w_{n'} \prec ... \prec w_1 \prec z_n \prec ... \prec z_1}  \right.
$$
\begin{equation}
\label{eqn:gist}
\frac {\text{power series in }z_1,...,z_n,w_1,...,w_{n'} \text{ with coefficients in }K_S}{\prod^{1\leq i \leq n}_{1 \leq i' \leq n'} \zeta_{12} \left( \frac {z_i}{w_{i'}} \right) \text{ or } \zeta_{12} \left( \frac {w_{i'}}{z_i} \right)} 
\end{equation}
$$
\left. \frac { \Psi(...,\CT_c - (z_1+...+z_n+w_1+...+w_{n'})l_c,...) \prod_{i=1}^n \wedge^\bullet \left(- \frac {z_i}{\CU} \right) \prod_{i'=1}^{n'} \wedge^\bullet \left(- \frac {w_{i'}}{\CU} \right)}{\left(1 - \frac {z_2 q}{z_1} \right) ... \left(1 - \frac {z_n q}{z_{n-1}} \right) \left(1 - \frac {w_2 q}{w_1} \right) ... \left(1 - \frac {w_{n'} q}{w_{n'-1}} \right) \prod_{i < j} \zeta \left(\frac {z_j}{z_i} \right) \prod_{i' < j'} \zeta \left(\frac {w_{j'}}{w_{i'}} \right)} \right) 
$$
Indeed, all we are doing is to compute \eqref{eqn:composition nr 1} and \eqref{eqn:composition nr 2} by expanding the integrand as a power series in non-negative powers of $\frac {z_i}{w_{i'}}$ and $\frac {w_{i'}}{z_i}$, respectively. Only finitely many of the terms in the expansion effectively contribute to the integral, because $d_1,...,d_n,d_1',...d'_{n'}$ are fixed. Then formula \eqref{eqn:gist} merely claims that all the coefficients in this expansion lie in $\Delta_*(K_S)$, because the summand ``1" in \eqref{eqn:foil zeta} is canceled out when taking the difference between \eqref{eqn:composition nr 1} and \eqref{eqn:composition nr 2}. Regardless of what particular power series arises in \eqref{eqn:gist}, we conclude that:
\begin{equation}
\label{eqn:stuck 1}
[e_{n,k}, e_{n',k'}] = \Delta_* \left( \sum_{\td_\bullet} \phi(o_{\td_\bullet}(q_1,q_2)) \cdot  f_{\td_\bullet} \right)
\end{equation}
for various scalars $o_{\td_\bullet}(q_1,q_2) \in \BK$, which do not depend on the surface $S$ (in fact, the only dependence on $S$ is the fact that the ring homomorphism $\phi$ specializes $q_1,q_2$ to the Chern roots of the cotangent bundle of $S$). Then \eqref{eqn:geom right} implies that:
\begin{equation}
\label{eqn:stuck 2}
[e_{n,k}, e_{n',k'}] = \Delta_* \left( \mathop{\sum^{n_1+...+n_t = n}_{\frac {k_1}{n_1} \geq ... \geq \frac {k_t}{n_t}}}^{n_i \in \BN, k_i \in \BZ} \phi(p_{k,k_1,...,k_t,k'}^{n,n_1,...,n_t,n'}(q_1,q_2)) \cdot e_{n_1,k_1}... e_{n_t,k_t}  \Big|_\Delta \right)
\end{equation}
for certain scalars $p_{k,k_1,...,k_t,k'}^{n,n_1,...,n_t,n'}(q_1,q_2) \in \BK$. These scalars are completely determined by the case when $S = \BA^2$ and $K$--theory is $\BC^* \times \BC^*$ equivariant, when \cite{Mod} implies that they coincide with the structure constants in the algebra $\CA$, i.e. the homonymous polynomials in the right-hand side of \eqref{eqn:e triangle double}. This proves \eqref{eqn:comm}. \\

\noindent \emph{Case 2:} $n'=0$, and take $k' > 0$ without loss of generality. Consider the operators: 
$$
p_{0,1},p_{0,2},... : \km \rightarrow \kms
$$
which are to $e_{0,1}, e_{0,2},...$ as power-sum are to elementary symmetric functions, i.e.:
$$
p_{0,1} = e_{0,1}, \qquad p_{0,2} = e_{0,1}e_{0,1} \Big|_\Delta - 2 e_{0,2}
$$ 
etc. It is elementary to show that formula \eqref{eqn:comm} for the commutator $[e_{n,k}, e_{0,k'}]$ is equivalent to an analogous formula for the commutator $[e_{n,k}, p_{0,k'}]$, which in turn follows from the subsequent formula for any $d_1,...,d_n \in \BZ$:
\begin{equation}
\label{eqn:hecke geom}
[f_{(d_n,...,d_1)}, p_{0,k'}] = - \Delta_* \left[ \frac {(1-q_1^{k'})(1-q_2^{k'})}{(1-q_1)(1-q_2)}  \sum_{i=1}^n f_{(d_n,...,d_i+k',...,d_1)}\right]
\end{equation}
Indeed, this mirrors the way formula \eqref{eqn:e triangle double} for $n' = 0$ follows from \eqref{eqn:hecke right} and \eqref{eqn:right}. In order to prove \eqref{eqn:hecke geom}, one notices that the definition of $p_{0,k'}$ in relation to $e_{0,k'}$, as well as the definition of the operators $e_{0,k'}$ in \eqref{eqn:important 3}, implies that:
$$
p_{0,k'} = \text{multiplication by }\CP_{k'}(\CU)
$$
where $\CP_{t}$ is the additive functor on $K$--theory which takes the class of a bundle to the sum of the $t$-th powers of its Chern roots. Formulas \eqref{eqn:univ decomp} and \eqref{eqn:two} imply the following equality (as before, we ignore the term $q^{k(r-1)}$ in the denominator of \eqref{eqn:two}, as this will not adversely affect the validity of the following argument):
$$
f_{(d_n,...,d_1)} \circ p_{0,k'} \Big( \Psi(...,\CT_c,...) \Big) = \int_{\infty - 0}^{z_n \prec ... \prec z_1} \CP_{k'}\Big[ \left(\CT_c - (z_1+...+z_n)l_c^{(1)} \right) \boxtimes l^{c(2)} \Big]
$$
$$
\frac{z_1^{d_1} ... z_n^{d_n} \Psi(...,\CT_c - (z_1+...+z_n) l_c^{(1)},...)}{\left(1 - \frac {z_2 q^{(1)}}{z_1} \right) ... \left(1 - \frac {z_n q^{(1)}}{z_{n-1}} \right) \cdot \prod_{i < j} \zeta^{(1)} \left(\frac {z_j}{z_i} \right)} \prod_{i=1}^n \wedge^\bullet \left( - \frac {z_i}{\CU^{(1)}} \right) 
$$
where $l_c^{(1)}$ (respectively $l_c^{(2)}$) denotes the pull-back to $S \times S$ of the class $l_c \in \ks$ via the first (respectively second) projection. The composition $p_{0,k'} \circ f_{(d_n,...,d_1)}$ is given by the same formula, but on the first line, one should replace:
$$
\CP_{k'}\Big[ \left(\CT_c - (z_1+...+z_n)l_c^{(1)} \right) \boxtimes l^{c(2)} \Big] \leadsto \CP_{k'}\Big[ \CT_c \boxtimes l^{c(2)} \Big]
$$
Since the functor $\CP_k$ is additive and $\sum_c l_c^{(1)} \boxtimes l^{c(2)} = [\CO_\Delta]$, we conclude that:
\begin{equation} \label{eqn:buena}
[f_{(d_n,...,d_1)}, p_{0,k'}] \Big( \Psi(...,\CT_c,...) \Big) = - \CP_{k'}(\CO_\Delta) \cdot 
\end{equation} 
$$
\int_{\infty - 0}^{z_n \prec ... \prec z_1} (z_1^{k'}+... + z_n^{k'})\frac{z_1^{d_1} ... z_n^{d_n} \Psi(...,\CT_c - (z_1+...+z_n) l_c^{(1)},...)}{\left(1 - \frac {z_2 q^{(1)}}{z_1} \right) ... \left(1 - \frac {z_n q^{(1)}}{z_{n-1}} \right) \cdot \prod_{i < j} \zeta^{(1)} \left(\frac {z_j}{z_i} \right)} \prod_{i=1}^n \wedge^\bullet \left( - \frac {z_i}{\CU^{(1)}} \right)
$$
Now we will use the identity: 
\begin{equation}
\label{eqn:plethy diag}
\CP_{k'}(\CO_\Delta) = \Delta_*\left[\frac {(1-q_1^{k'})(1-q_2^{k'})}{(1-q_1)(1-q_2)} \right]
\end{equation}
which can be proved akin to (3.6) of \cite{Univ} (specifically, one first proves \eqref{eqn:plethy diag} directly when $[\CO_\Delta] = [\CO] - [\CV] + [\det \CV]$ for a rank $2$ vector bundle $\CV$, and then obtains the general formula by deformation to the normal bundle of $\Delta \hookrightarrow S \times S$). If we use \eqref{eqn:two} and \eqref{eqn:plethy diag}, we transform relation \eqref{eqn:buena} into \eqref{eqn:hecke geom}, as desired. \\

\noindent \emph{Case 3:} $n' < 0$. We will show that for any $d_\bullet = (d_1,...,d_n)$ and $d_{\bullet}' = (d_{-n'}',...,d'_1)$, we have the following a priori weaker version of Conjecture \ref{conj:comm}:
\begin{equation}
\label{eqn:comm is nice}
[e_{d_\bullet}, f_{d_\bullet'}] = \Delta_* \left( \sum_{\tilde{d}_\bullet, k, \tilde{d}_\bullet'} \phi(p_{\tilde{d}_\bullet,k,\tilde{d}_\bullet'}(q_1,q_2)) \cdot e_{\tilde{d}_\bullet} h^\pm_k f_{\tilde{d}_\bullet'} \Big|_\Delta \right)
\end{equation}
where the coefficients $p_{\tilde{d}_\bullet,k,\tilde{d}_\bullet'}(q_1,q_2) \in \ring$ that appear in the right-hand side do not depend on the surface $S$. In the formula above:
$$
h_k^\pm : \km \rightarrow \kms
$$
are the coefficients of $h^\pm(z) = \sum_{k = 0}^\infty \frac {h^\pm_k}{z^{\pm k}}$, which acts on $\km$ as multiplication by:
$$
\wedge^\bullet \left(\frac {\CU(q^{-1}-1)}z \right)
$$
Its relevance to our purposes is the following simple formula for the $n=-n'=1$ case of \eqref{eqn:comm is nice}, which follows immediately from \eqref{eqn:one} and \eqref{eqn:two}, and was proved in \cite{Univ} even in the absence of Assumption B:
$$
\left[ \sum_{d\in \BZ} \frac {e_d}{z^d}, \sum_{d'\in \BZ} \frac {f_{d'}}{w^{d'}} \right] = \delta \left( \frac zw \right) \Delta_* \left( \frac {h^+(z) - h^-(w)}{1-q^{-1}} \right)
$$
Let us show that \eqref{eqn:comm is nice} implies \eqref{eqn:comm}. By \eqref{eqn:left} and \eqref{eqn:right}, the right-hand side of \eqref{eqn:comm is nice} can be written as a product of generators $e_{n,k}$ in clockwise order of $(n,k)\in \BZ^2$. Since the coefficients of the resulting decomposition do not depend on the surface $S$, and since for $\BC^* \times \BC^*$ equivariant $S = \BA^2$ the coefficients match those that occur in the algebra $\CA$ (as shown in \cite{W}, \cite{Mod}), relation \eqref{eqn:comm} follows. \\

\noindent Therefore, it remains to prove \eqref{eqn:comm is nice}, which will occupy the remainder of the present Section. We will use a residue computation, which will require us to assume that: 
\begin{equation}
\label{eqn:assumption w}
1-q^a \text{ is invertible in the ring }\ks
\end{equation}
for all $a \in \BN$. Since this never happens in the algebraic $K$--theory ring of a smooth projective surface (as $q$ is unipotent), we must argue for why this assumption is allowed. By Assumption B, in order to prove \eqref{eqn:comm is nice}, it is enough to show that: 
\begin{multline}
[e_{d_\bullet}, f_{d_\bullet'}] \Psi (...,\CT_c,...) = \label{eqn:comm is nice 2} \\ = \Delta_* \left( \sum_{\tilde{d}_\bullet, k, \tilde{d}_\bullet'} \phi(p_{\tilde{d}_\bullet,k,\tilde{d}_\bullet'}(q_1,q_2)) \cdot e_{\tilde{d}_\bullet} h^\pm_k f_{\tilde{d}_\bullet'} \Big|_\Delta \Psi (...,\CT_c,...) \right)
\end{multline}
for all symmetric Laurent polynomials $\Psi$ as in \eqref{eqn:psi taut}. Because of \eqref{eqn:one} and \eqref{eqn:two}, both sides of relation \eqref{eqn:comm is nice 2} will be Laurent polynomials in the tautological classes, with coefficients which are Laurent polynomials in $q_1,q_2 \in \ks$. To prove that a certain identity between such polynomials holds, it is enough to do it for a choice of $\CM$ and $S$ for which there do not exist any relations between the $K$-theory classes $q_1,q_2$ and the various tautological classes $\CT_c$. To this end: 

\text{ }

\begin{itemize}

\item consider $S = \BA^2$ \\

\item consider the $\BC^* \times \BC^*$ equivariant $K$--theory ring of $S$ (which is the ring $\BK$ of \eqref{eqn:ring}), localized with respect to the elements $1-q^a$ of \eqref{eqn:assumption w} \\

\item consider the moduli space $\CM$ of rank $r$ sheaves on $\BP^2$, for a sufficiently large natural number $r$, framed along the divisor $\infty \subset \BP^2$:
\begin{equation}
\label{eqn:framing}
\CF|_\infty \cong \CO_\infty^{\oplus r}
\end{equation}
(see \cite{W}, \cite{Mod} for a version of our treatment in the setting of framed sheaves). \\

\end{itemize}

\noindent Let a maximal torus $T \subset GL_r$ act on $\CM$ by left multiplication of the isomorphism \eqref{eqn:framing}. Therefore, in $\BC^* \times \BC^* \times T$ equivariant $K$--theory, the universal sheaf $\CU$ on $\CM \times S$ splits up as a direct sum $\oplus_c \ \CU_c \cdot t^c$, where $\{t^c\}$ goes over the characters of the torus $T$. If one chooses $r \in \BN$ large enough, then there will be no relations between the tautological bundles $\CT_c = \text{proj}_{1*}(\CU_c)$, hence there are no algebraic relations between $\{\Psi(...,\CT_c,...)\}_{\Psi \text{ Laurent polynomial}}$ that hold for all $r \in \BN$. \\


\noindent According to the preceding discussion, it remains to prove formula \eqref{eqn:comm is nice 2} subject to assumption \eqref{eqn:assumption w}. We start by invoking \eqref{eqn:one} and \eqref{eqn:two}:
\begin{equation}
\label{eqn:two comp 0}
e_{d_\bullet} f_{d_\bullet'} \Psi (...,\CT_c,...) = \int_{\infty - 0}^{z_1 \prec ... \prec z_n \prec w_{-n'} \prec ... \prec w_1} R(z_1,...,z_n,w_1,...,w_{-n'}) 
\end{equation}
where:
\begin{equation}
\label{eqn:def r}
R(z_1,...,z_n,w_1,...,w_{-n'})  = \frac {(-1)^n}{q^{n'(1-r)}} z_1^{d_1} ... z_n^{d_n} w_1^{d_1'} ... w_{-n'}^{d'_{-n'}} \prod^{1\leq i \leq n}_{1\leq j \leq -n'} \zeta_{12} \left( \frac {z_i}{w_j} \right)
\end{equation}
$$
\frac{\Psi(...,\CT_c + (z_1+...+z_n)l_c^{(1)} - (w_1+...+w_{-n'}) l^{(2)}_c,...) \prod_{i=1}^n \wedge^\bullet \left(\frac {z_iq}{\CU^{(1)}} \right)  \prod_{i=1}^{-n'} \wedge^\bullet \left(- \frac {w_i}{\CU^{(2)}} \right)}{\prod_{i=1}^{n-1} \left(1 - \frac {z_{i+1} q^{(1)}}{z_i} \right) \prod_{j=1}^{-n'-1} \left(1 - \frac {w_{j+1} q^{(2)}}{w_j} \right) \prod_{1\leq i < i' \leq n} \zeta^{(1)} \left( \frac {z_{i'}}{z_i} \right) \prod_{1\leq j < j' \leq -n'} \zeta^{(2)} \left( \frac {w_{j'}}{w_j} \right)}
$$
for $\zeta_{12}$ as in \eqref{eqn:foil zeta}. The composition $f_{d_\bullet'} e_{d_\bullet}$ is given by the same formula as \eqref{eqn:two comp 0}, but the order of the contours of integration is now $w_{-n'} \prec ... \prec w_1 \prec z_1 \prec ... \prec z_n$. Therefore, the commutator is given by taking residues when the contour of some variable $z_a$ passes over the contour of some other variable $w_b$:
\begin{equation}
\label{eqn:two comp 1}
[ e_{d_\bullet}, f_{d_\bullet'} ] \Psi (...,\CT_c,...) = \sum_{a=1}^n \sum_{b=1}^{-n'} 
\end{equation}
$$
\int_{\infty - 0}^{w_{-n'} \prec ... \prec w_{b+1} \prec z_1 \prec ... \prec z_{a-1} \prec z_a,w_b \prec z_{a+1} \prec ... \prec z_n \prec w_{b-1} \prec ... \prec w_1} ``\sum \text{Res}_{z_a,w_b} R"
$$
where $``\sum \text{Res}_{z_a,w_b} R"$ is shorthand for the sum of residues at all poles of the form $z_a =  \alpha w_b$, as $\alpha$ ranges over all possible constants $\notin \{0,\infty\}$. From now on, we will often write expressions such as $\text{Res}_{w_b = s}^{z_a = \alpha s} R$, which mean ``take the residue at $z_a = \alpha w_b$, and then relabel the variable $w_b = s$". Since $\zeta_{12}$ is given by formula \eqref{eqn:foil zeta}, the poles that appear in the right-hand side of \eqref{eqn:two comp 1} are $z_a = w_b$ and $z_a = w_b q^{-1}$. Moreover, the corresponding residues are both multiples of $[\CO_\Delta]$. We may therefore use the projection formula in the guise of the identity:
$$
[\CO_\Delta] \cdot \gamma = \Delta_*(\gamma|_\Delta) \qquad \forall \gamma \in \kss
$$
and conclude that:
\begin{equation}
\label{eqn:two comp 2}
[ e_{d_\bullet}, f_{d_\bullet'} ] \Psi (...,\CT_c,...) = \sum_{a=1}^n \sum_{b=1}^{-n'} 
\end{equation}
$$
\int_{\infty - 0}^{ ... \prec w_{b+1} \prec z_1 \prec ... \prec z_{a-1} \prec s \prec z_{a+1} \prec ... \prec z_n \prec w_{b-1} \prec ... } \Delta_* \left[ \frac { \text{Res}^{z_a=s}_{w_b=s} ( R|_\Delta ) + \text{Res}^{z_a=\frac sq}_{w_b=s} ( R|_\Delta ) }{(1-q_1)(1-q_2)} \right]
$$
\footnote{Right after \eqref{eqn:spec r 2}, we will explain how to interpret the denominator of the right-hand side} where $R|_\Delta$ refers to the restriction of the coefficients of the rational function $R$, which are a priori elements in $K_{S \times S}$, to the diagonal $\Delta : S \hookrightarrow S \times S$. In practice, this involves setting in \eqref{eqn:def r} $l_c^{(1)}, l_c^{(2)} \mapsto l_c$, $q^{(1)}, q^{(2)} \mapsto q$, $\CU^{(1)}, \CU^{(2)} \mapsto \CU$ and:
$$
\zeta_{12} (x) \mapsto \zeta(x) = \frac {\wedge^\bullet(\Omega_S^1 \cdot x)}{(1-x)(1-xq)} \in \ks(x)
$$
Explicitly, we have:
\begin{equation}
\label{eqn:def r restricted}
R|_\Delta(z_1,...,z_n,w_1,...,w_{-n'})  = z_1^{d_1} ... z_n^{d_n} w_1^{d_1'} ... w_{-n'}^{d_{-n'}'} \prod^{1\leq i \leq n}_{1\leq j \leq -n'} \zeta \left( \frac {z_i}{w_j} \right)
\end{equation}
$$
\frac{\Psi(...,\CT_c + (z_1+...+z_n-w_1-...-w_{-n'})l_c,...) \prod_{i=1}^n \wedge^\bullet \left(\frac {z_iq}{\CU} \right)  \prod_{i=1}^{-n'} \wedge^\bullet \left(- \frac {w_i}{\CU} \right)}{\prod_{i=1}^{n-1} \left(1 - \frac {z_{i+1} q}{z_i} \right) \prod_{j=1}^{-n'-1} \left(1 - \frac {w_{j+1} q}{w_j} \right) \prod_{1\leq i < i' \leq n} \zeta \left( \frac {z_{i'}}{z_i} \right) \prod_{1\leq j < j' \leq -n'} \zeta \left( \frac {w_{j'}}{w_j} \right)}
$$
Let us take the residue at $z_a = w_b$ in \eqref{eqn:def r restricted} and call the resulting variable $s$:
\begin{equation}
\label{eqn:spec r 1}
\text{Res}^{z_a = s}_{w_b=s} (R|_\Delta) = \frac {\frac {(1-q_1)(1-q_2)}{1-q} s^{d_a} s^{d'_b}  \frac {\wedge^\bullet \left( sq/\CU \right)}{\wedge^\bullet \left( s/\CU \right)} \prod_{i < a} \frac {\zeta \left( z_i/s \right)}{\zeta \left( s/z_i \right)} \prod_{j > b} \frac {\zeta \left( s/w_j \right)}{\zeta \left( w_j/s \right)}}{\left(1 - \frac {z_{a+1}q}s \right)\left(1 - \frac {sq}{z_{a-1}} \right) \left(1 - \frac {w_{b+1}q}s \right) \left(1 - \frac {sq}{w_{b-1}} \right)}
\end{equation}
$$
\frac{\prod_{i \neq a} z_i^{d_i} \prod_{j \neq b} w_j^{d'_j}  \prod^{i \neq a}_{j \neq b} \zeta \left( \frac {z_i}{w_j} \right) \Psi(...,\CT_c + (\sum_{i \neq a} z_i - \sum_{j \neq b} w_j)l_c,...) \frac {\prod_{i\neq a} \wedge^\bullet \left( z_iq / \CU \right)}{\prod_{j \neq b} \wedge^\bullet \left( w_j / \CU \right)}}{\prod_{i,i+1 \neq a} \left(1 - \frac {z_{i+1} q}{z_i} \right) \prod_{j,j+1 \neq b} \left(1 - \frac {w_{j+1} q}{w_j} \right) \prod_{i < i'}^{i,i' \neq a} \zeta \left( \frac {z_{i'}}{z_i} \right) \prod_{j < j'}^{j,j' \neq b} \zeta \left( \frac {w_{j'}}{w_j} \right)}
$$
and when we take the residue at $z_a = \frac sq$ and $w_b = s$ in \eqref{eqn:def r restricted}, we obtain:
\begin{equation}
\label{eqn:spec r 2}
\text{Res}^{z_a = \frac sq}_{w_b=s} (R|_\Delta) = \frac {-\frac {(1-q_1)(1-q_2)}{1-q} \left( \frac sq \right)^{d_a} s^{d'_b} \prod_{i > a} \frac {\zeta \left( s/(z_iq) \right)}{\zeta \left( z_i q/s \right)} \prod_{j < b} \frac {\zeta \left( w_j/s \right)}{\zeta \left( s/w_j \right)}}{\left(1 - \frac {z_{a+1}q^2}s \right)\left(1 - \frac {s}{z_{a-1}} \right) \left(1 - \frac {w_{b+1}q}s \right) \left(1 - \frac {sq}{w_{b-1}} \right)}
\end{equation}
$$
\frac{\prod_{i \neq a} z_i^{d_i} \prod_{j \neq b} w_j^{d'_j}  \prod^{i \neq a}_{j \neq b} \zeta \left( \frac {z_i}{w_j} \right) \Psi(...,\CT_c + (\sum_{i \neq a} z_i - \sum_{j \neq b} w_j+sq-s)l_c,...) \frac {\prod_{i\neq a} \wedge^\bullet \left( z_iq / \CU \right)}{\prod_{j \neq b} \wedge^\bullet \left( w_j / \CU \right)}}{\prod_{i,i+1 \neq a} \left(1 - \frac {z_{i+1} q}{z_i} \right) \prod_{j,j+1 \neq b} \left(1 - \frac {w_{j+1} q}{w_j} \right) \prod_{i < i'}^{i,i' \neq a} \zeta \left( \frac {z_{i'}}{z_i} \right) \prod_{j < j'}^{j,j' \neq b} \zeta \left( \frac {w_{j'}}{w_j} \right)}
$$
(note that we have used the identity $\zeta(x/q) = \zeta(1/x)$ on the first line of formula \eqref{eqn:spec r 2}). The denominator $(1-q_1)(1-q_2)$ in \eqref{eqn:two comp 2} simply means that one should remove the initial factor of $(1-q_1)(1-q_2)$ from formulas \eqref{eqn:spec r 1} and \eqref{eqn:spec r 2}. Meanwhile, the denominator $1-q$ is acceptable because of assumption \eqref{eqn:assumption w}. Now recall that \eqref{eqn:spec r 1} and \eqref{eqn:spec r 2} must be integrated over the contours \eqref{eqn:two comp 2}. \\

\begin{itemize}[leftmargin=*]

\item In the case of \eqref{eqn:spec r 2}, we can move the variable $s$ to the left of the chain of variables $... \prec w_{b+1} \prec ... \prec z_{a-1} \prec s \prec z_{a+1} \prec ... \prec w_{b-1} \prec ...$ and the only poles we would encounter are $z_{a-1} = s$ (with sign $+$) and $w_{b+1} = \frac sq$ (with sign $-$), because $\Psi$ has no poles in $s$ other than $0$ and $\infty$. \\

\item In the case of \eqref{eqn:spec r 1}, we can move the variable $s$ to the right of the chain of variables $... \prec w_{b+1} \prec ... \prec z_{a-1} \prec s \prec z_{a+1} \prec ... \prec w_{b-1} \prec ...$ and the only poles we would encounter are $z_{a+1} = \frac sq$ (with sign $-$) and $w_{b-1} = sq$ (with sign $+$), together with $0$ and $\infty$. \\

\end{itemize}

\noindent In the second bullet, note that the residues at $0$ and $\infty$ in $s$ can be computed by expanding \eqref{eqn:spec r 1} in powers of $s$, which yields some Laurent polynomial in the variables $\{z_i\}_{i\neq a}$, $\{w_j\}_{j \neq b}$ and some series coefficient of the ratio: 
$$
\frac {\wedge^\bullet \left( \frac {sq}{\CU} \right)}{\wedge^\bullet \left( \frac s{\CU} \right)}
$$ 
expanded in either positive or negative powers of $s$. The corresponding expression will be of the form in the right-hand side of \eqref{eqn:comm is nice 2}, which we will henceforth refer to as a \textbf{nice summand}. Putting all the residues together, we conclude that the RHS of \eqref{eqn:two comp 2} equals:
\begin{align*}
&-\sum_{a=1}^{n} \sum_{b=1}^{-n'-1} \int_{\infty - 0}^{ ... \prec w_{b+2} \prec s \prec z_1 \prec ... \prec z_{a-1} \prec z_{a+1} \prec ... \prec z_n \prec w_{b-1} \prec ... } \Delta_* \left[ \frac { \text{Res}^{z_{a} = \frac sq}_{w_b = s, w_{b+1} = \frac sq} ( R|_\Delta )}{(1-q_1)(1-q_2)} \right] + \\
&+ \sum_{a=2}^{n} \sum_{b=1}^{-n'} \int_{\infty - 0}^{ ... \prec w_{b+1} \prec z_1 \prec ... \prec z_{a-2} \prec s \prec z_{a+1} \prec ... \prec z_n \prec w_{b-1} \prec ... } \Delta_* \left[ \frac { \text{Res}^{z_{a-1}=s, z_{a} = \frac sq}_{w_b=s} ( R|_\Delta )}{(1-q_1)(1-q_2)} \right] - \\ 
&- \sum_{a=1}^{n-1} \sum_{b=1}^{-n'} \int_{\infty - 0}^{ ... \prec w_{b+1} \prec z_1 \prec ... \prec z_{a-1} \prec s \prec z_{a+2} \prec ... \prec z_n \prec w_{b-1} \prec ... } \Delta_* \left[ \frac { \text{Res}^{z_a=s, z_{a+1} = \frac sq}_{w_b=s} ( R|_\Delta )}{(1-q_1)(1-q_2)} \right] + \\ 
&+ \sum_{a=1}^n \sum_{b=2}^{-n'} \int_{\infty - 0}^{ ... \prec w_{b+1} \prec z_1 \prec ... \prec z_{a-1} \prec z_{a+1} \prec ... \prec z_n \prec s \prec w_{b-2} \prec ... } \Delta_* \left[ \frac { \text{Res}^{z_a=s}_{w_b=s, w_{b-1} = sq} ( R|_\Delta )}{(1-q_1)(1-q_2)} \right] 
\end{align*}
plus nice summands. Note that the second and third rows cancel out, because they give the same residues with opposite signs (upon relabeling $a \mapsto a-1$). The first and fourth rows consist of the same integrand (upon relabeling $b \mapsto b-1$ and $s \mapsto sq$), but the contours of integration have $s$ to the left/right of the chain of $z$ variables in the case of the first/fourth rows. Therefore, we conclude that the formula above is simply the base case (namely $k=1$) of the following result: \\

\begin{claim}
\label{claim:residues}

For any $k\geq 1$, the RHS of \eqref{eqn:two comp 2} equals the sum over all $A,B$ of:
$$
\int_{\infty - 0}^{ ... \prec w_{b+1} \prec \left( z_i \text{ for } i \notin A \right) \prec s  \prec w_{b-k-1} \prec ... } \Delta_* \left[ \frac {\emph{Res}^{z_{a_i+\e} = sq^{-\e-1+\sum_{j \leq i} c_j-a_j}}_{w_b = s,...,w_{b-k} = sq^k} ( R|_\Delta )}{(1-q_1)(1-q_2)} \right] -
$$
\begin{equation}
\label{eqn:residues}
- \int_{\infty - 0}^{ ... \prec w_{b+1} \prec s \prec \left( z_i \text{ for } i \notin A \right) \prec w_{b-k-1} \prec ... } \Delta_* \left[ \frac {\emph{Res}^{z_{a_i+\e} = sq^{-\e-1+\sum_{j \leq i} c_j-a_j}}_{w_b = s,...,w_{b-k} = sq^k} ( R|_\Delta )}{(1-q_1)(1-q_2)} \right]
\end{equation}
plus nice summands, as $A$ ranges over the following $k$-element subsets of $\{1,...,n\}$:
$$
A = \{\underbrace{a_1,...,c_1-1}_{\text{consecutive}},...,\underbrace{a_t,...,c_t-1}_{\text{consecutive}}\} \quad \text{for} \quad a_1 < c_1 \leq a_2 < c_2 \leq ... \leq a_t < c_t 
$$
and $B = \{b-k,...,b\}$ ranges over $(k+1)$-element subsets of $\{1,...,-n'\}$ consisting of consecutive elements. The contours of the variables $z_i$ in the two integrals are ordered in increasing order of $i$. The residues that appear in \eqref{eqn:residues} have denominators $1-q^a$, for various $a \in \BN$, but these are acceptable due to \eqref{eqn:assumption w}. \\

\end{claim}

\noindent To make the subsequent explanation more vivid, we will think of $A$ as a disjoint union of \textbf{strings}, each string consisting of consecutive numbers $a_i,...,c_i-1$. In \eqref{eqn:residues}, the variables $z_j$ for $a_i \leq j < c_i$ are specialized to $sq^*$ where $* \in \{0,...,k-1\}$ increases from one string to the next, but decreases with step 1 within a given string. First of all, let us note that Claim completes the proof of the Conjecture subject to Assumption B, because as soon as $k > \min(n,-n')$, the sum of integrals in \eqref{eqn:residues} is vacuous, and all we are left with are nice summands. \\

\noindent We now prove Claim \ref{claim:residues} by induction on $k$ (the base case $k = 1$ is the fact that \eqref{eqn:two comp 2} equals the four-line formula located immediately before the statement of Claim \ref{claim:residues}). For fixed sets $A$ and $B$, the integrand of the two integrals of \eqref{eqn:residues} is one and the same rational function in $s,\{z_i\}_{i\notin A}, \{w_j\}_{j \notin B}$, namely:
\begin{align*}
&\text{Res}^{z_{a_i+\e} = sq^{-\e-1+\sum_{j \leq i} c_j-a_j}}_{w_b = s,...,w_{b-k} = sq^k} \left( R|_\Delta \right) = \frac {\prod^{1 \leq i \leq t}_{a_i \leq e < c_i} (sq^{c_i-e-1+\sum_{j<i} c_j-a_j})^{d_e} \prod_{j=b}^{b-k} (sq^{b-j})^{d_j'}}{\wedge^\bullet \left(\frac s{\CU} \right)  \prod_{j = 1}^{b-k} \zeta \left(\frac {sq^k}w \right) \prod_{j = b+1}^{n'} \zeta \left(\frac sw \right)} \\ 
&\frac {\prod^{0\leq i \leq t}_{c_i \leq e < a_{i+1}} \zeta \left(\frac z{sq^{\sum_{j\leq i} c_j-a_j}} \right) \Psi(...,\CT_c + (\sum_{i \notin A} z_i - \sum_{j \notin B} w_j-sq^k)l_c,...)}{\left(1 - \frac {w_{b+1}q}s \right)\left(1 - \frac {sq^{k+1}}{w_{b-k-1}}\right) \prod^{1 \leq i \leq t}_{\text{if }a_i \neq c_{i-1}} \left(1 - \frac {sq^{\sum_{j\leq i} c_j-a_j}}{z_{a_i-1}} \right) \prod^{1 \leq i \leq t}_{\text{if }c_i \neq a_{i+1}} \left(1 - \frac {z_{c_i}}{sq^{-1+\sum_{j<i} c_j-a_j}} \right)} \\ 
&\frac{\text{constant} \cdot \prod_{i \notin A} z_i^{d_i} \prod_{j \notin B} w_j^{d'_j}  \prod^{i \notin A}_{j \notin B} \zeta \left( \frac {z_i}{w_j} \right)  \frac {\prod_{i \notin A} \wedge^\bullet \left( z_iq / \CU \right)}{\prod_{j \notin B} \wedge^\bullet \left( w_j / \CU \right)}}{\prod_{i,i+1 \notin A} \left(1 - \frac {z_{i+1} q}{z_i} \right) \prod_{j,j+1 \notin B} \left(1 - \frac {w_{j+1} q}{w_j} \right) \prod_{i < i'}^{i,i' \notin A} \zeta \left( \frac {z_{i'}}{z_i} \right) \prod_{j < j'}^{j,j' \notin B} \zeta \left( \frac {w_{j'}}{w_j} \right)}
\end{align*}
where we set $c_0 = 1$ and $a_{t+1} = n+1$. The term labeled ``constant" is an element of $\BK$ divided by various factors $\{1-q^a\}_{a\in \BN}$. The difference between the integrals in Claim \ref{claim:residues} has to do with the contour of the variable $s$ passing over the $z$ variables, and the only poles this picks up are those of the form $s = z_i \alpha$ for some $i\notin A$ and some $\alpha \in \BK$. Explicitly from the 3-row expression above, these are all of the form:
\begin{align}
&z_{a_i-1} = s q^{\sum_{j \leq i} c_j - a_j} & &\text{if } a_i \neq c_{i-1} \label{eqn:extended 1} \\
&z_{c_i} = s q^{-1 + \sum_{j < i} c_j - a_j} & &\text{if } c_i \neq a_{i+1} \label{eqn:extended 2} \\
&z_e = s q^{\sum_{j \leq i} c_j - a_j} & &\text{for some } c_i \leq e < a_{i+1} \label{eqn:extended 3} \\
&z_e = s q^{- 1 + \sum_{j < i} c_j - a_j} & &\text{for some } c_{i-1} \leq e < a_i \label{eqn:extended 4}
\end{align}
In terms of the interpretation of the set $A$ as a union of strings of consecutive numbers in $\{1,...,n\}$, where the index $i$ corresponds to the variable $z_i$, these four specializations mean the following: the first refers to adding the number $a_i-1$ to the left of the string $\{a_i,...,c_i-1\}$, the second refers to adding the number $c_i$ to the right of the string $\{a_i,...,c_i-1\}$, while the last two specializations correspond to creating the one-element string $\{e\}$. We will refer to either of these four situations as \textbf{extending the string} by the index $a_i - 1$, $c_i$ or $e$, respectively. \\

\noindent Let us now consider the powers $sq^*$ to which the variables $z_i$ are specialized in the extended string (specifically, the powers of $q$ which appear in \eqref{eqn:extended 1}--\eqref{eqn:extended 4}). There are only four types of extended strings where the powers never repeat, namely:
\begin{align}
&\eqref{eqn:extended 1} \text{ for }  i = t \label{eqn:good 1} \\
&\eqref{eqn:extended 2} \text{ for } i = 1 \label{eqn:good 2} \\
&\eqref{eqn:extended 3} \text{ for } i = t \text{ for some } e \geq c_t \label{eqn:good 3} \\
&\eqref{eqn:extended 4} \text{ for } i = 1 \text{ for some } e < a_1 \label{eqn:good 4}
\end{align}
All other extended strings are \textbf{repeated}, which means that the particular power of $q$ in $sq^*$ that appears in the right-hand side of \eqref{eqn:extended 1}--\eqref{eqn:extended 4} will appear again somewhere in the string. However, any repeated string arises twice in this way, since either of the two equal powers $sq^*$ can play the role of the one which was last added to yield the extended string. We leave it as an exercise to the interested reader to show that the corresponding residues of the two occurrences of any repeated string cancel each other out. We conclude that \eqref{eqn:residues} consists of nice summands plus the residues corresponding to the extended strings \eqref{eqn:good 1}--\eqref{eqn:good 4}, which are:
\begin{align*}
&\int_{\infty - 0}^{ ... \prec w_{b+1} \prec \left( z_i \text{ for } i < a_t-1 \right) \prec s \prec  \left(z_i \text{ for } i \geq c_t \right) \prec w_{b-k-1} \prec ... } \Delta_* \left[ \frac { \text{Res}^{(1)} ( R|_\Delta )}{(1-q_1)(1-q_2)} \right] - \\
- &\int_{\infty - 0}^{ ... \prec w_{b+1} \prec \left( z_i \text{ for } i < a_1 \right) \prec s \prec \left( z_i \text{ for } i > c_1 \right) \prec w_{b-k-1} \prec ... } \Delta_* \left[ \frac { \text{Res}^{(2)} ( R|_\Delta )}{(1-q_1)(1-q_2)} \right] + \\ 
+&\sum_{e \geq c_t} \int_{\infty - 0}^{ ... \prec w_{b+1} \prec \left( z_i \text{ for } i < e \right) \prec s \prec  \left(z_i \text{ for } i > e \right) \prec w_{b-k-1} \prec ... } \Delta_* \left[ \frac { \text{Res}^{(3)} ( R|_\Delta )}{(1-q_1)(1-q_2)} \right] - \\ 
-&\sum_{e < a_1} \int_{\infty - 0}^{ ... \prec w_{b+1} \prec \left( z_i \text{ for } i < e \right) \prec s \prec  \left(z_i \text{ for } i > e \right) \prec w_{b-k-1} \prec ... } \Delta_* \left[ \frac { \text{Res}^{(4)} ( R|_\Delta )}{(1-q_1)(1-q_2)} \right]  
\end{align*}
where $\text{Res}^{(1)}$,..., $\text{Res}^{(4)}$ denote the residues at:
\begin{align}
&\left\{ z_{a_i+\e} = sq^{-\e-1+\sum_{j \leq i} c_j-a_j}, z_{a_t-1} = sq^k, w_b = s,...,w_{b-k} = sq^k \right\} \label{eqn:residue list 1} \\ 
&\left\{ z_{a_i+\e} = sq^{-\e-1+\sum_{j \leq i} c_j-a_j}, z_{c_1} = \frac sq, w_b = s,...,w_{b-k} = sq^k \right\} \label{eqn:residue list 2} \\
&\left\{ z_{a_i+\e} = sq^{-\e-1+\sum_{j \leq i} c_j-a_j}, z_{e} = s q^k, w_b = s,...,w_{b-k} = sq^k \right\} \label{eqn:residue list 3} \\
&\left\{ z_{a_i+\e} = sq^{-\e-1+\sum_{j \leq i} c_j-a_j}, z_{e} = \frac sq, w_b = s,...,w_{b-k} = sq^k \right\} \label{eqn:residue list 4}
\end{align}
respectively. In each of the four integrands, the set of indices of the $z$ variables form a string of total size $k+1$ (henceforth called \textbf{long string}), since the string of size $k$ (henceforth called \textbf{short string}) in the statement of Claim \ref{claim:residues} has been extended by an extra index. However, this extra index can be placed either at the end of the short string (as in \eqref{eqn:residue list 1} and \eqref{eqn:residue list 3}, a placement we will call \textbf{posterior}) or at the beginning of the short string (as in \eqref{eqn:residue list 2} and \eqref{eqn:residue list 4}, a placement we will call \textbf{anterior}). So we conclude that every long string appears twice in the sum of four integrals: once in \eqref{eqn:residue list 1} or \eqref{eqn:residue list 3} and once in \eqref{eqn:residue list 2} or \eqref{eqn:residue list 4}. For example, \eqref{eqn:spec r 1}/\eqref{eqn:spec r 2} are the posterior/anterior situations when $k=0$. \\

\begin{itemize}[leftmargin=*]

\item For the integrals corresponding to anterior placement, namely the residues at \eqref{eqn:residue list 2} and \eqref{eqn:residue list 4}, we move the $s$ variable to the very left of the chain of variables, and the only pole we encounter is $w_{b+1} = \frac sq$. The value of the corresponding integral is exactly equal to the first line of \eqref{eqn:residues} for $k$ replaced by $k+1$, $s$ replaced by $\frac sq$, the short string replaced by the long string, and $b$ replaced by $b+1$. \\

\item For the integrals corresponding to posterior placement, namely the residues at \eqref{eqn:residue list 1} and \eqref{eqn:residue list 3}, we move the $s$ variable to the very right, and the only poles we encounter are $w_{b-k-1} = sq^{k+1}$, as well as $s = 0$ and $\infty$. \\

 \begin{itemize}[leftmargin=*]
 	
 \item The value of the corresponding residue at the pole $s = w_{b-k-1} q^{-k-1}$ is exactly equal to the second line of \eqref{eqn:residues} for $k$ replaced by $k+1$ and the short string replaced by the long string. \\
 	
 \item The value of the corresponding residue at $s \in \{0, \infty\}$ is a nice summand. This is proved similarly to the case $k=1$ in the third bullet preceding Claim \ref{claim:residues}. 
 	
 \end{itemize}

\end{itemize} 

\end{proof}

\section{The $\CA_r$-action}
\label{sec:w acts}

\subsection{} The dg scheme $\fZ_2^\bullet$ of Definition \ref{def:fine} has the following very interesting intersection theoretic property, which will later be recognized as the reason why the higher $W$--algebra currents vanish for the moduli space of stable sheaves on $S$. \\

\begin{lemma} 
\label{lem:int th}

Consider the following diagram:
\begin{equation}
\label{eqn:int th}
\xymatrix{ & \fZ_2^\bullet \times_{\fZ_1'} \fZ_2^\bullet \ar[d]^{\e} \\
\fZ_1 \ar@{^{(}->}[r]^-{\delta} & \fZ_1 \underset{\CM \times S}{\times} \fZ_1 }
\end{equation}
with maps given in terms of closed points by:
\begin{equation}
\label{eqn:int th points}
\xymatrix{ & (\CF_1 \subset_x \CF \subset_x \tCF ) \times (\CF_2 \subset_x \CF \subset_x \tCF) \ar[d]^\e \\
\{\CF_1 = \CF_2\} \ar@{^{(}->}[r]^-{\delta} & (\CF_1 \subset_x \CF) \times (\CF_2 \subset_x \CF)}
\end{equation}
Note that $\fZ_1 = \fZ_1'$ are denoted differently in \eqref{eqn:int th} to emphasize the fact that the former has $\CF$ as the bigger sheaf, while the latter has $\CF$ as the smaller sheaf:
$$
\fZ_1 = \{\CF' \subset_x \CF\}  \qquad \text{and} \qquad \fZ_1' = \{\CF \subset_x \tCF\}
$$
Consider the line bundles $\CL, \CL_1, \CL_2, \tCL$ on the spaces in \eqref{eqn:int th} whose fibers over the closed points \eqref{eqn:int th points} are given by the one-dimensional quotients:
$$
\CL_1 = \CF_x/\CF_{1x}, \quad \CL_2 = \CF_x/\CF_{2x}, \quad \tCL = \tCF_x/\CF_x \quad \text{and} \quad \CL = \delta^*(\CL_1) = \delta^*(\CL_2)
$$
Consider also the universal sheaf $\CU$ that parametrizes the sheaves denoted $\CF$ in the diagram \eqref{eqn:int th points}. Then we have the following equality in the $K$--theory of $\fZ_1 \underset{\CM \times S}{\times} \fZ_1$:
\begin{equation}
\label{eqn:sony}
\delta_* [\CL^k] -  \e_*[\tCL^{k-r}] (-1)^r q^{k-r} (\det \CU) = \int_{\infty - 0} \frac {z^k \wedge^\bullet \left( \frac {\CU}z \right)}{\left(1 - \frac z{\CL_1}\right)\left(1 - \frac z{\CL_2}\right)}
\end{equation}
for all $k \in \BZ$. Recall that the notation $\int_{\infty - 0}$ was introduced in \eqref{eqn:residue def}. \\

\end{lemma}

\subsection{} Before proving Lemma \ref{lem:int th}, we need to set up certain concepts and notation from homological algebra. Given a space $X$ (which for us will be a dg scheme supported over a Noetherian scheme) with structure sheaf $\CO$, and a locally free sheaf $\CE$ endowed with a co-section $s : \CE \rightarrow \CO$, one may define the Koszul complex:
$$
\wedge^\bullet(\CE,s) = \left[... \xrightarrow{d_s} \wedge^2 \CE \xrightarrow{d_s} \CE \xrightarrow{d_s} \CO \right]
$$
where we define:
\begin{equation}
\label{eqn:maps 1}
\wedge^k \CE \xrightarrow{d_s} \wedge^{k-1}\CE
\end{equation}
$$
e_1 \wedge ... \wedge e_k \leadsto \sum_{i=1}^k (-1)^{i-1} s(e_i) \cdot e_1 \wedge ... \wedge \widehat{e_i} \wedge ... \wedge e_k 
$$
The Koszul complex is a dg algebra, which can be best seen by packaging it as:
\begin{equation}
\label{eqn:dg 1}
\wedge^\bullet(\CE,s) = \left( \bigoplus_{a = 0}^\infty \wedge^a \CE[-a], d_s \right)
\end{equation}
where the differential $d_s$ is the direct sum of the maps \eqref{eqn:maps 1}, and the homological grading is marked by the number inside the square brackets (so an element $v \in \wedge^a \CE$ has degree $|v| = -a$). The multiplication in the dg algebra \eqref{eqn:dg 1} is given by wedge product, and it is commutative (in the dg sense) since: 
$$
v \wedge w = (-1)^{|v||w|} w \wedge v
$$
More generally, if $\iota : \CE' \rightarrow \CE$ is a map of locally free sheaves on the space $X$ such that $s \circ \iota = 0$, then we consider the commutative dg algebra:
\begin{equation}
\label{eqn:dg 2}
\wedge^\bullet \left(\CE' \xrightarrow{\iota} \CE, s \right) =  \left( \bigoplus_{a,b = 0}^\infty S^a \CE'\otimes \wedge^b \CE [-2a-b], d_\iota + d_s \right)
\end{equation}
where the differential is given by summing up the maps:
\begin{equation}
\label{eqn:maps 2}
S^a \CE' \otimes \wedge^b \CE \xrightarrow{d_\iota} S^{a-1} \CE' \otimes \wedge^{b+1} \CE
\end{equation}
$$
e_1'...e_a' \otimes e_1 \wedge ... \wedge e_b \leadsto \sum_{i=1}^a e_1' ... \widehat{e_i'} ... e_a' \otimes \iota(e_i') \wedge e_1 \wedge ... \wedge e_b
$$
with the maps $d_s : S^a \CE' \otimes \wedge^b \CE \rightarrow S^a \CE' \otimes \wedge^{b-1} \CE$ given by the same formula as \eqref{eqn:maps 1}. \\

\begin{proposition}
\label{prop:dg}

Consider a map of locally free sheaves $\iota : \CE' \rightarrow \CE$ and a co-section $s : \CE \rightarrow \CO$ on a space $X$, and let $s' = s \circ \iota : \CE' \rightarrow \CO$. The natural map of dg algebras:
$$
\CO_{Z'} := \wedge^\bullet(\CE',s') \rightarrow \wedge^\bullet(\CE,s) =: \CO_Z 
$$
induced by $\iota$ gives rise to a map of dg schemes $Z \rightarrow Z'$. Moreover, we have:
\begin{equation}
\label{eqn:quasi-iso}
\CO_Z \stackrel{\emph{q.i.s.}}\cong \wedge^\bullet_{Z'} \left(\CE'|_{Z'} \xrightarrow{\iota} \CE|_{Z'}, s \right)
\end{equation}
as dg modules over $\CO_{Z'}$. \\

\end{proposition}

\begin{proof} Note that in order for \eqref{eqn:dg 2} to be complex, one needs the composition of $s$ and $\iota$ to vanish. In the case at hand, this does not hold over $X$ by assumption, but it does hold over the dg scheme $Z'$ on account of the latter being the derived zero locus of the co-section $s' = s \circ \iota$. Explicitly, for all $a \geq 0$, the map:
$$
S^a \CE' \xrightarrow{\gamma} S^{a-1} \CE' 
$$
$$
e'_1...e'_a \leadsto \sum_{i=1}^a s'(e_i') \cdot e_1' ... \widehat{e_i'} ... e_a'
$$
is null-homotopic over the dg scheme $Z'$. To see this, we observe that ``restricting" to the dg scheme $Z'$ is the same thing as tensoring with the Koszul complex $\wedge^\bullet(\CE',s')$, and therefore the induced map:
$$
S^a \CE' \otimes \wedge^\bullet(\CE', s') \xrightarrow{\gamma \otimes \text{Id}_{\wedge^\bullet(\CE',s')}} S^{a-1} \CE' \otimes \wedge^\bullet(\CE', s')
$$
has the property that $\gamma \otimes \text{Id}_{\wedge^\bullet(\CE',s')} = [h,d_{s'}]$, where the homotopy is:
$$
S^a \CE' \otimes \wedge^c \CE' \xrightarrow{h} S^{a-1} \CE' \otimes \wedge^{c+1} \CE'
$$
$$
e_1'...e_a' \otimes e_1 \wedge ... \wedge e_c \leadsto \sum_{i=1}^a e_1'... \widehat{e_i'} ... e_a' \otimes e_i' \wedge e_1 \wedge ... \wedge e_c
$$
Therefore, the right-hand side of \eqref{eqn:quasi-iso} can be presented as a complex over $X$ as:
\begin{equation}
\label{eqn:resolution}
\left(\bigoplus_{a,b,c=0}^\infty S^a \CE'[-2a] \otimes \wedge^b \CE[-b] \otimes \wedge^c \CE'[-c], d_\iota + d_s + d_{s'} - (-1)^c h \right)
\end{equation}
(the fact that this is a chain complex is a consequence of the equality $(d_\iota+d_s)^2 = \gamma \otimes \text{Id}_{\wedge^\bullet(\CE',s')} \otimes \text{Id}_{\wedge^\bullet(\CE,s)} = [h,d_{s'}] \otimes \text{Id}_{\wedge^\bullet(\CE,s)}$, which is straightforward). However, the complex above can be filtered according to $n = a+c$, with associated graded:
$$
\bigoplus_{n=0}^\infty \left(\underline{\bigoplus_{a=0}^n S^a \CE'[-a] \otimes \wedge^{n-a}\CE', h }\right)[-n] \otimes \wedge^\bullet(\CE,s)
$$
It is well-known that the underlined complex is acyclic unless $n=0$, so therefore \eqref{eqn:resolution} is quasi-isomorphic to $\wedge^\bullet(\CE,s) =  \CO_Z$. Hence \eqref{eqn:quasi-iso} holds in the derived category of coherent sheaves on $X$. However, it also holds in the derived category of coherent sheaves on $Z'$ because all complexes involved (including \eqref{eqn:resolution}) are dg algebras which receive a natural map from $\CO_{Z'} = \wedge^\bullet(\CE',s')$. 
	
\end{proof}

\subsection{} We will now use Proposition \ref{prop:dg} to prove Lemma \ref{lem:int th}. To keep the notation easily legible, throughout the proof below we will use the notation $\times$ (with no further subscript) for derived fiber product over the scheme $\CM\times S$, where $\CM$ parametrizes the sheaves denoted by $\CF$ in the statement of Lemma \ref{lem:int th}. Similarly, we will write $\BP(\CE)$ for the projectivization of the locally free sheaf $\CE$ on $\CM\times S$. \\

\begin{proof} \emph{of Lemma \ref{lem:int th}:} We will compute $\delta_*(\CL^k)$ and $\e_*(\tCL^k)$ as complexes in:
$$
D^-\left(\text{Coh} \left(\fZ_1 \times \fZ_1 \right) \right)
$$
and we will show that the complexes are identical, except for finitely many terms at the right of the complex. These terms will be expressed in terms of $\CU$, $\CL_1$, $\CL_2$ and their powers, and the resulting expression will be shown to match the right-hand side of \eqref{eqn:sony}, thus concluding the proof of the Lemma. We recall that $\CU = \CV/\CW$ where $\CW  \hookrightarrow \CV$ are vector bundles on $\CM \times S$, which leads to the presentation \eqref{eqn:desc 1} of the map of dg schemes $\fZ_1 \rightarrow \CM \times S$, $(\CF' \subset_x \CF) \mapsto (\CF,x)$:
\begin{equation}
\label{eqn:temp}
\fZ_1 \hookrightarrow \BP (\CV)
\end{equation}
where the embedding is the derived zero locus of the co-section:
$$
\CW \otimes \CO(-1) \rightarrow \CV \otimes \CO(-1) \rightarrow \CO
$$
(we write $\CO(1)$ for the tautological line bundle on the projectivization in \eqref{eqn:temp}, and abuse notation by writing $\CV,\CW$ for the aforementioned vector bundles on $\CM\times S$, as well as their pull-backs to the projectivization). With this in mind, let us consider the following diagram of derived fiber products:
\begin{equation}
\label{eqn:fiber}
\xymatrix{\fZ_1 \ar@{^{(}->}[r]^-{\delta} \ar[dr]_p & \fZ_1 \times \fZ_1 \ar[d]^q  \\
& \BP (\CV) \times \fZ_1}
\end{equation} 
By definition, the embedding $q$ is cut out by the composition $\CW \otimes \CO_1(-1) \rightarrow \CV \otimes \CO_1(-1) \rightarrow \CO$ (where we use the notation $\CO_1(1)$ and $\CO_2(1)$ to denote the tautological line bundles on the first and second factors in the fiber products \eqref{eqn:fiber}, respectively). Meanwhile, because $p$ is obtained from the diagonal embedding:
$$
\BP (\CV) \hookrightarrow \BP (\CV) \times  \BP (\CV)
$$
by derived base change, then Beilinson's resolution implies that the embedding $p$ is cut out by the composition $\sigma : \CN_2 \otimes \CO_1(-1) \rightarrow \CV \otimes \CO_1(-1) \rightarrow \CO$, where:
$$
\CN_i = \text{Ker} \Big( \CV \twoheadrightarrow \CO_i(1) \Big), \ \forall i \in \{1,2\}
$$
We may then invoke Proposition \ref{prop:dg} to conclude that the embedding $\delta$ is cut out by the induced map of complexes $[\CW \otimes \CO_1(-1) \rightarrow \CN_2 \otimes \CO_1(-1)] \rightarrow \CO$, i.e.:
$$
\CO_{\delta} \qis \wedge^\bullet (\CW \otimes \CL_1^{-1} \rightarrow \CN_2 \otimes \CL_1^{-1}, \sigma)
$$
on $\fZ_1 \times \fZ_1$ (this is because the line bundle $\CL$ on $\fZ_1$ is the restriction of $\CO(1)$ on $\BP (\CV)$, see Definition \ref{def:desc}). Plugging in \eqref{eqn:dg 2} for the Koszul complex, we obtain:
$$
\delta_*(\CO) \cong  \left( \bigoplus_{a,b = 0}^\infty S^a \CW \otimes \wedge^b \CN_2 \otimes \CL_1^{-a-b}[-2a-b], d_\iota + d_\sigma \right)
$$
where $d_\sigma$ and $d_\iota$ are the differentials \eqref{eqn:maps 1} and \eqref{eqn:maps 2} associated to the maps $\sigma : \CN_2 \otimes \CL_1^{-1} \rightarrow \CO$ and $\iota : \CW \rightarrow \CN_2 = \text{Ker} (\CV \twoheadrightarrow \CO_2(1))$. Since $\CL = \delta^*(\CL_1) = \delta^*(\CL_2)$:
\begin{equation}
\label{eqn:res diag applied}
\delta_*(\CL^k) = \left( \bigoplus_{a,b \geq 0} S^a \CW \otimes \wedge^b \CN_2 \otimes \CL_1^{k - 1 - a - b} \CL_2 [-2a-b], d_\iota + d_\sigma  \right)
\end{equation}
Formula \eqref{eqn:res diag applied} gives the first summand in the left-hand side of \eqref{eqn:sony}, and now we must obtain a similar expression for the second summand. We have the diagram:
\begin{equation}
\label{eqn:curved}
\xymatrix{\fZ_2^\bullet \underset{\fZ_1'}{\times} \fZ_2^\bullet \ar@{^{(}->}[r]^-\eta  \ar@/^2pc/[rr]^\e \ar@{^{(}->}[dr]^\nu & \BP_{\fZ_1 \times \fZ_1} \left( {\CW}^\vee \otimes \omega_S \right) \ar@{->>}[r]^-{\pi} \ar@{^{(}->}[d]^\rho & 
\fZ_1 \times \fZ_1  \\ \ & \BP (\CV)  \times \BP (\CV)  \times \BP\left( {\CW}^\vee \otimes \omega_S \right) & }
\end{equation}
By Definition \ref{def:fine}, the closed embedding $\fZ_2^\bullet \hookrightarrow \BP(\CV) \times \BP(\CW^\vee \otimes \omega_S)$ is cut out by: 
$$
\text{Coker } \left( \CO_1(-1) \otimes \CO_2(-1) \otimes \omega_S \rightarrow \begin{matrix} \CW \otimes \CO_1(-1) \\ \oplus \\ \CV^\vee \otimes \CO_2(-1) \otimes \omega_S \end{matrix} \right) \rightarrow \CO
$$
where the arrows are the tautological maps on $\BP(\CV)$ and $\BP(\CW^\vee \otimes \omega_S)$ and their duals. Recall from \eqref{eqn:koszul 2} that $\fZ_1' \hookrightarrow \BP(\CW^\vee \otimes \omega_S)$ is cut out by the composition: 
$$
\CV^\vee \otimes \CO(-1) \otimes \omega_S \rightarrow \CW^\vee \otimes \CO(-1) \otimes \omega_S  \rightarrow \CO
$$
on $\BP(\CW^\vee \otimes \omega_S)$. Hence, the embedding $\nu$ in diagram \eqref{eqn:curved} is cut out by:
\begin{equation}
\label{eqn:big complex}
\text{Coker } \left( \begin{matrix} \CO_1(-1) \otimes \CO(-1) \otimes \omega_S \\ \oplus \\  \CO_2(-1) \otimes \CO(-1) \otimes \omega_S \end{matrix} \rightarrow \begin{matrix} \CW \otimes \CO_1(-1) \\ \oplus \\ \CW\otimes \CO_2(-1) \\ \oplus \\ \CV^\vee \otimes \CO(-1) \otimes \omega_S \end{matrix} \right) \xrightarrow{\sigma'} \CO
\end{equation}
on $\BP (\CV)  \times \BP (\CV)  \times \BP\left( {\CW}^\vee \otimes \omega_S \right)$, where $\CO_1(1), \CO_2(1), \CO(1)$ denote the tautological line bundles on the three projectivizations, and all arrows in \eqref{eqn:big complex} are the tautological maps and their duals. Similarly, according to Proposition \ref{prop:desc}, the closed embedding $\rho$ in diagram \eqref{eqn:curved} is cut out by the co-section:
$$
\begin{matrix} \CW \otimes \CO_1(-1) \\ \oplus \\ \CW\otimes \CO_2(-1) \end{matrix} \longrightarrow \CO
$$
Combining this fact with \eqref{eqn:big complex} allows us to apply Proposition \ref{prop:dg} to the setting of \eqref{eqn:curved}. We thus conclude that the closed embedding $\eta$ has the property that:
$$
\CO_\eta \qis \wedge^\bullet \left( \begin{matrix} \CO_1(-1) \otimes \CO(-1) \otimes \omega_S \\ \oplus \\  \CO_2(-1) \otimes \CO(-1) \otimes \omega_S \end{matrix} \rightarrow \CV^\vee \otimes \CO(-1) \otimes \omega_S, \sigma' \right)
$$
where the Koszul complex is computed on $\BP_{\fZ_1 \times \fZ_1}(\CW^\vee \otimes \omega_S)$. Since the cokernel of $\CO_2(-1) \rightarrow \CV^\vee$ is the vector bundle $\CN_2^\vee$, we may simplify the complex above to:
\begin{multline*} 
\CO_\eta \qis \wedge^\bullet \left( \CO_1(-1) \otimes \CO(-1) \otimes \omega_S \rightarrow \CN_2^\vee \otimes \CO(-1) \otimes \omega_S, \sigma' \right) = \\
=  \left( \bigoplus_{a,b = 0}^\infty \wedge^b \CN_2^\vee \otimes \CO_1(-a) \otimes \CO(-a-b) \otimes \omega_S^{a+b} [-2a-b], d_{\sigma} + d_{\sigma'} \right)
\end{multline*}
where we recall the co-section $\sigma : \CN_2 \rightarrow \CO_1(1)$. Recall that $\CO_i(1)|_{\fZ_1} = \CL_i$, hence:
\begin{multline} 
\e_*(\tCL^k) = \pi_* \circ \eta_*(\CO(-k)) = \pi_*\left(\CO(-k) \otimes \CO_\eta \right) = \\
= \pi_* \left( \bigoplus_{a,b = 0}^\infty \wedge^b \CN_2^\vee \otimes \CO(-a-b-k) \otimes \CL_1^{-a}  \otimes \omega_S^{a+b} [-2a-b], d_{\sigma} + d_{\sigma'} \right) \label{eqn:big push}
\end{multline}
The push-forward $\pi_*$ does not affect $\CN_2$, $\CL_1$ or $\omega_S$, which are pulled back from the base $\fZ_1 \times \fZ_1$, but it has the following well-known effect on powers of $\CO(-1)$:
\begin{equation}
\label{eqn:cases proj}
\pi_*(\CO(l)) = \begin{cases} S^l \CW^\vee \otimes \omega_S^l & \text{if } l \geq 0 \\
S^{-l-w} \CW \otimes \det \CW \otimes \omega_S^l[w-1] & \text{if } l \leq - w  \\
0 & \text{otherwise} \end{cases}
\end{equation}
where $w = \text{rank }\CW$. For simplicity, let us assume we are dealing with $k > 0$ in \eqref{eqn:big push}, so that we uniformly plug in the second formula in \eqref{eqn:cases proj}:
$$
\e_*(\tCL^k) = \left( \bigoplus_{a,b = 0}^\infty \wedge^b \CN_2^\vee \otimes S^{a+b+k-w} \CW \otimes \CL_1^{-a} \otimes \omega_S^{-k} \otimes \det \CW [-2a-b+w-1] \right)
$$
(we will not write down the differential explicitly anymore, but note that it matches the one in the complex \eqref{eqn:res diag applied}, because applying $\pi_*$ to the tautological map $\sigma' : \CN_2^\vee \otimes \omega_S \rightarrow \CO(1)$ yields precisely the map $\iota$). We may use the natural isomorphism $\wedge^b \CN_2^\vee \cong (\det \CN_2)^\vee \otimes \wedge^{v-1-b} \CN_2$, where $v = \text{rank } \CV = w+r$, to obtain the formula:
$$
\e_*(\tCL^k) = \left( \bigoplus_{a,b = 0}^\infty \wedge^{v-1-b} \CN_2 \otimes S^{a+b+k-w} \CW \otimes \CL_1^{-a} \omega_S^{-k} \frac {\det \CW}{\det \CN_2} [-2a-b+w-1] \right)
$$
Let us relabel indices $a' = a+b+k-w$ and $b' = v-1-b$, and thus:
$$
\e_*(\tCL^k) = \left( \bigoplus_{0 \leq b' \leq v-1}^{a'+b' \geq k+r-1} S^{a'} \CW  \otimes \wedge^{b'}\CN_2 \otimes \CL_1^{k+r-1-a'-b'} \CL_2 \omega_S^{-k} \frac {\det \CW}{\det \CV} [-2a'-b'+2k+r-2] \right)
$$
Once again relabeling $k \leadsto k-r$, $a' \leadsto a$, $b' \leadsto b$ gives us:
$$
\e_*(\tCL^{k-r}) = \left( \bigoplus_{0 \leq b \leq v-1}^{a+b \geq k-1} S^{a} \CW  \otimes \wedge^{b}\CN_2 \otimes \frac {\CL_1^{k-1-a-b} \CL_2 \omega_S^{-k+r}}{\det \CU} [-2a-b+2k-r-2] \right) \Rightarrow
$$
\begin{multline*}
\Rightarrow \e_*(\tCL^{k-r}) \otimes (\det \CU)\otimes \omega_S^{k-r}[-2k+r+2] = \\ = \left( \bigoplus_{0 \leq b \leq v-1}^{a+b \geq k-1} S^{a} \CW  \otimes \wedge^{b}\CN_2 \otimes \CL_1^{k-1-a-b} \CL_2  [-2a-b] \right)
\end{multline*}
Comparing the right-hand side of the relation above with \eqref{eqn:res diag applied} implies that there is a map in $D^-(\text{Coh}(\fZ_1 \times \fZ_1))$:
$$
\delta_*(\CL^k) \longrightarrow \e_*(\tCL^{k-r}) \otimes (\det \CU)\otimes \omega_S^{k-r}[-2k+r+2]
$$
whose cone is the finite complex: 
$$
\left( \bigoplus_{0 \leq b \leq v-1}^{0 \leq a+b \leq k-2} S^{a} \CW  \otimes \wedge^{b}\CN_2 \otimes \CL_1^{k-1-a-b} \CL_2  [-2a-b], d_\iota + d_\sigma \right)
$$
The class of this complex in $K$--theory is precisely the right-hand side of \eqref{eqn:sony}, as a consequence of $[\CN_2] = [\CV] - [\CL_2]$, $[\CU] = [\CV] - [\CW]$ and the additivity of $\wedge^\bullet$. 

\end{proof}

\subsection{}

Armed with Lemma \ref{lem:int th}, it is now time to prove Theorem \ref{thm:w acts}. We assume Conjecture \ref{conj:shuf acts}, namely the existence of an ``action" (in the sense of Definition \ref{def:shuf acts}):
\begin{equation}
\label{eqn:conj acts}
\CA \curvearrowright \km \qquad \text{i.e. an assignment} \qquad \CA \rightarrow \text{Hom}(\km, \kms)
\end{equation}
which sends:
$$
E_{d_\bullet} \in \CA^\leftarrow \subset \CA \qquad \text{to the operator} \qquad e_{d_\bullet} \text{ of \eqref{eqn:def e fine}}
$$
$$
F_{d_\bullet} \in \CA^\rightarrow \subset \CA \qquad \text{to the operator} \qquad f_{d_\bullet} \text{ of \eqref{eqn:def f fine}}
$$
for all $d_\bullet = (d_1,...,d_k)$. Moreover, we recall that $\km$ is a good module of $\CA$, since the grading on $\km$ by second Chern class is bounded below in virtue of the Bogomolov inequality. This implies that $f_{(d_1,...,d_k)}$ annihilates any given element of $\km$ if $k$ is large enough, hence the conjectural action \eqref{eqn:conj acts} induces an action of:
$$
\wCA^\uparrow \curvearrowright \km \qquad \text{and hence an action} \qquad \CA_\infty \curvearrowright \km
$$
(by the definition of $\CA_\infty$ in \eqref{eqn:def a infty}, there exists a homomorphism $\CA_\infty \rightarrow \wCA^\uparrow$). The extended algebra $\CA_\infty^\ext$ of \eqref{eqn:def a ext infty} also acts on $\km$ by the same reasons. \\

\subsection{} It would be rather difficult to prove that relations \eqref{eqn:w vanish} and \eqref{eqn:w exp} hold in $\km$ based on the definition \eqref{eqn:def w} of the generating currents of the deformed $W$--algebra, so we appeal to a different formula that was worked out in (3.20) of \cite{W}: 
\begin{equation}
\label{eqn:formula}
W_{d,k} = \sum^{d_2 - d_1 = d}_{k_0 + k_1 + k_2 = k} T_{d_1,k_1}^\leftarrow E_{0,k_0} T_{d_2, k_2}^\rightarrow \cdot q^{(k-1)d_2}
\end{equation}
where $k_0,k_1,k_2, d_1, d_2$ go over the non-negative integers. The elements $E_{0,k_0}$ in \eqref{eqn:formula} are simply the standard generators of $\CA^{\text{diag}} \subset \CA$ (see Proposition \ref{prop:hecke}), while:
$$
T_{n,k}^\leftarrow \in \CA^\leftarrow \cong \CS \quad \text{and} \quad T_{n,k}^\rightarrow \in \CA^\rightarrow \cong \CS^\op
$$
correspond to the following elements of the shuffle algebra, for all $n,k \in \BN$:
\begin{equation}
\label{eqn:def t}
T_{n,k} (z_1,...,z_n) = \sym \left[ \frac {(-1)^{k-1} z_n^k}{\prod_{i=1}^{n-1} \left(1 - \frac {z_{i+1}q}{z_i} \right)} \prod_{1\leq i < j \leq n} \zeta \left( \frac {z_i}{z_j} \right) \right]
\end{equation}
(we also let $T_{n,0} = T_{0,n} = \delta_n^0$). Formula \eqref{eqn:formula} was derived in \cite[Proposition 2.16]{W} by showing that $T_{n,k}^\leftarrow$ and $T_{n,k}^\rightarrow$ are equal to the sums of $E_{n_1,k_1}...E_{n_t,k_t}$ as in \eqref{eqn:def w}, but the sums only run over $n_i < 0$ and $n_i > 0$, respectively. By Conjecture \ref{conj:shuf acts} (and more specifically, the prescription of \eqref{eqn:generators act 1}--\eqref{eqn:generators act 2}), the shuffle elements:
\begin{align}
&T_{n,k}^\leftarrow \quad \text{act on }\km \text{ as} \quad (-1)^{k-1}(p_+ \times p_S)_*\left( \CL_n^k \cdot p_-^* \right) \label{eqn:t plus} \\ 
&T_{n,k}^\rightarrow \quad \text{act on }\km \text{ as} \quad (-1)^{k-1} \frac {(\det \CU)^{\otimes n}}{(-1)^{nr} q^{n(r-1)}} \cdot (p_- \times p_S)_* \left( \frac {\CL_n^k}{\CQ^r} \cdot p_+^* \right) \label{eqn:t minus}
\end{align}
where $\CQ = \CL_1...\CL_n$ and the maps $p_+,p_-,p_S$ are as in diagram \eqref{eqn:diag fine}. \\ 


\subsection{} Finally, recall that $E_{0,k}$ acts on $\km$ as multiplication by the coefficient of $z^{-k}$ in $\wedge^\bullet \left( \frac {\CU}z \right)$. Therefore, \eqref{eqn:formula} implies that $W_{d,k} : \km \rightarrow \kms$ is given by:
\begin{equation}
\label{eqn:explicit w 0}
\sum^{d_2 - d_1 = d}_{k_0 + k_1 + k_2 = k}  q^{(k-1)d_2} (\pi_1 \times \pi_S)_* \left[ (-1)^{k_1-1} \CL_1^{k_1} \cdot \wedge^{k_0} \tCU \cdot \frac {(-1)^{k_2-1} ( \det \tCU)^{d_2} \CL_2^{k_2}}{(-1)^{rd_2} q^{d_2(r-1)} \CQ^{r} } \pi_2^* \right]
\end{equation}
where $k_1,k_2 \in \BN$, $k_0,d_1,d_2 \in \BN \sqcup \{0\}$, and the maps are given by:
\begin{equation}
\label{eqn:big diag}
\xymatrix{
& \fV_{d_1,d_2} := \{\CF_1 \subset_x ... \subset_x \CF \subset_x \tCF \supset_x \CF' \supset_x ... \supset_x \CF_2\} \ar[ld]_{\pi_1} \ar[d]^{\pi_S} \ar[rd]^{\pi_{2}} & \\
\{\CF_1\} & \{x\} & \{\CF_2\}} 
\end{equation}
Here, $\CL_1, \CL_2$ parametrize the length 1 sheaves $\tCF_x/\CF_x$ and $\tCF_x/\CF_x'$, respectively, and $\CQ$ is the determinant of the rank $d_2$ vector bundle $\tCF_x/\CF_{2,x}$. The number of inclusion signs $\subset_x$, $\supset_x$ in \eqref{eqn:big diag} is $d_1$, $d_2$, respectively, and the dg scheme structure is:
$$
\fV_{d_1,d_2} = \underbrace{\fZ_2^\bullet \times_{\fZ_1} ... \times_{\fZ_1} \times \fZ_2^\bullet}_{d_1 - 1} \underset{\tCM \times S}{\times} \underbrace{\fZ_2^\bullet \times_{\fZ_1} ... \times_{\fZ_1} \times \fZ_2^\bullet}_{d_2 - 1}
$$
where $\tCM$ is the moduli space that parametrizes the sheaves denoted by $\tCF$ in \eqref{eqn:big diag}. Note that we may rewrite \eqref{eqn:explicit w 0} as a residue, namely:
\begin{equation}
\label{eqn:explicit w 1}
W_{d,k} = (-1)^k \sum_{d_2 - d_1 = d} q^{(k-r)d_2} (\pi_1 \times \pi_S)_* \left[ \underset{z=\infty}{\text{Res}}  \frac {z^k \wedge^{\bullet}\left(\frac {\tCU}z\right)\frac {( \det \tCU)^{d_2}}{(-1)^{rd_2} \CQ^r}dz}{\left(1 - \frac z{\CL_1}\right)\left(1 - \frac z{\CL_2}\right)z} \cdot \pi_2^* \right] 
\end{equation}
Theorem \ref{thm:w acts} reduces to the following, which we prove regardless of Conjecture \ref{conj:shuf acts}. \\

\begin{theorem}
\label{thm:vanish}

The RHS of \eqref{eqn:explicit w 1} vanishes if $k>r$, while for $k=r$ it equals:
\begin{equation}
\label{eqn:explicit w 2}
\sum_{d_1 - d_2 = d} (\pi_1 \times \pi_S)_* \left[\frac { ( \det \tCU)^{d_2+1}}{(-1)^{rd_2} \CQ^r} \cdot \pi_2^* \right] 
\end{equation}

\end{theorem}

\text{ }

\noindent We observe that the operator \eqref{eqn:explicit w 2} precisely matches the action on $\km$ of:
$$
u \sum_{d_2 - d_1 = d}^{d_1,d_2 \geq 0} h_{-d_1} h_{d_2} \in \CA
$$
where $u = \det \tCU$, and $h_{\pm k}$ are to the elements $p_{\pm k}$ of \eqref{eqn:gardel} as complete symmetric polynomials are to power-sum functions. Indeed, as shown in \cite[equations (2.9) and (2.28)]{W}, the elements: 
$$
h_{-k} = H_{-k,0} \in \CA^\leftarrow \qquad \text{and} \qquad h_k = q^{k(r-1)} H_{k,0} \in \CA^\rightarrow
$$
correspond to setting $d_1=...=d_k =0$ in \eqref{eqn:def e fine} and \eqref{eqn:def f fine}, respectively (the fact that this statement also holds in the module $\km$ requires assuming Conjecture \ref{conj:shuf acts}). \\

\begin{proof} \emph{of Theorem \ref{thm:vanish}:} We claim that both cases $k>r$ and $k=r$ boil down to:
\begin{equation}
\label{eqn:explicit w 3}
0 = \sum_{d_2 - d_1 = d} q^{(k-r)d_2} (\pi_1 \times \pi_S)_* \left[ \frac { ( \det \tCU)^{d_2}}{(-1)^{rd_2}\CQ^r} \i  \frac {z^k \wedge^{\bullet}\left(\frac {\tCU}z\right)}{\left(1 - \frac z{\CL_1}\right)\left(1 - \frac z{\CL_2}\right)} \cdot \pi_2^* \right]
\end{equation}
for all $k \geq r$. This is a consequence of the fact that the integrand in the right-hand side of \eqref{eqn:explicit w 3} is $z^{k-r}$ times a rational function which is regular at $0$. The value of this rational function at $z=0$ is equal to the right-hand side of \eqref{eqn:explicit w 1}. \\

\noindent Therefore, the task has become to prove \eqref{eqn:explicit w 3}. Let us consider the space $\fW_{d_1,d_2}$ which parametrizes flags:
\begin{equation}
\label{eqn:big diag 2}
\xymatrix{& {\tCF} & \\
& \CF_1' \subset_x ... \subset_x \CF_1 \subset_x \CF \supset_x \CF_2 \supset_x ... \supset_x \CF_2' \ar@{^{(}->}[u]_x &}
\end{equation}
where the total number of inclusions in the horizontal row is $d_1$ and $d_2$, respectively. 
As a dg scheme, $\fW_{d_1,d_2}$ is defined as the derived fiber product:

$$
\xymatrix{\fW_{d_1,d_2} \ar[d] \ar[r]^-{a^{(d_1,d_2)}} & \fV_{d_1,d_2} \ar[d] \\
\fZ_2 \underset{\fZ_1'}{\times} \fZ_2 \ar[r]^-\e & \fZ_1 \underset{\CM \times S}{\times} \fZ_1 }
$$
where the vertical maps remember $\{\CF_1,\CF_2 \subset_x \CF \subset_x \tCF \}$ and $\{\CF_1, \CF_2 \subset_x \CF\}$, respectively, and the horizontal maps forget $\tCF$. We also have a derived fiber square:
$$
\xymatrix{\fW_{d_1-1,d_2-1} \ar[d] \ar@{^{(}->}[r]^-{b^{(d_1,d_2)}} & \fV_{d_1,d_2} \ar[d] \\
\fZ_1 \ar@{^{(}->}[r]^-{\delta} & \fZ_1 \underset{\CM \times S}{\times} \fZ_1 }
$$
where the vertical projection maps remember $\{\CF \subset_x \tCF\}$ and $\{\CF \subset_x \tCF \supset_x \CF'\}$, respectively. Because both squares are Cartesian, we may invoke Lemma \ref{lem:int th}:
$$
b^{(d_1,d_2)}_*[\CL^k] - a^{(d_1,d_2)}_* [\tCL^{k-r}] (-1)^r q^{k-r} (\det \CU) = \i \frac {z^k \wedge^\bullet \left( \frac {\CU}z \right)}{\left(1 - \frac z{\CL_1} \right)\left(1 - \frac z{\CL_2} \right)}
$$
where $\tCL$ parametrizes the one-dimensional vector space $\tCF_x/\CF_x$. As we sum the formula above over all non-negative integers $d_1,d_2$ such that $d_2 - d_1 = d$, we obtain:
\begin{multline}
\sum_{d_2 - d_1 = d} (\pi_1 \times \pi_S)_*  \left[ \frac { q^{(k-r)d_2} ( \det \tCU)^{d_2}}{(-1)^{rd_2}\CQ^r} b^{(d_1,d_2)}_*[\CL^k]   \right] \cdot \pi_2^* -  \\
- \sum_{d_2 - d_1 = d} (\pi_1 \times \pi_S)_*  \left[ \frac { q^{(k-r)(d_2+1)} ( \det \tCU)^{d_2+1}}{(-1)^{r(d_2+1)}\CQ^r} a^{(d_1,d_2)}_* [\tCL^k] \right] \cdot \pi_2^* = \\
= \sum_{d_2 - d_1 = d}  (\pi_1 \times \pi_S)_* \left[q^{(k-r)d_2} \frac { ( \det \tCU)^{d_2}}{(-1)^{rd_2} \CQ^r} \i  \frac {z^k \wedge^{\bullet}\left(\frac {\tCU}z\right)}{\left(1 - \frac z{\CL_1}\right)\left(1 - \frac z{\CL_2}\right)} \cdot \pi_2^* \right] \label{eqn:final}
\end{multline}
The equality above uses the identity $\CQ_{d_2-1} \tCL = \CQ_{d_2}$, where we write $\CQ_{d_2}$ for the line bundle parametrizing $\det \CF_x/\CF_{2,x}'$ in the diagram \eqref{eqn:big diag 2} with $d_2$ symbols $\supset_x$. \\

\noindent Moreover, in \eqref{eqn:final} we used the equality of line bundles $\det \CU = \det \tCU$, which holds because the determinants of all possible universal bundles on the correspondences $\fZ_1$ and $\fZ_2^\bullet$ are canonically isomorphic. The LHS of \eqref{eqn:final} vanishes because the first term for a pair $(d_1,d_2)$ precisely cancels out the second term for $(d_1-1,d_2-1)$. Therefore, the RHS of \eqref{eqn:final} also vanishes, thus establishing \eqref{eqn:explicit w 3}.
\end{proof}

\end{document}